\documentclass{article}

\usepackage[margin=1in]{geometry}
\usepackage[utf8]{inputenc}
\usepackage{csquotes}
\usepackage{amsfonts}
\usepackage{amsmath,amssymb,accents,empheq}
\usepackage{amsthm} 
\usepackage{enumitem}
\usepackage{subcaption}
\usepackage{graphicx}
\usepackage{float}
\usepackage{xcolor}
\usepackage{hyperref}
\usepackage{booktabs}

\usepackage{cleveref}
\usepackage{ulem}

\usepackage{algorithm}
\usepackage{algpseudocode}
\numberwithin{equation}{section}

\hypersetup{
    colorlinks=true,
    linkcolor=green,
    citecolor=green,
    filecolor=green,      
    urlcolor=green,
}

\newtheorem{theorem}{Theorem}[section]
\newtheorem{lemma}{Lemma}[section]

\newcommand{\email}[1]{\href{mailto:#1}{#1}}

\usepackage{listings}
 
\definecolor{codegreen}{rgb}{0,0.6,0}
\definecolor{codegray}{rgb}{0.5,0.5,0.5}
\definecolor{codepurple}{rgb}{0.58,0,0.82}
\definecolor{backcolour}{rgb}{0.95,0.95,0.92}
\title{A conservative adaptive rank method for the Wigner-Poisson system}
\author{Andrew J. Christlieb 
\thanks{Department of Computational Mathematics, Science and Engineering, Michigan State University, East Lansing, MI, 48824}
\and 
Sining Gong \footnotemark[1]
\and 
F. Alejandro Padilla-Gomez \thanks{Corresponding author. Department of Computational Mathematics, Science and Engineering, Michigan State University, East Lansing, MI, 48824 (\email{padill77@msu.edu})}
\and 
Jing-Mei Qiu \thanks{Department of Mathematical Sciences, University of Delaware, Newark, DE, 19716}}

\begin{document}

\maketitle

\begin{abstract}
    We propose a conservative adaptive rank method for the 1D1V Wigner-Poisson system. The method addresses a central challenge in deterministic quantum kinetic simulations: reducing the cost of phase-space evolution without destroying the macroscopic invariants that determine physical fidelity. The scheme combines a sampling-based adaptive rank  Wigner-Poisson update \cite{christlieb2025sampling} with a conservative macroscopic correction. A conservative density-momentum solve supplies local macroscopic updates, a Fermi-Dirac-type reconstruction transfers these updates to the kinetic solution, and a global quadratic moment correction enforces the discrete total energy constraint at the kinetic level. Unlike the Maxwell-Boltzmann-type correction often used in classical kinetic settings, the proposed reconstruction is based on a Fermi-Dirac-type form motivated by the quantum-statistical structure of the model. The corrected state is incorporated into an ACA-SVD representation, allowing the numerical rank to adapt to the phase-space complexity generated by the nonlocal Wigner operator and the self-consistent Poisson field.

    Numerical experiments for the two-stream instability, strong Landau damping, and bump-on-tail instability show that the method captures benchmark Wigner-Poisson dynamics for several values of the quantum parameter H, maintains bounded adaptive ranks, and preserves the specified global discrete invariants with conservation errors near machine precision in the reported diagnostics. We also compare the present formulation, which uses a local density-momentum macroscopic correction together with a global total energy correction, with a related globally conservative formulation for mass, momentum, and energy \cite{christlieb2026structurepreservingadaptiverankapproachhighdimensional}. The two approaches produce nearly identical phase-space and diagnostic results for the periodic benchmark test considered here. This comparison indicates that both local density-momentum correction and fully global moment correction can be made compatible with adaptive rank compression for Wigner-Poisson dynamics in the tested 1D1V periodic setting.
    
\end{abstract}

\section{Introduction}

As a phase-space reformulation of quantum mechanics, the Wigner equation captures tunneling, interference and nonlocal quantum effects in a transport form that is close enough to classical kinetics to invite deterministic PDE discretization \cite{wigner1932quantum, markowich2012semiconductor, weinbub2018recent}. For the self-consistent 1D1V Wigner-Poisson system, however, that conceptual advantage becomes a numerical liability. The solution is oscillatory and sign-indefinite in momentum, the Wigner operator is nonlocal and pseudodifferential, and the Poisson coupling feeds back a global field at every step \cite{markowich2012semiconductor, weinbub2018recent, ringhofer1990spectral, ringhofer1992spectral}. The central computational tension is therefore sharper than a standard accuracy versus cost tradeoff, one wants compression of phase-space complexity, but not at the price of corrupting the macroscopic structure on which physical credibility depends \cite{einkemmer2021mass, einkemmer2023robust, coughlin2024robust, GuoVlasovFlowMap2022, guo2024conservative, guo2024local}.

This issue is important because Wigner-Poisson models remain a standard description for nonequilibrium transport in resonant tunneling and related semiconductor devices, and they are increasingly used in dense-plasma and warm-dense regimes where quantum diffraction alters collective electrostatic dynamics. In such settings, the questions of interest are often transient and nonlinear, including switching, bistability, intrinsic oscillations, Landau-type damping and wave-particle interaction. Consequently, long-time robustness is at least as important as local truncation accuracy \cite{hu2022kinetic, jiang2023numerical}.

Deterministic Wigner solvers have advanced substantially over the last three decades. Ringhofer's spectral and spectral-collocation methods, followed by Arnold-Ringhofer operator splitting, establish the analytical and numerical foundation for grid-based Wigner and Wigner-Poisson solvers \cite{ringhofer1990spectral, ringhofer1992spectral, arnold1995operator}. Subsequent work introduced conservative adaptive spectral elements, WENO schemes with adaptive momentum resolution, semi-spectral formulations, mixed characteristic-spectral methods, and higher-order splitting or SBP-SAT/pseudospectral discretization for transient and higher-dimensional problems \cite{shao2011adaptive, dorda2015weno, furtmaier2016semi, xiong2016advective, chen2022higher, sun2024hybrid}. These developments deliver impressive accuracy, but they remain fundamentally full-rank in phase-space; once resolution, dimension, or integration time grows, the memory footprint and data movement become the dominant obstacles \cite{jiang2023numerical, xiong2016advective, chen2022higher, sun2024hybrid}.

Dynamical low-rank approximation offers an attractive way to reduce this cost. After the tangent-space framework by Koch and Lubich and the projector-splitting integrator of Lubich and Oseledets \cite{koch2007dynamical, lubich2014projector}, low-rank kinetic solvers matured rapidly, including Vlasov-specific projector-splitting algorithms and a sequence of conservative variants that target mass, momentum, and energy defects induced by compression \cite{einkemmer2018low, einkemmer2021mass, einkemmer2023robust,coughlin2024robust}. In parallel, tensor-based low-rank constructions by Guo and Qiu demonstrated that conservation can be embedded directly in the low-rank representation, culminating in conservative and locally macroscopic conservative formulations \cite{GuoVlasovFlowMap2022, guo2024conservative, guo2024local}.

For Wigner-Poisson, the compatibility of sampling-based adaptive compression with conservative correction remains much less developed. A quantum low-rank solver must confront a nonlocal oscillatory operator, self-consistent field coupling, and boundary-sensitive transient dynamics, while also respecting the fact that Wigner negativity and low-order moment information are both numerically consequential \cite{wigner1932quantum, weinbub2018recent, frensley1990boundary, jiang2023numerical, chen2022higher, sun2024hybrid}. More broadly, moment-closure, hyperbolic-reduction, Wigner-Poisson-BGK, and quantum-statistical equilibrium studies all support the same conclusion: reduced descriptions remain reliable only when they are tied to physically meaningful macroscopic observables and, when appropriate, to Fermi-Dirac-type equilibria \cite{cai2012quantum, li2014numerical, crouseilles2014asymptotic, bonilla2005wigner, barletti2012derivation}. What is presently missing is a method that treats low-rank compression and conservation for Wigner-Poisson as one coupled design problem rather than two separate corrections.

In this paper, we propose a conservative adaptive rank method for the Wigner-Poisson system, with an application to the 1D1V model. The method combines a conservative-density-momentum macroscopic solve with a global quadratic moment correction. The macroscopic solve provides local density and momentum updates before filtering and recompression, while the final conservation statement for the corrected low-rank state is global.
The method combines a locally conservative density-momentum macroscopic solve with a global quadratic moment correction for total energy. A related formulation enforcing global conservation of mass, momentum, and energy is presented in \cite{christlieb2026structurepreservingadaptiverankapproachhighdimensional}. We compare the two formulations in section \ref{sec:numerical results} and find nearly identical phase-space and diagnostic results for the periodic benchmarks considered here.

The rest of the paper is organized as follows. Section \ref{sec:model} introduces the nondimensional Wigner-Poisson system, its conservation laws, and the Fermi-Dirac-type equilibrium reconstruction. Section \ref{sec:full-rank solve} develops the full-rank conservative correction, including the macroscopic moment solve and the filtered equilibrium-based correction. Section \ref{sec:low-rank solve} incorporates this correction into the adaptive rank framework. Section \ref{sec:numerical results} presents numerical tests comparing conservative formulations and demonstrating accuracy, conservation, rank adaptivity, and efficiency. Section \ref{sec:conclusion} concludes and discusses future extensions.

\section{Model setting}\label{sec:model}

This section fixes the notation used throughout the paper. We first introduce the nondimensional 1D1V Wigner-Poisson system and the quantum parameter $H$. We then introduce the moment identities that motivate the conservative macroscopic correction. Finally, we introduce the Fermi-Dirac-type reconstruction used to transfer corrected density and momentum back to the solution.
\subsection{Wigner–Poisson system}\label{subsec:Wigner poisson system}

The Wigner–Poisson system in 1D1V, as presented in~\cite[Chapter~4]{haas2011quantum}, takes the form:
\begin{subequations}\label{eqn:WPmodel}
\begin{align}
\frac{\partial \tilde{f}}{\partial \tilde{t}} + \tilde{v} \frac{\partial \tilde{f}}{\partial s} 
&= -\frac{iem_e}{2\pi \hbar^2} \iint d\tilde{v}'\, ds'\, \exp\left( i m_e \frac{(\tilde{v}' - \tilde{v}) s'}{\hbar} \right) \left[ \phi\left(s + \frac{s'}{2}\right) - \phi\left(s - \frac{s'}{2}\right) \right] \tilde{f}(s, \tilde{v}', \tilde{t}), \\
\frac{\partial^2 \phi}{\partial s^2} &= -\frac{e}{\epsilon_0} \left( \int d\tilde{v}\, \tilde{f} - n_0 \right),
\end{align}
\end{subequations}
where \(\tilde{f}(s, \tilde{v}, \tilde{t})\) is the Wigner distribution function in phase space, and \(\phi(s)\) denotes the electrostatic potential. The physical constants are: \(m_e\) the particle mass, \(e\) the elementary charge, \(\epsilon_0\) the vacuum permittivity, \(n_0\) the background number density, \(\hbar\) the reduced Planck constant, and \(\tilde{t}\) the physical time. 

We now nondimensionalize \eqref{eqn:WPmodel}; see details in \cite{christlieb2025sampling}. Let \(\tau\), \(l\), and \(\bar{\phi}\) denote the characteristic time, length, and potential scales. We introduce the nondimensional variables
\[
\tilde{t}=\tau t,\qquad s=lx,\qquad \phi=\bar{\phi}\Phi,\qquad
\tilde{v}=\frac{l}{\tau}v,\qquad
f=\frac{l}{n_0\tau}\tilde{f}.
\]
We choose the potential scale, time scale, and length scale according to the characteristic plasma quantities
\[
\bar{\phi}=\frac{e n_0 l^2}{\epsilon_0},\qquad
\tau=\omega_{pe}^{-1} = \left(\sqrt{\frac{e^2 n_0}{m_e\epsilon_0}}\right)^{-1},\qquad
l=\lambda_D =\sqrt{\frac{\epsilon_0 k_B \mathcal{T}_0}{n_0 e^2}}.
\]
Here \(\mathcal{T}_0\) is a reference temperature, with the physical temperature written as \(\mathcal{T}=\mathcal{T}_0\hat{\mathcal{T}}\). With these choices, the velocity scale becomes the thermal velocity,
\[
\frac{l}{\tau}=\lambda_D\omega_{pe}
=\sqrt{\frac{k_B \mathcal{T}_0}{m_e}}=v_{th}.
\]
The dimensionless parameters are therefore
\begin{equation}\label{eqn:quantum parameter}
C=\frac{e\bar{\phi}}{m_e l^2\omega_{pe}^2}=1,\qquad
H=\frac{\hbar}{m_e\lambda_D^2\omega_{pe}},\qquad
D=\frac{e n_0 l^2}{\bar{\phi}\epsilon_0}=1.
\end{equation}
Thus \(H\) is the only remaining dimensionless parameter and measures the strength of quantum effects.

The nondimensional Wigner-Poisson system is then
\begin{subequations}\label{eqn:nonD-WP-system}
\begin{align}
\frac{\partial f}{\partial t}+v\frac{\partial f}{\partial x}
&= -\frac{i}{2\pi H^2}
\iint dv'\,dx'\,
\exp\left(i\frac{v'-v}{H}x'\right)
\left[
\Phi\left(x+\frac{x'}{2}\right)-\Phi\left(x-\frac{x'}{2}\right)
\right]f(x,v',t),\label{eq:wigner}\\
-\frac{\partial^2\Phi}{\partial x^2}
&= \int f\,dv-1.\label{eq:poisson}
\end{align}
\end{subequations}
The parameter \(H\), defined in~\eqref{eqn:quantum parameter}, is prescribed in the numerical experiments. As \(H\to 0\), system~\eqref{eqn:nonD-WP-system} formally approaches the classical Vlasov--Poisson limit, while larger values of \(H\) correspond to stronger quantum effects. In the numerical tests below, we use several values of \(H\) to assess the proposed conservative low-rank method across different quantum regimes.

\subsection{Moment identities for the Wigner-Poisson system} 
We next derive the moment identities used by the conservative correction. These identities are not introduced as a closed quantum-fluid model. Instead, the density and momentum equations provide the macroscopic variables used in the correction step, while the second-moment identity identifies the kinetic part of the total energy. In the numerical method, the implicit macroscopic solver uses the density-momentum system, and the final total energy constraint is imposed afterward at the kinetic level through a global quadratic moment correction. This distinction is important: the macroscopic solver gives local density and momentum updates before filtering and recompression, whereas the final conservation statement for the corrected low-rank state is global.

Let the Wigner potential term be $W(\Phi, f)$, which is
\[
W(\Phi, f) := -\frac{i }{2 \pi H^2} \int \int dv' dx' \exp( i \frac{v' - v}{H} x' ) [ \Phi(x + \frac{x'}{2}) - \Phi(x - \frac{x'}{2})] f(x,v', t).
\]
Integrating this term with respect to velocity gives:
\begin{subequations}
\begin{align}
\label{eqn:WignerIdentity1}
    \int W dv & = 0;\\
    \label{eqn:WignerIdentity2}
     \int v ~ W dv & = \frac{i }{2 \pi H^2} \int \int \int dv' dx' dv ~ iH ~\frac{\partial}{\partial x'} \left(\exp( i \frac{v' - v}{H} x' ) \right)[ \Phi(x + \frac{x'}{2}) - \Phi(x - \frac{x'}{2})] f(x,v', t) \nonumber\\
     & = -\rho \Phi_x;\\
    \label{eqn:WignerIdentity3}
     \int v^2 ~ W dv & = \frac{i }{2 \pi H^2} \int \int \int dv' dx' dv ~ -H^2 \frac{\partial}{\partial x'} \left(\exp( i \frac{v' - v}{H} x' ) \right)[ \Phi(x + \frac{x'}{2}) - \Phi(x - \frac{x'}{2})] f(x,v', t)  \nonumber\\
     & = -2\rho u \Phi_x;
\end{align}
\end{subequations}
where we use the following identity in the second and the third equations:

\[
v ~ e^{-i v x} = i \frac{\partial}{\partial x} e^{-i v x}.
\]
We can now define the macroscopic moments as:
\begin{subequations}
    \begin{align}
\rho(x, t) &= \int f(x, v, t) \, dv &&\text{(density)} \\
u(x, t) &= \frac{1}{\rho} \int v \, f \, dv &&\text{(bulk velocity)} \\
\mathcal{T}(x, t) &= \frac{1}{\rho} \int v^2 \, f \, dv \; - \;u^2 &&\text{(temperature)} \label{eq:temp} \\
P(x, t) &= \int (v - u)^2 f \, dv &&\text{(pressure)}. \\
Q(x, t) &= \int (v - u)^3 f \, dv &&\text{(heat flux)}.
\end{align}
\end{subequations}
\begin{theorem}\label{thm:fluid model}
    The following moment identities hold for the Wigner-Poisson system:
    \begin{subequations}
        \begin{align}
            \frac{\partial \rho}{\partial t} + \frac{\partial (\rho \, u)}{\partial x} & = 0 \\
            \frac{\partial (\rho u) }{\partial t} + \frac{\partial}{\partial x} \left(P(x,t) + \rho \, u^2\right) & = -\rho \Phi_x. \\
            \frac{\partial}{\partial t} \left(P(x,t) + \rho \, u^2 \right) + \frac{\partial}{\partial x} \left( \rho \, u^3 + Q(x,t) + 3u\, P(x,t)\right) & = -2\rho \, u \Phi_x.
        \end{align}
    \end{subequations}
Furthermore, the conservation laws of mass, momentum, and energy 
can be derived through these  equations. 
\end{theorem}
Proof of this theorem is provided in the Appendix \ref{apdix:thm 2.1}.

\subsection{Fermi-Dirac-type reconstruction}\label{subsubsec:Fermi}
A key element of the conservative correction is the use of an equilibrium-inspired reconstruction of the system in a way that lets us make small corrections to the kinetic solution. For this reason, we use the so-called \textit{Fermi-Dirac distribution} as opposed to the Maxwell-Boltzmann distribution, as our model is quantum in nature. We now nondimensionalize the \textit{Fermi-Dirac distribution} to make it match the original model and analyze some properties that matter when we design the conservation scheme.

The \textit{Fermi-Dirac distribution} describes the occupation of quantum states by fermions (particles obeying the Pauli exclusion principle, such as electrons). 
 
            There are a variety of forms of the Fermi-Dirac distribution, but when applying it, we need to make sure they correspond to the form of  model being solved, especially for the derivation of unitless and dimensionless forms. The distribution function in \textit{global equilibrium} in this case is given by:
\[
f_a(r, p) = \frac{1}{h_a} \cdot \frac{1}{\exp\left[\beta \left(\lambda_a - \frac{1}{2m_a}  (p - m_a u)^2 \right)\right] - d_a}
\]
where: $f_a(r, p)$ is the phase-space distribution function for species $a$, $h_a$ is the effective phase-space volume per quantum state:  
  \[
  h_a = \frac{(2\pi \hbar)}{g_a}
  \]  
where $g_a$ is the spin multiplicity factor (e.g., $g_e = 2$ for electrons). In addition, $\lambda_a$ is the chemical potential of species $a$, and $u$ determines the average energy at which states are filled. At zero temperature, it corresponds to the Fermi energy $E_F$. We use $m_a$ to represent mass of the particles of species $a$ and for electrons, $m_a = m_e$ is the \textbf{electron mass}. And $u$ is macroscopic flow velocity of the system. If the system is at rest in equilibrium, $u = 0$. $d_a$ is a quantum statistical factor, with $d_a = -1$ for fermions (Fermi-Dirac distribution), $d_a = +1$ for {bosons} (Bose-Einstein distribution) and $d_a = 0$ reduces to the {classical Maxwell-Boltzmann distribution}. And $\beta$ (or $1/k_B T$) is inverse temperature that scales the energy. For electrons, the chemical potential $\lambda_e$ is related to the electron density $n_e$ through the \textit{Fermi-Dirac integral}. 

In the \textbf{nondegenerate (high-temperature) limit}, quantum effects become negligible, implying that the thermal energy $k_B T$ significantly exceeds quantum energy-level spacings \cite{landau1980,pathria2021}. In this limit, we are saying we want to include Heisenberg uncertainty, but we expect the Pauli exclusion principle not to play a significant role, and hence we can assume the distribution is mostly non-degenerate. Under this condition, the chemical potential simplifies to:
\begin{equation}
\lambda_e \approx k_B \mathcal{T} \ln \left( \frac{n_e \lambda_{\text{th}}}{2} \right)
\end{equation}
where $\lambda_{\text{th}} = \frac{\hbar}{\sqrt{2\pi m_e k_B \mathcal{T}}}$ is the \textbf{thermal de Broglie wavelength}.
Therefore, in our case, we can define the Fermi distribution as:
\begin{align*}
\tilde{\tilde{f}} &= \frac{l}{n_0 \tau} \tilde{f} = \frac{l m_e}{\pi \hbar n_0 \tau} \cdot  \frac{1}{\exp\left( -\ln\left(\frac{n_e \lambda_{th}}{2} \right) + \frac{m_e}{2k_B \mathcal{T}} (\tilde{v} - \tilde{u})^2 \right) + 1}
\end{align*}
From the nondimensional setting of Wigner-Poisson systems, we have:
\begin{align*}
\tilde{\tilde{f}} &= \frac{\lambda_D \omega_{pe} m_e}{\pi \hbar n_0} \cdot \frac{1}{\exp\left( \frac{-1}{k_B \mathcal{T}} \lambda_e + \frac{m_e}{2k_B \mathcal{T}} (\tilde{v} - \tilde{u})^2 \right) + 1} = \frac{\lambda_D \omega_{pe} m_e}{\pi \hbar n_0}  \cdot \frac{1}{\left(\frac{n_e \lambda_{th}}{2} \right)^{-1} \exp\left(\frac{1}{2\hat{\mathcal{T}}} (v - u)^2 \right)+ 1}
\end{align*}
where $\tilde{\rho} = \int \tilde{\tilde{f}} ~ d\tilde{v} =\lambda_D \omega_{pe} \int \tilde{\tilde{f}}~ dv$ and 
$\tilde{u} = \frac{1}{\tilde{\rho}}\int \tilde{v} \tilde{\tilde{f}}~d\tilde{v} = \frac{1}{\lambda_D \omega_{pe} \rho}\int (\lambda_D \omega_{pe})^2 v \tilde{\tilde{f}}~dv  = \lambda_D \omega_{pe} \frac{1}{\rho}\int v \tilde{\tilde{f}}~dv = \lambda_D \omega_{pe} u$. And note that:
\[
\frac{\lambda_D \omega_{pe} m_e}{\pi \hbar n_0} = \frac{\lambda_D \omega_{pe} m_e}{(\pi H (n_0 m_e \lambda_D^2 \omega_{pe}))} = \frac{1}{\pi H  \lambda_D n_0}, \]

\[\left(\frac{n_e \lambda_{th}}{2} \right)^{-1} = \frac{2}{n_e \lambda_{th}} = \frac{2}{n_e (\frac{\hbar}{\sqrt{2\pi m_e k_B \mathcal{T}}})} = \frac{2 \sqrt{2\pi\hat{\mathcal{T}}}}{n_e H \lambda_{D}} = \frac{2 \sqrt{2\pi\hat{\mathcal{T}}}}{n_0 \rho H \lambda_{D}}.
\]
Thus, the Fermi-Dirac distribution function in dimensionless form becomes:
\begin{equation}\label{eq:fermi-dirac}
    \tilde{\tilde{f}} =  \frac{1}{\pi H  \lambda_D n_0} \cdot \frac{1}{\frac{2 \sqrt{2\pi \hat{\mathcal{T}}}}{n_0 \rho H \lambda_{D}} \exp\left(\frac{1}{2\hat{\mathcal{T}}} (v - u)^2 \right)+ 1}  ~.~
\end{equation}
Due to the nature of what the Fermi-Dirac distribution describes, if we choose an initial density of $\rho=\rho_0 = 1$ in \eqref{eq:fermi-dirac},
%
it is important to note that the non-dimensional form integrates to $\int \int \tilde{\tilde{f}} dx\,dv = C$ instead of integrating to $1$. To relate $\tilde{\tilde{f}}$ to our non-dimensional $f$ in the first part of this section, we introduce the following scaling Fermi-Dirac $f = \frac{1}{C} \tilde{\tilde{f}}$ to have the property that $\int f dv = \rho$. This modification is to make the non-dimensional form of scaled Fermi-Dirac distribution consistent with  our choices for non-dimensionalization of Wigner in part one of this section.  Unlike the Maxwell-Boltzmann distribution, the constant $C$ cannot be explicitly computed and an approximated number will be given later in Section 3. This is a form explicitly showing the classical Maxwell-Boltzmann statistics recovered at high temperatures. 

\begin{equation}
    f =  \frac{1}{C} \frac{1}{\pi H  \lambda_D n_0} \cdot \frac{\rho}{\frac{2 \pi \sqrt{2\pi \hat{\mathcal{T}}}}{\pi H \lambda_{D} n_0} \exp\left(\frac{1 }{2\hat{\mathcal{T}}} (v - u)^2 \right)+ \rho } ~,~
\end{equation}
Define $\alpha=\frac{1}{\pi H  \lambda_D n_0}$,
\begin{equation}
    f =  \frac{\alpha}{C}  \cdot \frac{\rho}{\alpha 2 \pi \sqrt{2\pi \hat{\mathcal{T}}}\exp\left(\frac{1 }{2\hat{\mathcal{T}}} (v - u)^2 \right)+ \rho } ~.~
\end{equation}
It's straightforward to show that as $\alpha \rightarrow \infty$, $f\rightarrow M$, where $M$ is the Maxwell-Boltzmann distribution.  We note that in practice,  while $H$ is a scale invariant parameter,  $\alpha$ depends on the plasma under consideration.  However, for demonstration in the numerical section, we replace $\lambda_D$ and $n_0$ in $\alpha$ with a  non-dimensional $\tilde{\lambda}_D = \sqrt{ \frac{\hat{\mathcal{T}}_0}{\hat{n}_0} }$ and $\hat{n}_0$, where 
$\hat{\mathcal{T}}_0$ is the initial non-dimensional bulk \lq temperature' of the initial condition over the domain and $\hat{n}_0$ is the initial non-dimensional bulk \lq density' over the domain, which here will be taken to be $\hat{n}_0=1$. In this form, we note that $C$ will need to be computed numerically.   

\begin{lemma}[Moment consistency of the Fermi--Dirac]\label{lem:Fermi consistency}
Let \(\rho>0\), \(\hat {\mathcal{T}}>0\), and define
\begin{equation}
    f^{FD}(v)=\frac{\alpha \rho}{C}\frac{1}{A\exp\left(\frac{(v-u)^2}{2\hat {\mathcal{T}}}\right)+\rho},
    \qquad
    A=\alpha 2\pi\sqrt{2\pi\hat{\mathcal{T}}}.
\end{equation}
Assume the velocity domain is symmetric about \(u\), or equivalently take \(v\in(-\infty,\infty)\). If
\begin{equation}
    C=\alpha\sqrt{2\hat {\mathcal{T}}}\,I_0,
    \qquad
    I_0=\int \frac{ds}{A\exp(s^2)+\rho},
\end{equation}
then
\begin{equation}
    \int f^{FD}\,dv=\rho,
    \qquad
    \int v f^{FD}\,dv=\rho u.
\end{equation}
Moreover,
\begin{equation}
    \int \frac{1}{2}v^2 f^{FD}\,dv=\rho\left(\frac{1}{2}u^2+e^{FD}\right),
    \qquad
    e^{FD}=\hat {\mathcal{T}}\frac{I_1}{I_0},
    \qquad
    I_1=\int \frac{s^2\,ds}{A\exp(s^2)+\rho}.
\end{equation}
Consequently, the relaxation operator \(Q(f)=\frac{1}{\tau}(f^{FD}-f)\) preserves mass and momentum. It also preserves the second moment if \(\hat {\mathcal{T}}\) is chosen so that \(e^{FD}=e\), where \(e\) is the internal energy of \(f\).
\end{lemma}
Proof of this lemma is provided in the Appendix \ref{apdix:lem fermi}.

\section{Full-rank formulation of the conservative correction}\label{sec:full-rank solve}

The low-rank method introduced later relies on a conservative macroscopic correction. We first describe this correction in the full-rank setting, where all phase-space degrees of freedom are available. To our knowledge, this gives the first conservative formulation of the original full-rank Strang-splitting Wigner-Poisson method of Suh~\cite{suh1991numerical}. The formulation couples the kinetic update to a macroscopic density--momentum solve, producing a kinetic solution that is locally conservative in mass and momentum, while total energy is enforced globally at the kinetic level. This improvement is important because conservation of these invariants directly affects the physical fidelity of the computed Wigner-Poisson dynamics. The full-rank construction also provides the foundation for the conservative correction incorporated into the adaptive rank framework in Section~\ref{sec:low-rank solve}.

In the macroscopic solve, we use
\begin{equation}\label{eqn:temperature closure}
    P_i = \rho_i \mathcal{T}_i,
\end{equation}
where the temperature $\mathcal{T}_i$ is computed from the provisional kinetic distribution $f^{*,n+1}$ using~\eqref{eq:temp} and is then frozen during the implicit density--momentum solve. Thus, the fluid solve supplies corrected density and momentum targets without introducing a closure for the third moment.
The choice of $\mathcal{T}_i$ is tied to the global energy balance: the temperature is computed from the kinetic state, while the final conservation of total energy accounts for both kinetic and electrostatic field energy. This treatment is used throughout both the full-rank and low-rank conservative corrections.

\subsection{Conservative Formulation}

We define the electric field by \(E=-\Phi_x\). Since the nondimensional Poisson equation is \(-\Phi_{xx}=\rho-1\), we have  \(E_x=\rho-1\).
Therefore,
\begin{equation}
\rho E=(1+E_x)E
      =E+\frac12\partial_x(E^2)
      =\partial_x\left(\frac12E^2-\Phi\right).
\end{equation}
Thus the electrostatic force contribution in the momentum equation can
be written in conservative form with
\begin{equation}
F_{\rm EM}=\frac12E^2-\Phi.
\end{equation}
This reformulation is useful numerically because the force contribution is written only in terms of the electrostatic potential \(\Phi\), or equivalently the electric field \(E\), so that once the Poisson solve is completed it can be treated as a known source in the macroscopic update.
The continuous fluid system can be expressed in conservative vector form:
\begin{equation}
    \frac{\partial \mathbf{U}}{\partial t} + \frac{\partial \mathbf{F}}{\partial x} = \mathbf{S} \label{Newton_method}
\end{equation}
where the conservative state \(\mathbf{U}\), flux \(\mathbf{F}\), and source \(\mathbf{S}\) are given by:

\begin{equation}
    \mathbf{U}=\begin{pmatrix} \rho \\ \rho u \\ \rho {\mathcal{T}}+\rho u^2 \end{pmatrix}, \quad
    \mathbf{F}=\begin{pmatrix} \rho u \\ \rho {\mathcal{T}}+\rho u^2 \\ \rho u^3+Q+3\rho u{\mathcal{T}} \end{pmatrix}, \quad
    \mathbf{S}=\begin{pmatrix} 0 \\ \frac{\partial F_{\text{EM}}}{\partial x} \\ -2\rho u\Phi_x \end{pmatrix}.
\end{equation}
The vector form above records the first three moment identities associated with the Wigner-Poisson equation. We do not solve this full three-equation system as a closed quantum-fluid model, because the third equation contains the unclosed flux $Q$. Instead, the implicit macroscopic solve uses only the density and momentum equations. These equations provide corrected target moments $\rho^{n+1}$ and $(\rho u)^{n+1}$. The total-energy constraint is imposed afterward at the kinetic level through the discrete moment correction described in Section 3.4.

Therefore, the fluid system used in the implicit solve is the reduced conservative system
\begin{equation}
    \mathbf{U} = \begin{pmatrix} \rho \\ \rho u \end{pmatrix}, \quad \mathbf{F}(U) = \begin{pmatrix} \rho u \\ \rho {\mathcal{T}} + \rho u^2 \end{pmatrix}, \quad \mathbf{S} = \begin{pmatrix} 0 \\ \frac{\partial F_{\text{EM}}}{\partial x} \end{pmatrix} \label{fluid_model}
\end{equation}

\subsection{Finite Volume Discretization}
Integrating over a computational cell \(i\) of width \(\Delta x\) using a Backward Euler time discretization with step \(\Delta t\), the implicit update is:
\begin{equation}
    \mathbf{U}_i^{n+1} - \mathbf{U}_i^n + \lambda \left( \mathbf{F}_{i+1/2}^{n+1} - \mathbf{F}_{i-1/2}^{n+1} \right) - \Delta t \mathbf{S}_i^{n+1} = 0
\end{equation}
where \(\lambda = \Delta t / \Delta x\). The interface fluxes are evaluated using the Rusanov (Local Lax-Friedrichs) scheme:
\begin{equation}
    \mathbf{F}_{i+1/2}^{n+1} = \frac{1}{2} \left( \mathbf{F}_i^{n+1} + \mathbf{F}_{i+1}^{n+1} \right) - \frac{\alpha_R}{2} \left( \mathbf{U}_{i+1}^{n+1} - \mathbf{U}_i^{n+1} \right)
\end{equation}
The numerical diffusion coefficient \(\alpha\) is determined by the maximum absolute characteristic wave speed of the local fluid. By transforming the governing equations into primitive form, the eigenvalues of the primitive Jacobian are found to be \(\lambda_{1,2} = u \pm \sqrt{{\mathcal{T}}}\), where \(\sqrt{{\mathcal{T}}}\) is the isothermal wave speed. Thus, we set \(\alpha = |u| + \sqrt{{\mathcal{T}}}\).

Substituting the Rusanov fluxes into the finite volume update and simplifying yields the exact non-linear residual \(\mathbf{R}_i\) for cell \(i\) at the advanced time level (dropping the \(n+1\) superscript for clarity):
\begin{equation}
    \mathbf{R}_i = \mathbf{U}_i - \mathbf{U}_i^n + \frac{\lambda}{2}(\mathbf{F}_{i+1} - \mathbf{F}_{i-1}) - \frac{\lambda \alpha_R}{2}(\mathbf{U}_{i+1} - \mathbf{U}_i) + \frac{\lambda \alpha_L}{2}(\mathbf{U}_i - \mathbf{U}_{i-1}) - \Delta t \mathbf{S}_i = 0
\end{equation}

\subsection{Newton-Raphson Linearization}
To find the root of the non-linear residual \(\mathbf{R}(\mathbf{U}) = 0\), we apply a Newton-Raphson iteration. Crucially, to prevent division by zero in near-vacuum regions where \(\rho \to 0\), we formulate the update vector with respect to the variables \(\mathbf{W} = (\rho, u)^T\). 

The linearized system takes the form \(\mathbf{J} \Delta \mathbf{W} = -\mathbf{R}\), where \(\mathbf{J} = \frac{\partial \mathbf{R}}{\partial \mathbf{W}}\). The analytic block Jacobian requires the derivatives of the conservative state and flux vectors with respect to the primitive variables:
\begin{equation}
    \mathbf{J}_U = \frac{\partial \mathbf{U}}{\partial \mathbf{W}} = \begin{pmatrix} 1 & 0 \\ u & \rho \end{pmatrix}, \quad \mathbf{J}_F = \frac{\partial \mathbf{F}}{\partial \mathbf{W}} = \begin{pmatrix} u & \rho \\ {\mathcal{T}} + u^2 & 2\rho u \end{pmatrix}
\end{equation}
Note that \(\frac{\partial \mathbf{S}}{\partial \mathbf{W}} = 0\) since \(F_{\text{EM}}\) is treated as constant during the fluid update step.

Differentiating the residual \(\mathbf{R}_i\) with respect to the adjacent primitive states yields the \(2 \times 2\) blocks comprising the global sparse tridiagonal Jacobian matrix:

\noindent \textbf{Left Block (\(i-1\)):}
\begin{equation}
    \mathbf{J}_{\text{Left}} = -\frac{\lambda}{2} \mathbf{J}_{F, i-1} - \frac{\lambda \alpha_L}{2} \mathbf{J}_{U, i-1} = \begin{pmatrix} -\frac{\lambda}{2} (u_{i-1} + \alpha_L) & -\frac{\lambda}{2} \rho_{i-1} \\ -\frac{\lambda}{2}({\mathcal{T}}_{i-1} + u_{i-1}^2 + \alpha_L u_{i-1}) & -\frac{\lambda}{2}\rho_{i-1}(2u_{i-1} + \alpha_L) \end{pmatrix}
\end{equation}

\noindent \textbf{Right Block (\(i+1\)):}
\begin{equation}
    \mathbf{J}_{\text{Right}} = \frac{\lambda}{2} \mathbf{J}_{F, i+1} - \frac{\lambda \alpha_R}{2} \mathbf{J}_{U, i+1} = \begin{pmatrix} \frac{\lambda}{2} (u_{i+1} - \alpha_R) & \frac{\lambda}{2} \rho_{i+1} \\ \frac{\lambda}{2}({\mathcal{T}}_{i+1} + u_{i+1}^2 - \alpha_R u_{i+1}) & \frac{\lambda}{2}\rho_{i+1}(2u_{i+1} - \alpha_R) \end{pmatrix}
\end{equation}

\noindent \textbf{Center Block (\(i\)):}
\begin{equation}
    \mathbf{J}_{\text{Center}} = \left( 1 + \frac{\lambda \alpha_R}{2} + \frac{\lambda \alpha_L}{2} \right) \mathbf{J}_{U, i} = \begin{pmatrix} K & 0 \\ K u_i & K \rho_i \end{pmatrix}
\end{equation}
where the scalar multiplier is defined as \(K = 1 + \frac{\lambda}{2}(\alpha_R + \alpha_L)\). Solving this sparse block tridiagonal Newton system gives the primitive increments \(\Delta\rho\) and \(\Delta u\). The nonlinear density-momentum residual is solved to the prescribed Newton tolerance.

The conservative correction has two stages. First, a Fermi-Dirac-type correction transfers the density and momentum update from the macroscopic solve back to the kinetic grid. In the full-rank implementation this correction is smoothed in the spatial direction to reduce grid-scale oscillations introduced by the cellwise normalization. Second, after the filtered correction is added to the
kinetic state, a global quadratic moment multiplier is applied so that the final state satisfies the prescribed global mass, momentum, and total-energy constraints. The distinction between these two stages is important: the Fermi-Dirac correction is local before filtering, while the final quadratic moment correction is global.

\subsection{Filtered Fermi-Dirac correction}

After the implicit macroscopic solve, the full-rank kinetic state is corrected through a Fermi-Dirac-type reconstruction. Let \(f^{\ast,n+1}_{i,j}\) denote the provisional kinetic state obtained after the Wigner-Poisson update, and let \(\rho_i^\ast\), \(u_i^\ast\), and \({\mathcal{T}}_i^\ast\) be the corresponding density, bulk velocity, and temperature computed from this state. The conservative macroscopic solve provides updated target values \(\rho_i^{n+1}\) and \(u_i^{n+1}\). We then construct the Fermi-Dirac-type correction
\begin{equation}
\Delta F^{FD}_{i,j}
=
F^{FD}(\rho_i^{n+1},u_i^{n+1},{\mathcal{T}}_i^\ast;v_j)
-
F^{FD}(\rho_i^\ast,u_i^\ast,{\mathcal{T}}_i^\ast;v_j), \label{fermi_rec}
\end{equation}
where \(F^{FD}\) denotes the discrete Fermi-Dirac-type reconstruction defined in Section 2.3. The temperature is kept fixed at \({\mathcal{T}}_i^\ast\) during this step. Thus, the purpose of the Fermi-Dirac correction is to transfer the density and momentum update from the macroscopic solve back to the full kinetic grid, not to solve an additional temperature-matching problem.

In the full-rank implementation, this correction is represented on the entire \(N_x\times N_v\) phase-space grid. Since the Fermi-Dirac normalization is computed locally in each spatial cell, small grid-scale variations in \(\rho_i\), \(u_i\), or \({\mathcal{T}}_i\) can produce sharper oscillations in the reconstructed correction. This effect is especially pronounced through the local normalization factor \(C_i\), which depends nonlinearly on the macroscopic fields through the velocity quadrature. For this reason, the full-rank scheme applies a spatial Gaussian filter to the Fermi-Dirac correction before adding it to the kinetic solution.

The filter is applied independently to each velocity column. For fixed \(v_j\), the filtered correction is defined by
\begin{equation}
\overline{\Delta F}^{FD}_{i,j}
=
\sum_{a=-m}^{m}G_a\,\Delta F^{FD}_{i-a,j},
\end{equation}
where
\begin{equation}
G_a
=
\frac{\exp(-a^2/(2\sigma^2))}
{\sum_{\ell=-m}^{m}\exp(-\ell^2/(2\sigma^2))},
\qquad a=-m,\ldots,m.
\end{equation}
Here \(\sigma\) is measured in spatial grid cells. In the reported full-rank computations we use a Gaussian filter with \(\sigma=1\) for a weak filter, and \(\sigma=4\) for a strong filter. When the stencil width is not prescribed, the implementation takes \(m=\lceil 3\sigma\rceil\), so that the stencil includes the dominant support of the discrete Gaussian kernel. The filter is applied with periodic wrapping in the spatial direction for the periodic benchmarks considered here.

The normalization of the Gaussian weights gives
\begin{equation}
\sum_{a=-m}^{m}G_a=1.
\end{equation}
Therefore, under periodic boundary conditions and without clipping, the filter preserves the spatial sum of each velocity column:
\begin{equation}
\sum_i \overline{\Delta F}^{FD}_{i,j}
=
\sum_i \Delta F^{FD}_{i,j}.
\end{equation}
Consequently, the filter preserves the globally integrated velocity moments of the correction after summing over both \(x_i\) and \(v_j\). This is a global conservation property, not a pointwise one. The filter redistributes the Fermi-Dirac correction across neighboring spatial cells and therefore does not preserve the corrected density or momentum at each individual cell.

The filtered full-rank state is then updated by
\begin{equation}
\widetilde f^{n+1}_{i,j}
=
f^{\ast,n+1}_{i,j}
+
\overline{\Delta F}^{FD}_{i,j}.
\end{equation}
This state is subsequently passed to the global discrete moment correction described in the next subsection. In the computations reported here, no positivity floor is applied after filtering. This is consistent with the sign-indefinite nature of the Wigner distribution: the Gaussian filter is used only as a full-rank regularization of the Fermi-Dirac correction, not as a positivity-preserving limiter.

This filtering step is specific to the full-rank implementation. In the adaptive rank scheme described in Section 4, we do not apply this explicit Gaussian filter. The low-rank update and recompression already remove components that are not represented in the retained low-rank subspace, so the additional full-rank smoothing step is not used there.

\subsection{Global kinetic moment correction for the total-energy constraint}

A central contribution of this work is a global kinetic moment correction that enforces the discrete total-energy constraint through a small algebraic system. Moment correction based on Lagrange multipliers are commonly used to enforce density, and sometimes momentum, constraints. Here, however, the multiplier is used to correct the kinetic second moment so that the final distribution is compatible with the specified total energy. To the best of our knowledge, this is the first use of such global velocity-polynomial multiplier for enforcing the total energy constraint in a conservative adaptive rank Wigner-Poisson solver.

The Fermi-Dirac-type correction is constructed from the locally updated density and momentum supplied by the macroscopic solve and transfers this update back to the kinetic solution before filtering and recompression. These density and momentum updates are obtained from the locally conservative macroscopic solve, not from purely global constraint. However, after the Wigner-Poisson update, the Fermi-Dirac-type correction, filtering, and quadrature errors, the corrected kinetic state may no longer satisfy the desired discrete total energy. We therefore apply a global algebraic correction to the filtered state $\widetilde{f}^{n+1}$. This step  should be distinguished from the globally conservative formulation in \cite{christlieb2026structurepreservingadaptiverankapproachhighdimensional}: here, mass and momentum are  supplied by the local macroscopic solve, while the global multiplier is used to match the corresponding integrated mass and momentum and to enforce the total energy constraint.

The correction uses a velocity-dependent quadratic multiplier that is constant in space:
\begin{equation}
f^{n+1}_{i,j}
=
\widetilde f^{n+1}_{i,j}
\left(k_0+k_1v_j+k_2v_j^2\right).
\end{equation}
The same coefficients \(k_0,k_1,k_2\) are applied to all interior spatial cells. Thus, this is a global kinetic moment correction, instead of a cellwise or pointwise correction.

Let \(w_j\) denote the trapezoidal velocity weights, with \(w_0=w_{N_v-1}=1/2\) and \(w_j=1\) for \(1\leq j\leq N_v-2\). For \(\ell=0,\ldots,4\), define the global raw moments of the filtered state by
\begin{equation}
M_\ell
=
\Delta x\,\Delta v
\sum_{i=1}^{N_x}
\sum_{j=0}^{N_v-1}
w_j v_j^\ell \widetilde f^{n+1}_{i,j},
\end{equation}
where only physical interior cells are included.

The target mass and momentum are the spatial integrals of the locally corrected macroscopic variables:
\begin{equation}
M_0^\star
=
\Delta x\sum_{i=1}^{N_x}\rho_i^{n+1},
\qquad
M_1^\star
=
\Delta x\sum_{i=1}^{N_x}m_i^{n+1},
\qquad
m_i^{n+1}=\rho_i^{n+1}u_i^{n+1}.
\end{equation}
These two constraints ensure that the global multiplier does not change the total mass and momentum supplied by the macroscopic solve, while maintaining local structure of mass and momentum.

The third constraint comes from the discrete total energy. Let
\begin{equation}
\mathcal E_{tot}^0
=
\frac12 \Delta x\,\Delta v
\sum_{i=1}^{N_x}
\sum_{j=0}^{N_v-1}
w_j v_j^2 f^0_{i,j}
+
\frac12\Delta x
\sum_{i=1}^{N_x}
(E_i^0)^2
\end{equation}
be the initial total energy. At time \(t^{n+1}\), compute
\begin{equation}
E_i^{n+1}
=
-\frac{\Phi_{i+1}^{n+1}-\Phi_{i-1}^{n+1}}{2\Delta x},
\qquad
\mathcal E_{\rm field}^{n+1}
=
\frac12\Delta x
\sum_{i=1}^{N_x}
(E_i^{n+1})^2 .
\end{equation}
The target kinetic second moment is then
\begin{equation}
M_2^\star
=
2\left(\mathcal E_{tot}^0-\mathcal E_{\rm field}^{n+1}\right),
\end{equation}
so that enforcing \(M_2^\star\) enforces the kinetic energy required by the total-energy balance.

Substituting the corrected distribution into the three integrated moment constraints gives
\begin{equation}
\begin{pmatrix}
M_0 & M_1 & M_2\\
M_1 & M_2 & M_3\\
M_2 & M_3 & M_4
\end{pmatrix}
\begin{pmatrix}
k_0\\
k_1\\
k_2
\end{pmatrix}
=
\begin{pmatrix}
M_0^\star\\
M_1^\star\\
M_2^\star
\end{pmatrix}. \label{lagrange_system}
\end{equation}
When the determinant of the moment matrix is larger than the prescribed threshold, the system is solved explicitly using Cramer's rule. In the current implementation the threshold is \(10^{-25}\). If the determinant does not exceed this threshold, the correction is skipped by setting \(k_0=1\), \(k_1=0\), and \(k_2=0\).

This correction is therefore global only in the sense that the multiplier is determined from spatially integrated kinetic moments. It does not replace the local mass and momentum conservation supplied by the macroscopic solve, nor does it enforce density or momentum pointwise in \(x\). Its role is to make the final kinetic distribution consistent with the integrated mass and momentum targets and with the prescribed total energy.

Because Wigner distributions may be sign-indefinite, the multiplier is not a positivity-preserving limiter. The polynomial factor \(k_0+k_1v_j+k_2v_j^2\) is used only for moment matching. Large coefficients could amplify unresolved velocity oscillations if the moment matrix is ill-conditioned; the determinant threshold is therefore a conditioning safeguard, not a positivity safeguard.

\begin{algorithm}[H]
\caption{Full-rank conservative correction for Wigner-Poisson}
\begin{algorithmic}[1]
\Require Distribution \(f^n_{i,j}\), potential \(\Phi^n_i\), time step
\(\Delta t\), velocity grid \(\{v_j\}\), and filter weights \(\{G_a\}\).

\State Advance the full-rank Wigner-Poisson equation over one time step to
obtain the provisional kinetic state \(f^{\ast,n+1}_{i,j}\).

\State Compute provisional moments
\(\rho_i^\ast\), \((\rho u)_i^\ast\), and \({\mathcal{T}}_i^\ast\) from
\(f^{\ast,n+1}_{i,j}\).

\State Solve Poisson's equation using \(\rho_i^\ast\) or the corrected density
specified by the time discretization, and compute \(\Phi^{n+1}\) and
\(E^{n+1}\).

\State Solve the implicit conservative density-momentum system to obtain target moments \(\rho_i^{n+1}\) and \((\rho u)_i^{n+1}\) (\ref{fluid_model}).

\State Construct the Fermi-Dirac-type correction
\(\Delta F_{i,j}\) from
\((\rho_i^\ast,u_i^\ast,{\mathcal{T}}_i^\ast)\) to
\((\rho_i^{n+1},u_i^{n+1},{\mathcal{T}}_i^\ast)\) (\ref{fermi_rec}).

\State Apply the spatial filter in \(x\) to obtain
\(\overline{\Delta F}_{i,j}\).

\State Set
\(\widetilde f^{n+1}_{i,j}
=
f^{\ast,n+1}_{i,j}
+
\overline{\Delta F}_{i,j}\).

\State Compute the moment-correction coefficients
\(k_{0}\), \(k_{1}\), and \(k_{2}\) from (\ref{lagrange_system})

\State Set
\(f^{n+1}_{i,j}
=
\widetilde f^{n+1}_{i,j}
(k_{0}+k_{1}v_j+k_{2}v_j^2)\).

\State Advance \(t\leftarrow t+\Delta t\).
\end{algorithmic}
\end{algorithm}

\section{Adaptive rank scheme} \label{sec:low-rank solve}
While the full-rank method deals with the complete phase-space solution, the low-rank scheme uses the SVD of the solution \cite{christlieb2025sampling}. The conservative low-rank scheme evaluates the required moments directly from the low-rank factors. If
\begin{equation}
    f^n=U^n\Sigma^n (V^n)^T, \label{f_low}
\end{equation}
then
\begin{align}
    \rho_i &= \Delta v \sum_{\ell=1}^{r} U_{i\ell}\Sigma_{\ell\ell} \sum_{j=1}^{N_v} w_jV_{j\ell}, \label{rho_low} \\
    u_i &= \frac{\Delta v}{\rho_i} \sum_{\ell=1}^{r} U_{i\ell}\Sigma_{\ell\ell} \sum_{j=1}^{N_v} w_jv_j V_{j\ell}, \label{u_low} \\
    {\mathcal{T}}_i &= \frac{\Delta v}{\rho_i} \sum_{\ell=1}^{r} U_{i\ell}\Sigma_{\ell\ell} \sum_{j=1}^{N_v} w_jv_j^2 V_{j\ell} - u_i^2 , \label{T_low}
\end{align}

where \(w_j\) denotes the trapezoidal velocity weights, with \(w_0=w_{N_v-1}=1/2\) and \(w_j=1\) for \(1\leq j\leq N_v-2\).

The same Newton solver is then used. In the low-rank implementation, however, the correction is not immediately added as a full-rank update; instead, the required correction data are stored and incorporated through the sampling function used in the next semi-Lagrangian step.

One important difference between the full-rank and low-rank scheme is that the low-rank does not need a filter, as demonstrated in Fig. \ref{fig:rho_diff} where one can see that the low-rank truncation acts as a compression-induced  filter suppressing high-frequency modes.

\begin{figure}[htbp]
    \centering

    \begin{subfigure}{0.49\textwidth}
        \centering
        \includegraphics[width=\linewidth]{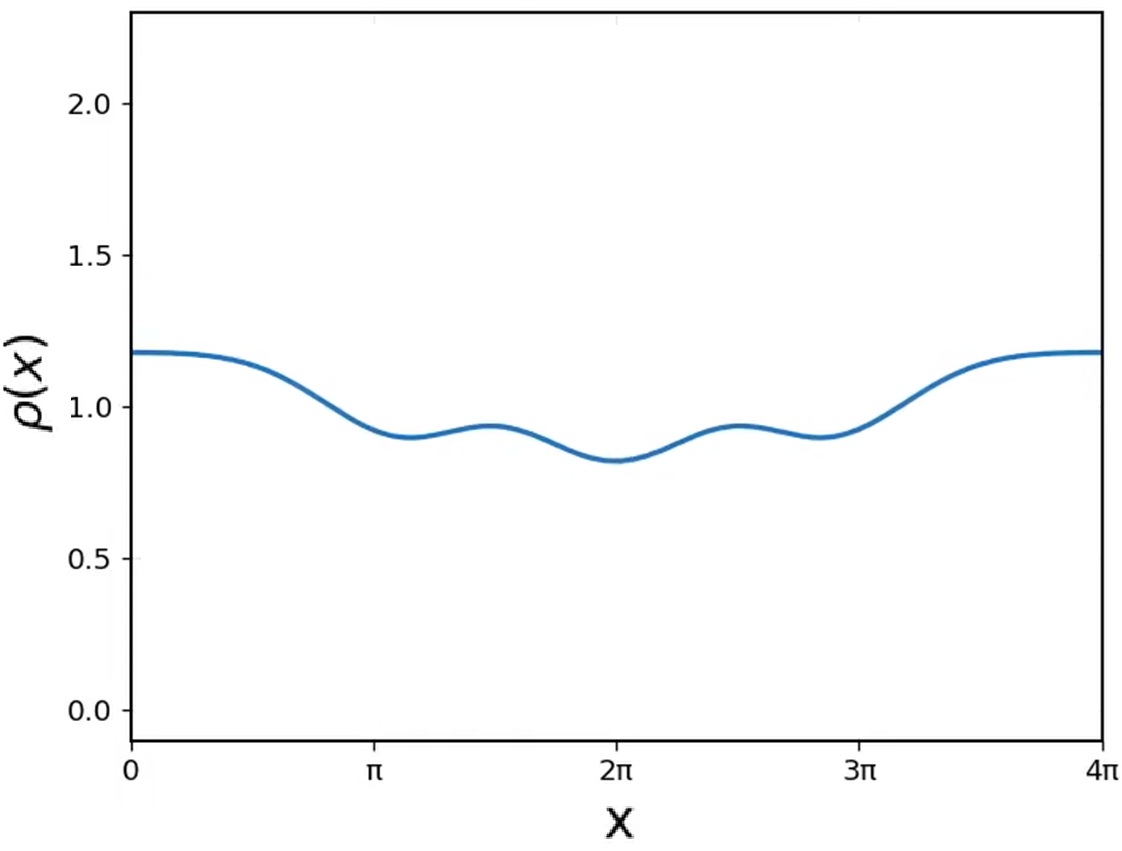}
        \caption{}
        \label{fig:rho_lowrank}
    \end{subfigure}
    \hfill
    \begin{subfigure}{0.49\textwidth}
        \centering
        \includegraphics[width=\linewidth]{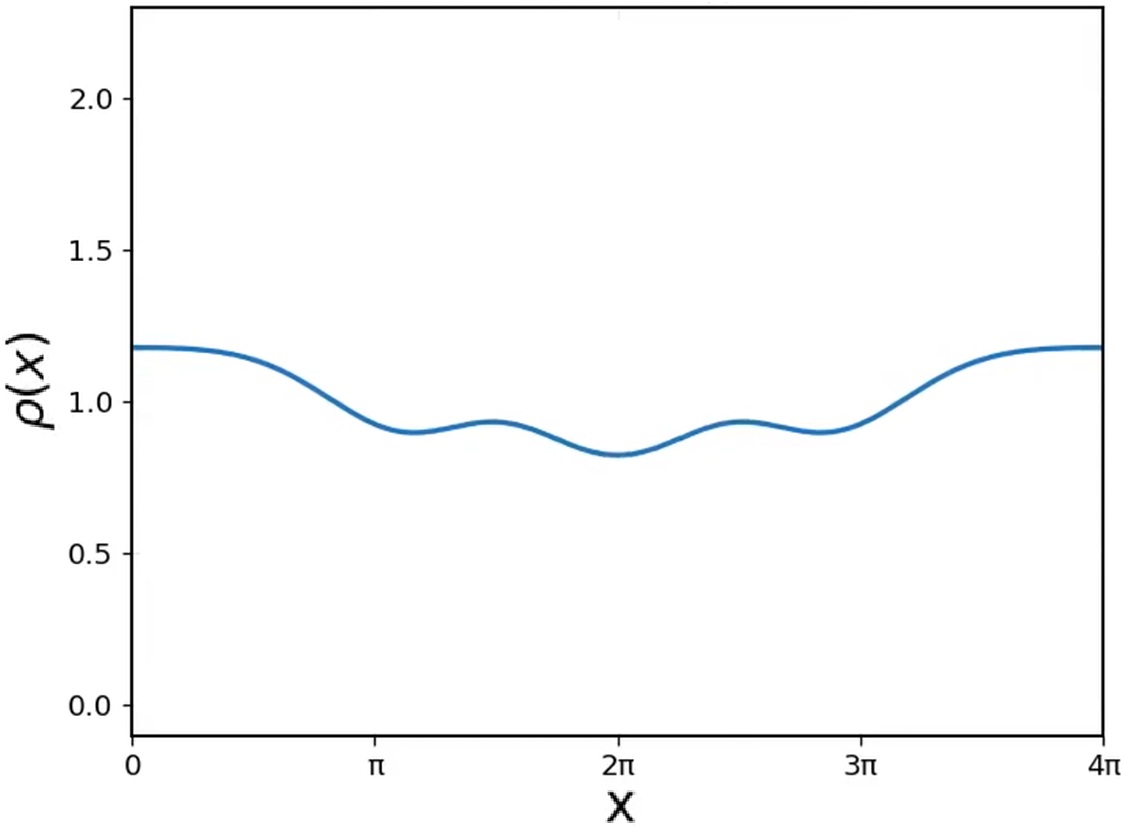}
        \caption{}
        \label{fig:rho_full_4}
    \end{subfigure}

    \vspace{1em} 

    \begin{subfigure}{0.49\textwidth}
        \centering
        \includegraphics[width=\linewidth]{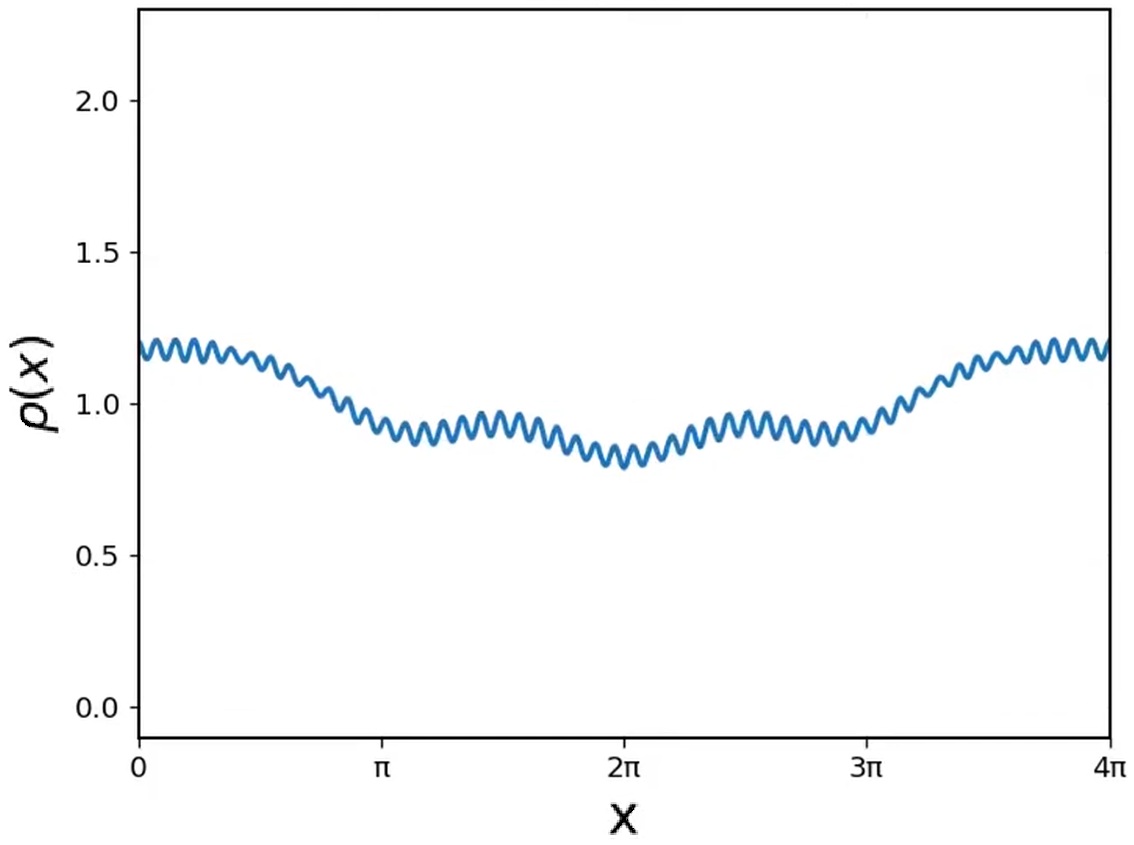}
        \caption{}
        \label{fig:rho_full_1}
    \end{subfigure}

    \caption{
Density \(\rho(x)=\int f(x,v)\,dv\) for the two-stream instability at \(t=40\), with \(H=1\), \(N_x=N_v=512\), CFL \(=10\), \(x\in[0,4\pi]\), and \(v\in[-2\pi,2\pi]\). Panel (a) shows the conservative adaptive rank solution, panel without an explicit Gaussian filter. Panel (b) shows the full-rank solution with Gaussian filter width \(\sigma=4\), and panel (c) shows the full-rank solution with Gaussian filter width \(\sigma=1\). The low-rank result is computed without the explicit full-rank Gaussian filter. 
The weakly filtered full-rank run shows visible grid-scale oscillations, whereas the adaptive rank solution is closer to the strongly filtered result, indicating that ACA-SVD compression suppresses unresolved high-frequency structure.
}\label{fig:rho_diff}
\end{figure}

A practical difference between the full-rank and adaptive rank implementations is the role of filtering. In the full-rank solver, an explicit Gaussian filter is used to control high-frequency oscillations generated by the Wigner update. The adaptive rank method, by contrast, does not require such an additional filtering step. The truncation in the ACA-SVD recompression already removes components that are not represented in the dominant low-dimensional subspace.

To quantify the filtering effect suggested by the density profiles in Fig.~\ref{fig:rho_diff}, we compare the Fourier spectra of the adaptive rank solution with full-rank solutions computed using strong and weak Gaussian filters. For each solution, we subtract the phase-space mean and compute the normalized spectral energy
\[
    \mathcal E(k_x,k_v)
    =
    \frac{|\widehat f(k_x,k_v)|^2}
    {\sum_{k_x,k_v}|\widehat f(k_x,k_v)|^2}.
\]
Fig.~\ref{fig:spectra} shows the folded spatial and velocity marginal spectra,
\[
    P_x(|k_x|)
    =
    \sum_{k_v}\mathcal E(k_x,k_v),
    \qquad
    P_v(|k_v|)
    =
    \sum_{k_x}\mathcal E(k_x,k_v),
\]
together with the high-frequency energy fraction outside the box
\[
    \max\left(
    \frac{|k_x|}{k_{x,\max}},
    \frac{|k_v|}{k_{v,\max}}
    \right) > 0.35 .
\]
In Fig.~\ref{fig:spatial_spectrum}, the spatial marginal spectra for the adaptive rank and strongly filtered full-rank solutions remain close over most resolved spatial modes. In Fig.~\ref{fig:velocity_spectrum}, the separation is strongest in velocity: the adaptive rank spectrum decays most rapidly at high $|k_v|$. Fig.~\ref{fig:high_frequency} summarizes this effect by showing that the adaptive rank high-frequency energy fraction $7.33\times 10^{-2}$ is essentially the same as the strongly filtered full-rank value and lower than the weakly filtered result.

This observation is important for the conservative algorithm. The Fermi-Dirac-type macroscopic correction is applied after the low-rank kinetic update, and the corrected state is then incorporated through ACA-SVD compression. Since the low-rank representation already suppresses unresolved high-frequency content, the conservative correction can be imposed without adding a separate full-rank filtering step.

\begin{figure}[htbp]
    \centering

    \begin{subfigure}{0.49\textwidth}
        \centering
        \includegraphics[width=\linewidth]{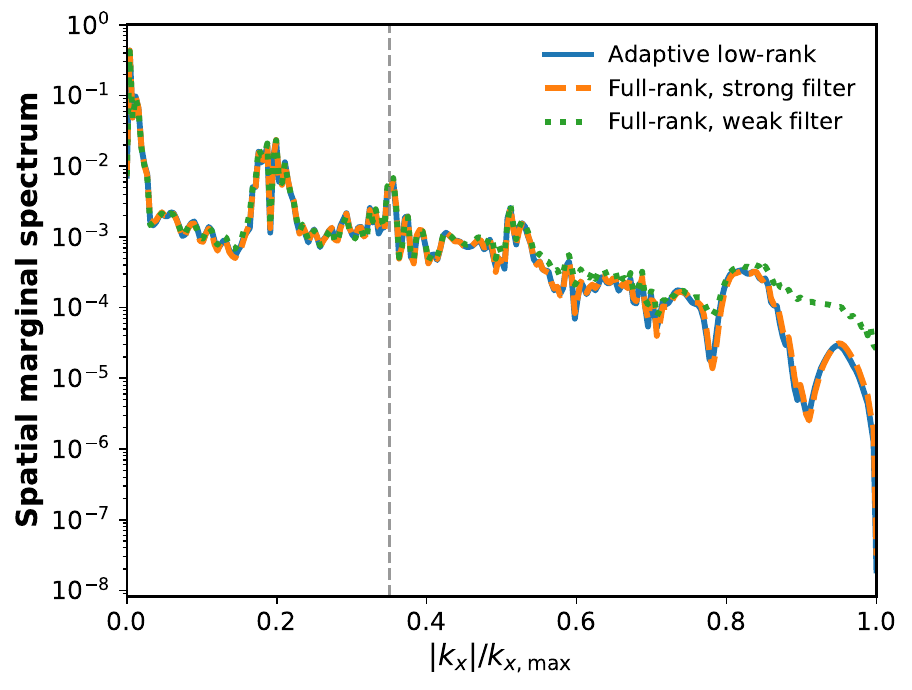}
        \caption{}
        \label{fig:spatial_spectrum}
    \end{subfigure}
    \hfill
    \begin{subfigure}{0.49\textwidth}
        \centering
        \includegraphics[width=\linewidth]{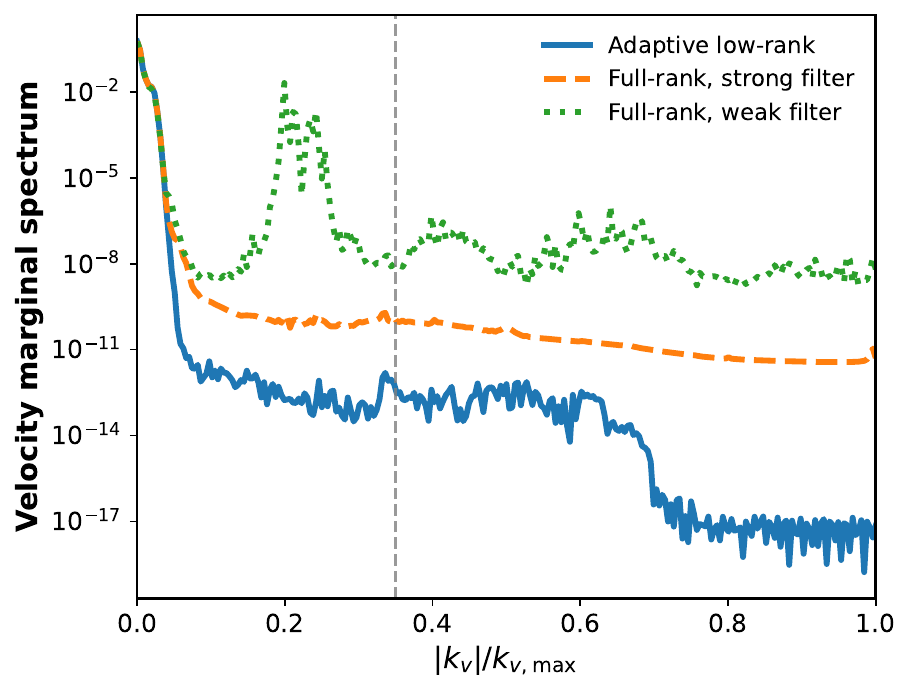}
        \caption{}
        \label{fig:velocity_spectrum}
    \end{subfigure}

    \vspace{1em} 

    \begin{subfigure}{0.49\textwidth}
        \centering
        \includegraphics[width=\linewidth]{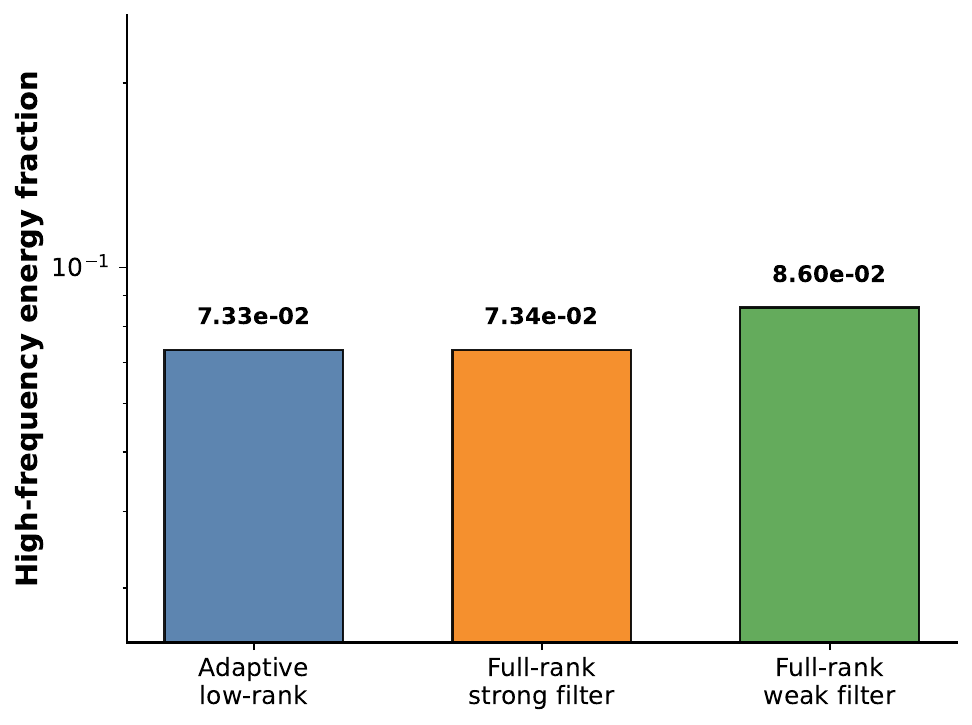}
        \caption{}
        \label{fig:high_frequency}
    \end{subfigure}

    \caption{
Spectral comparison for the two-stream instability at \(t=45\), with \(H=1\), \(N_x=N_v=512\), and CFL \(=10\). The adaptive rank solution is compared with full-rank solutions using strong and weak Gaussian filtering. \textbf{(a)} shows the normalized spatial marginal spectrum as a function of \(|k_x|/k_{x,\max}\), \textbf{(b)} shows the normalized velocity marginal spectrum as a function of \(|k_v|/k_{v,\max}\) where shows the clearest damping of high-frequency content by the adaptive rank representation, while \textbf{(c)} confirms the same trend reporting the high-frequency energy fraction outside the cutoff \(\kappa_c=0.35\). The adaptive rank solution has high-frequency fraction \(7.33\times10^{-2}\), close to the strongly filtered full-rank result \(7.34\times10^{-2}\), while the weakly filtered full-rank run retains \(8.60\times10^{-2}\). This indicates that the adaptive rank representation suppresses unresolved high-frequency content without the explicit full-rank Gaussian filter.
}\label{fig:spectra}
\end{figure}

\begin{algorithm}[H]
\caption{Low-rank incorporation of the conservative correction}
\label{alg:low_rank_conservative_correction}
\begin{algorithmic}[1]
\Require Low-rank factors
\(U^{\ast,n}\), \(\Sigma^{\ast,n}\), \(V^{\ast,n}\),
time step \(\Delta t\), velocity grid \(\{v_j\}\), ACA-SVD tolerance
\(\varepsilon_{\rm svd}\), and stored correction payload
\(\mathcal C^n\). Set \(\mathcal C^0_{i,j}=0\).

\State Define the corrected-start sampling function
\[
\mathcal G^n(i,j)
=
\sum_{\ell=1}^{r_n}
U^{\ast,n}_{i,\ell}
\Sigma^{\ast,n}_{\ell,\ell}
V^{\ast,n}_{j,\ell}
+
\mathcal C^n_{i,j}.
\]

\State Use \(\mathcal G^n\) as the input function handle in the adaptive
low-rank Wigner-Poisson solver of \cite{christlieb2025sampling}
to obtain the provisional low-rank state
\[
f^{\ast,n+1}
=
U^{\ast,n+1}\Sigma^{\ast,n+1}(V^{\ast,n+1})^T .
\]

\State Compute provisional moments
\(\rho_i^\ast\), \((\rho u)_i^\ast\), and \({\mathcal{T}}_i^\ast\)
from the low-rank factors using
\eqref{f_low}-\eqref{T_low}.

\State Solve Poisson's equation using the density specified by the time
discretization, and compute \(\Phi^{n+1}\) and \(E^{n+1}\).

\State Solve the implicit conservative density-momentum system
\eqref{Newton_method},\eqref{fluid_model}
to obtain target moments
\(\rho_i^{n+1}\) and \((\rho u)_i^{n+1}\).

\State Construct the Fermi-Dirac-type correction
\(\Delta F^{n+1}_{i,j}\) from
\((\rho_i^\ast,u_i^\ast,{\mathcal{T}}_i^\ast)\) to
\((\rho_i^{n+1},u_i^{n+1},{\mathcal{T}}_i^\ast)\)
using \eqref{fermi_rec}.

\State Compute the global quadratic moment multiplier
\[
a_j^{n+1}
=
k_0^{n+1}
+
k_1^{n+1}v_j
+
k_2^{n+1}v_j^2
\]
from \eqref{lagrange_system}.

\State Store the correction for the next function handle:
\[
\mathcal C^{n+1}_{i,j}
=
a_j^{n+1}
\left(
f^{\ast,n+1}_{i,j}
+
\Delta F^{n+1}_{i,j}
\right)
-
f^{\ast,n+1}_{i,j}.
\]

\State Advance \(t\leftarrow t+\Delta t\).

\end{algorithmic}
\end{algorithm}

In the current 1D1V implementation the correction term $\mathcal C^{n}$ is stored densely on the $N_x \times N_v $ grid. This is sufficient for the test reported here, but in its not a scalable high-dimensional implementation. For higher-dimensional implementations, this would not be scalable.

\section{Numerical results}\label{sec:numerical results}

The numerical tests are organized around the three claims made in the introduction: accuracy on standard Wigner-Poisson benchmark dynamics, preservation of the targeted global invariants, and retention of low-rank computational efficiency. Each benchmark therefore reports phase-space or density dynamics, electric-energy evolution, adaptive rank behavior, and conservation diagnostics. We emphasize that the conservation diagnostics are global discrete diagnostics computed from the final corrected distribution. They should not be interpreted as local or pointwise density/momentum conservation statements after Gaussian filtering, after the global quadratic moment correction or after ACA-SVD recompression.

We consider three standard 1D1V kinetic benchmarks: two-stream instability, strong Landau damping, and bump-on-tail instability. These tests probe different aspects of the algorithm. The two-stream instability generates nonlinear phase-space deformation, the strong Landau damping problem tests damping and long-time energy exchange, and the bump-on-tail problem examines a beam-driven nonequilibrium instability. Unless otherwise stated, all simulations are performed with $N_x=N_v =512$ and a CFL number of $10$, $\Delta t = \mathrm{CFL} \; * \; \frac{dx}{V}$, with $\Delta x = \frac{L}{Nx}$, $\Delta v=\frac{2V}{Nv}$. For each problem, we report  the electric energy, the adaptive numerical rank, and the conservation errors in mass, momentum, and total energy.

A central point of these experiments is that the conservative correction is built from the Wigner-Poisson moment system, and uses a Fermi-Dirac-type reconstruction, rather than the Maxwellian reconstruction used in other conservative methods for Wigner-Poisson \cite{bonilla2005wigner} and classical Vlasov-Poisson methods. Thus, the results below are intended not only to demonstrate conservation, but also to verify that the quantum-statistical correction is compatible with the adaptive rank representation. In the tests that follow, we observe that the method captures the expected dependence on the quantum parameter $H$, preserves the global discrete invariants to machine precision, and maintains moderate ranks throughout the simulation.

\begin{table}[t]
\centering
\caption{Benchmark parameters for the numerical experiments. Unless otherwise
stated, \(N_x=N_v=512\), CFL \(=10\), and
\(\varepsilon_{\rm svd}=10^{-3}\).}
\label{tab:benchmark-parameters}
\begin{tabular}{lccccc}
\toprule
Benchmark  & \(x\)-domain & \(v\)-domain & \(t_{\rm final}\) & \(H\) values \\
\midrule
Two-stream instability 
& \([0,4\pi]\) 
& \([-2\pi,2\pi]\) 
& \(50\) 
& \(1,2,8\) \\

Strong Landau damping  
& \([0, \frac{2\pi}{0.4}]\) 
& \([-2\pi,2\pi]\) 
& \(50\) 
& \(1,2,8\) \\

Bump-on-tail instability  
& \([0, \frac{2\pi}{0.3}]\) 
& \([-4\pi, 4\pi]\) 
& \(30\) 
& \(1,2,8\) \\

Runtime study  
& \([0,4\pi]\) 
& \([-2\pi,2\pi]\) 
& \(5\) 
& \(1,2,4,8\) \\

CFL sensitivity 
& \([0,4\pi]\) 
& \([-2\pi,2\pi]\) 
& \(50\) 
& \(1\) \\

SVD tolerance
& \([0,4\pi]\) 
& \([-2\pi, 2\pi]\)
& \(25\)
& \(1\) \\
\bottomrule
\end{tabular}
\end{table}

\subsection{Computational efficiency and rank behavior}

We first compare the computational cost of the proposed low-rank method with the corresponding full-rank solver. Fig. \ref{fig:walltime} reports the wall-clock time as the spatial resolution is refined for $H=1,2,8$. The full-rank method exhibits significantly faster growth in runtime because it evolves the entire phase-space solution. In contrast, the proposed low-rank method evolves a compressed representation of the distribution function and adaptively increases the rank only when additional phase-space complexity is required.

\begin{figure}
    \centering
    \begin{subfigure}{0.49\textwidth}
        \centering
        \includegraphics[width=\linewidth]{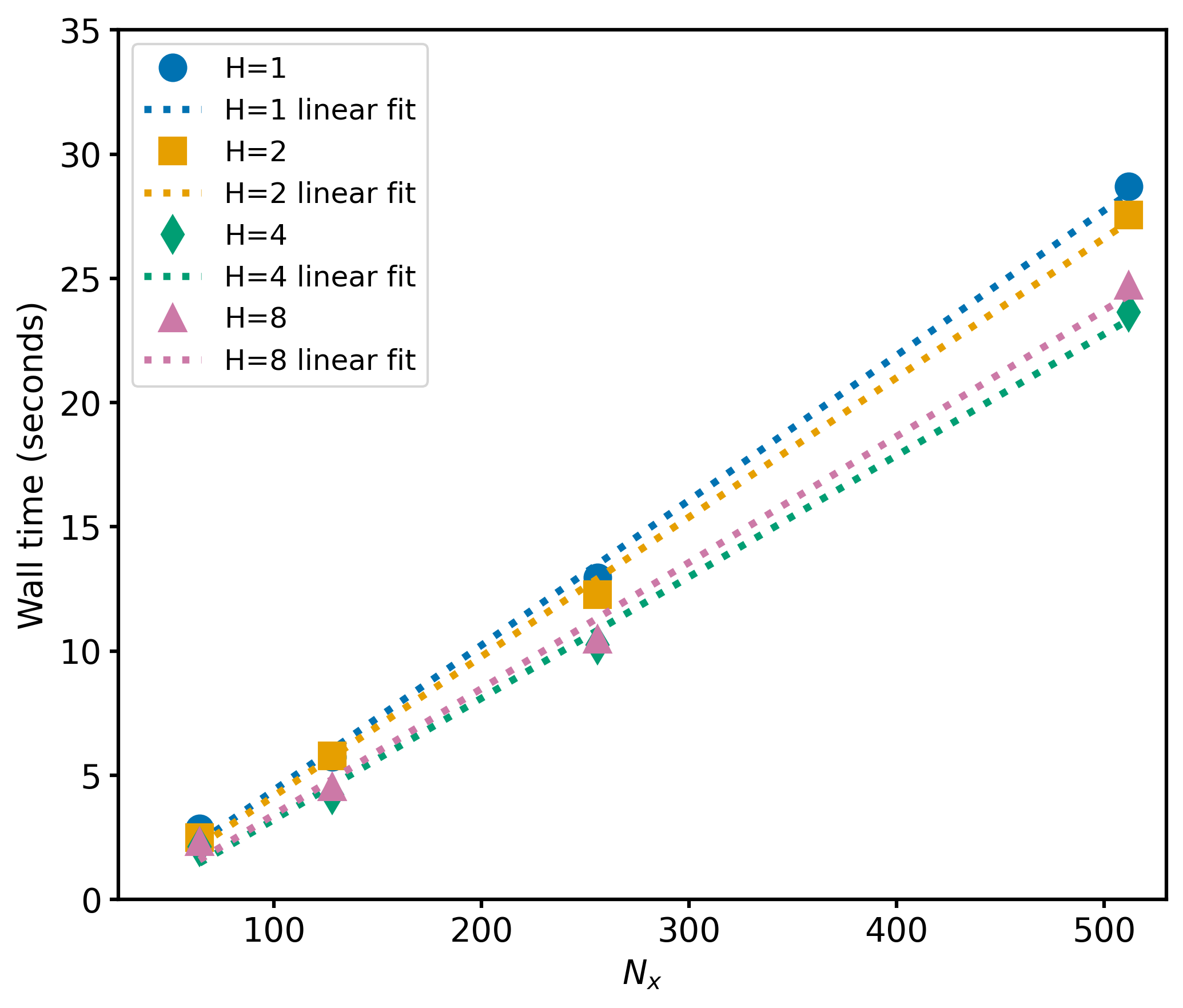}
        \caption{Low rank method}
        \label{fig:low-rank}
    \end{subfigure}
    \hfill
    \begin{subfigure}{0.49\textwidth}
        \centering
        \includegraphics[width=\linewidth]{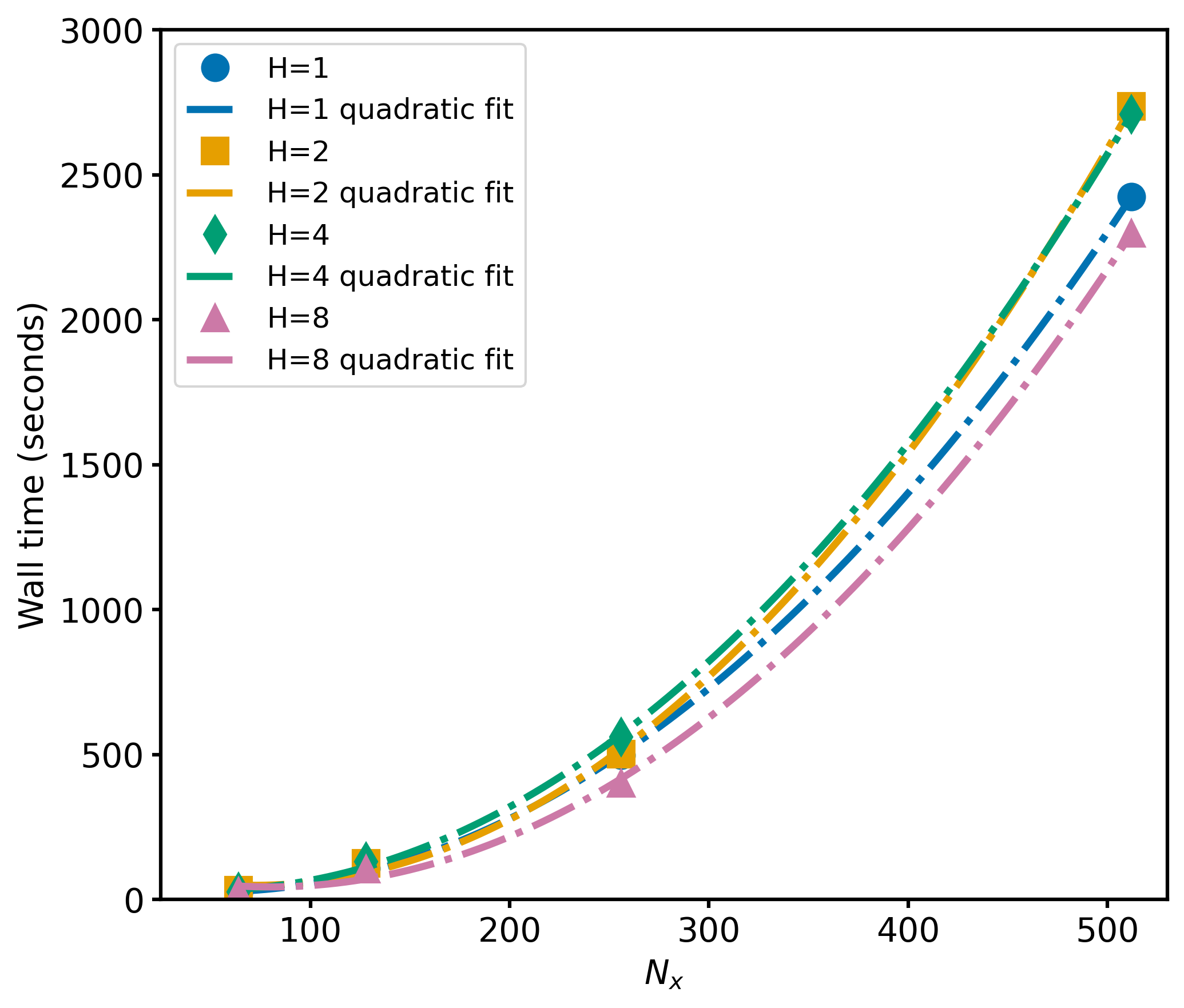}
        \caption{Full-rank method}
        \label{fig:full-rank}
    \end{subfigure}
    \caption{
Wall-clock runtime versus \(N_x=N_v\) for the two-stream instability with \(H=1,2,4,8\),  \(dt=0.01\), final time \(t=5\), and SVD tolerance \(\varepsilon_{\rm svd}=10^{-3}\). Times are measured in serial using one core of an Apple M4 Pro chip in C++. \textbf{(a)} shows the conservative adaptive rank method which shows approximately linear growth, and \textbf{(b)} shows the corresponding full-rank method which shows approximately quadratic growth. Thus the conservative correction preserves the computational advantage.
}
    \label{fig:walltime}
\end{figure}

The timing results show that the low-rank method provides a clear computational advantage over the full-rank method. More importantly, the conservative correction does not remove this advantage. Although the method enforces macroscopic conservation through the Fermi-Dirac reconstruction and the energy correction, the additional cost remains small compared with the cost of evolving the full phase-space solution.

Fig.~\ref{fig:rank_con_vs_noncon} compares the numerical rank of the conservative and non-conservative full-rank two-stream simulations. At each time, we compute the number of singular modes required to retain \(95\%\), \(99\%\), and \(99.99\%\) of the energy. The conservative correction consistently requires fewer modes, most notably at the \(99.99\%\) threshold, where the non-conservative solution requires substantially higher ranks over most of the simulation. This shows that the correction does not cause rank inflation; instead, it makes the full-rank solution more compressible and therefore more compatible with a low-rank representation.

We next test the dependence of the adaptive rank solution on the ACA-SVD truncation tolerance. The purpose of this study is to quantify the accuracy-rank tradeoff associated with the compressed representation. We use the two-stream instability with \(H=1\), \(N_x=N_v=512\),
\(x\in[0,4\pi]\), \(v\in[-2\pi,2\pi]\), CFL \(=10\), and final time \(t=25\). Table~\ref{tab:svd_cur_absolute_tolerance_study} shows the expected accuracy-rank tradeoff. Reducing $\varepsilon_{\rm SVD}$ from $10^{-2}$ to $10^{-4}$ decreases the sampled phase-space errors by roughly two orders of magnitude, while the maximum rank increases only from $18$ to $25$. Tightening the tolerance to $10^{-5}$ gives the reference run and increases the maximum rank to $45$. These results support the use of $\varepsilon_{\rm SVD}=10^{-3}$ as a practical default in the benchmark tests; it keeps the rank small while giving errors close to the tighter tolerances on the tested time interval. If the truncation tolerance is made sufficiently small, the adaptive rank representation may recover the full-rank behavior. In that regime, the same high-rank components that require filtering in the full-rank computation can reappear, and the error may increase unless the full-rank filtering strategy is also applied. Therefore, the table is intended to illustrate the meaningful low-rank accuracy-rank tradeoff before the computation effectively returns to the full-rank regime.

\begin{table}[htbp]
\centering
\caption{ACA-SVD tolerance sensitivity for the two-stream instability with
\(H=1\), \(N_x=N_v=512\), \(x\in[0,4\pi]\), \(v\in[-2\pi,2\pi]\),
CFL \(=10\), and final comparison time \(t=25\). For each row,
the CUR tolerance is chosen one order of magnitude smaller than the SVD
tolerance. Phase-space errors are computed relative to the run with
\((\varepsilon_{\rm svd},\varepsilon_{\rm cur})=(10^{-5},10^{-6})\).
The table reports the maximum absolute phase-space difference over the
sampled time interval \(t\le 25\) and the maximum adaptive rank.}
\label{tab:svd_cur_absolute_tolerance_study}
\begin{tabular}{c c c c}
\hline
\((\varepsilon_{\rm svd},\varepsilon_{\rm cur})\)
&
\(\max \| f-f_{\rm ref}\|_2\)
&
\(\max \|f-f_{\rm ref}\|_\infty\)
&
max Rank
\\
\hline
\((10^{-2},10^{-3})\) & \(1.1267\times 10^{-3}\) & \(1.3377\times 10^{-2}\) & \(18\) \\
\((10^{-3},10^{-4})\) & \(1.0177\times 10^{-4}\) & \(9.4892\times 10^{-4}\) & \(21\) \\
\((10^{-4},10^{-5})\) & \(6.6490\times 10^{-6}\) & \(1.2383\times 10^{-4}\) & \(25\) \\
\((10^{-5},10^{-6})\) & reference & reference & \(45\) \\
\hline
\end{tabular}
\end{table}

\begin{figure}
    \centering
    \includegraphics[width=0.65\linewidth]{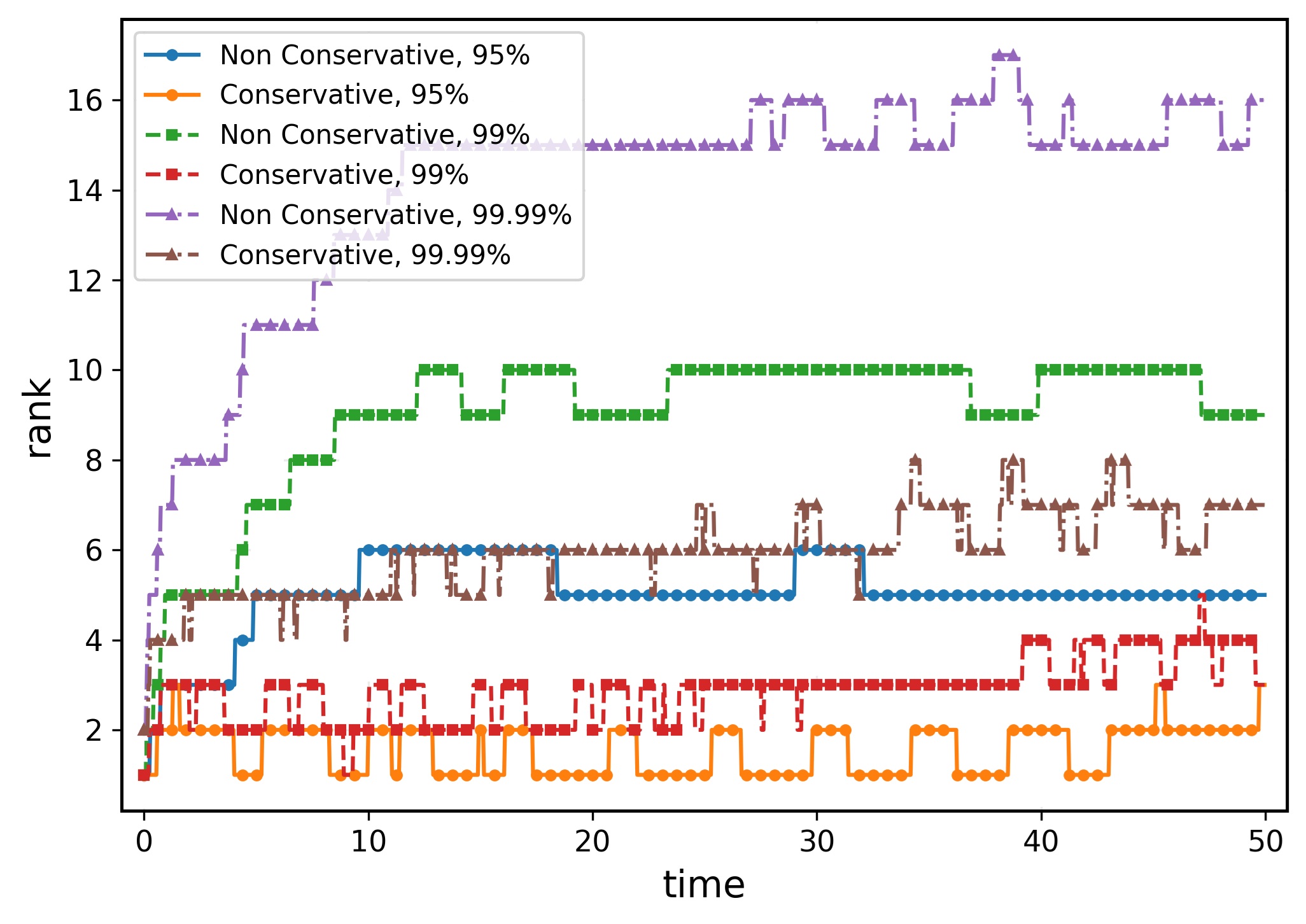}
    \caption{
Numerical-rank diagnostic for full-rank two-stream simulations with \(H=1\), \(N_x=N_v=512\), CFL \(=10\), and \(t=50\). The curves compare the conservative full-rank solver with the kinetic full-rank solver without conservative correction. The plotted ranks are the numbers of singular modes required to retain \(95\%\), \(99\%\), and \(99.99\%\) of the energy. The largest separation occurs for the $99.99\%$ energy threshold: the non-conservative curve climbs to roughly twice the conservative rank over much of the simulation. The $95\%$ and $99\%$ thresholds show the same direction but with smaller gaps. This supports the claim that the correction does not inflate rank; it makes the full-rank solution more compressible.
}
    \label{fig:rank_con_vs_noncon}
\end{figure}

Together, these results show that, in the 1D1V periodic setting tested here, the conservative correction is compatible with adaptive rank compression. The method maintains bounded ranks while preserving the specified global discrete invariants. These results are encouraging for higher-dimensional extensions, but such extensions will require a sampled, low-rank, or tensor representation of the correction term rather than the dense storage used in the present implementation.

\subsection{Two-stream instability}

We simulate the two-stream instability with the initial condition
\begin{equation}
    f(x,v,t=0)=\frac{v^2}{\sqrt{8\pi}}\left( 2+\cos \left( \frac{x}{2} \right) \right)\exp \left( -\frac{v^2}{2} \right)
\end{equation}
used in \cite{christlieb2025sampling}. 

In Fig.~\ref{fig:TSI_phase_H1}, the $H=1$ solution develops a sharply deformed central structure with fine horizontal oscillations. Fig.~\ref{fig:TSI_phase_H2} $H=2$ solution retains the same central deformation but with less fine-scale structure. Fig.~\ref{fig:TSI_phase_H8}, for $H=8$, is dominated by smoother horizontal bands, showing that larger H suppresses the more classical filamentary structure. This behavior is consistent with the role of $H$ as the dimensionless quantum parameter: smaller $H$ corresponds to a more classical, filamentary regime, while larger $H$ produces stronger quantum regularization \cite{christlieb2025sampling}.

\begin{figure}[htbp]
    \centering
    \begin{subfigure}{0.49\textwidth}
        \centering
        \includegraphics[width=\linewidth]{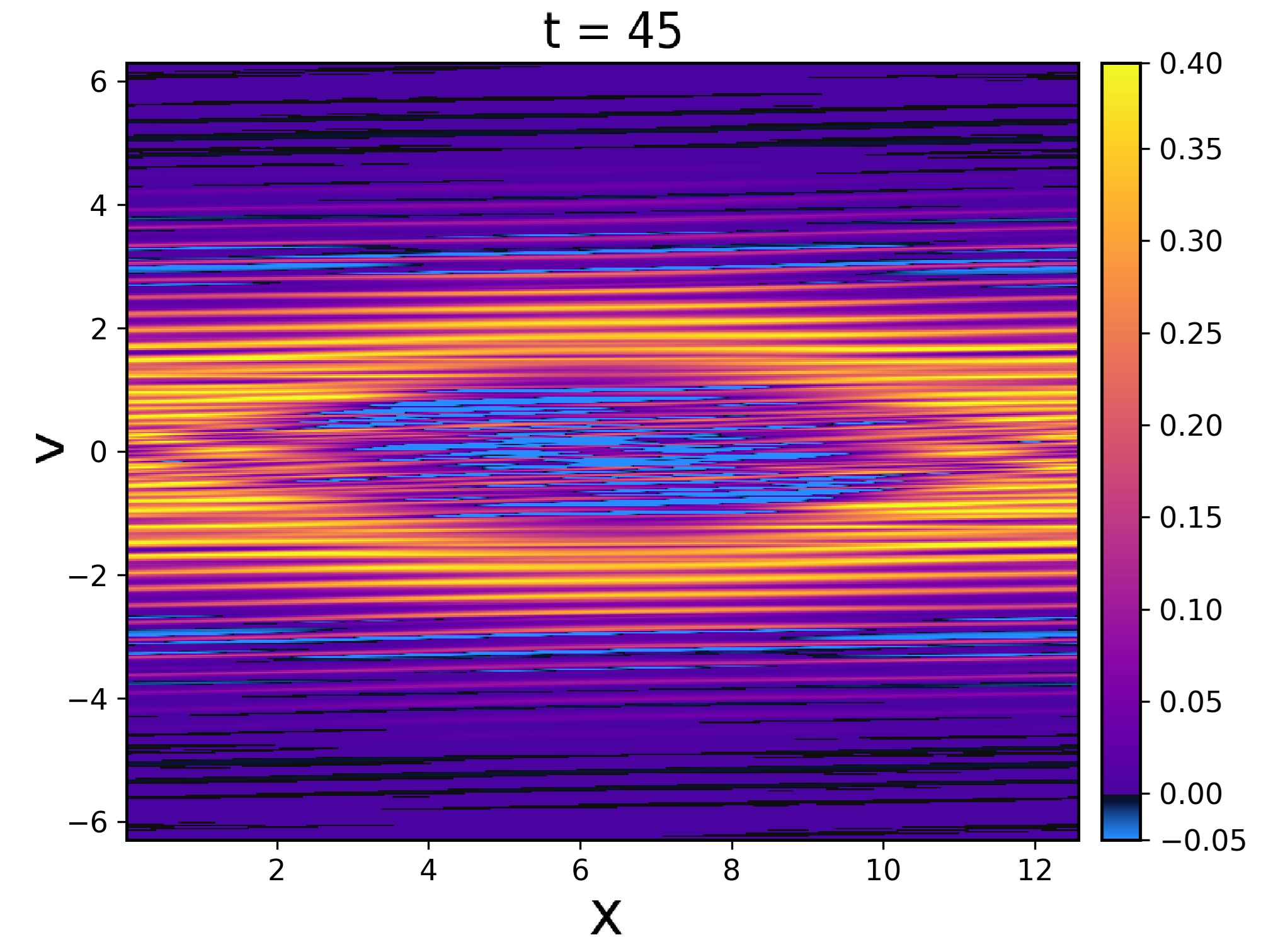}
        \caption{H=1}
        \label{fig:TSI_phase_H1}
    \end{subfigure}
    \hfill
    \begin{subfigure}{0.49\textwidth}
        \centering
        \includegraphics[width=\linewidth]{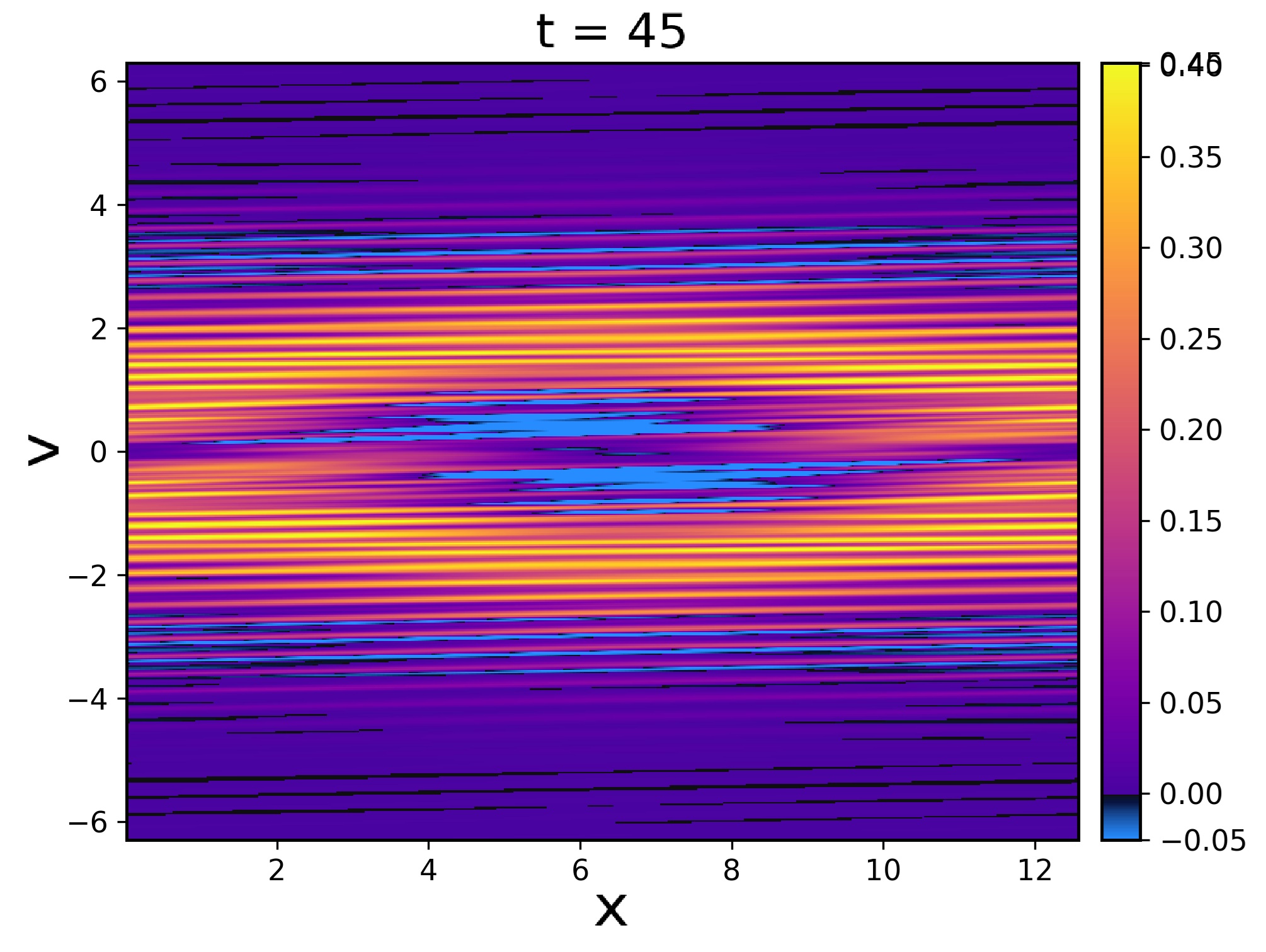}
        \caption{H=2}
        \label{fig:TSI_phase_H2}
    \end{subfigure}

    \vspace{1em} 

    \begin{subfigure}{0.49\textwidth}
        \centering
        \includegraphics[width=\linewidth]{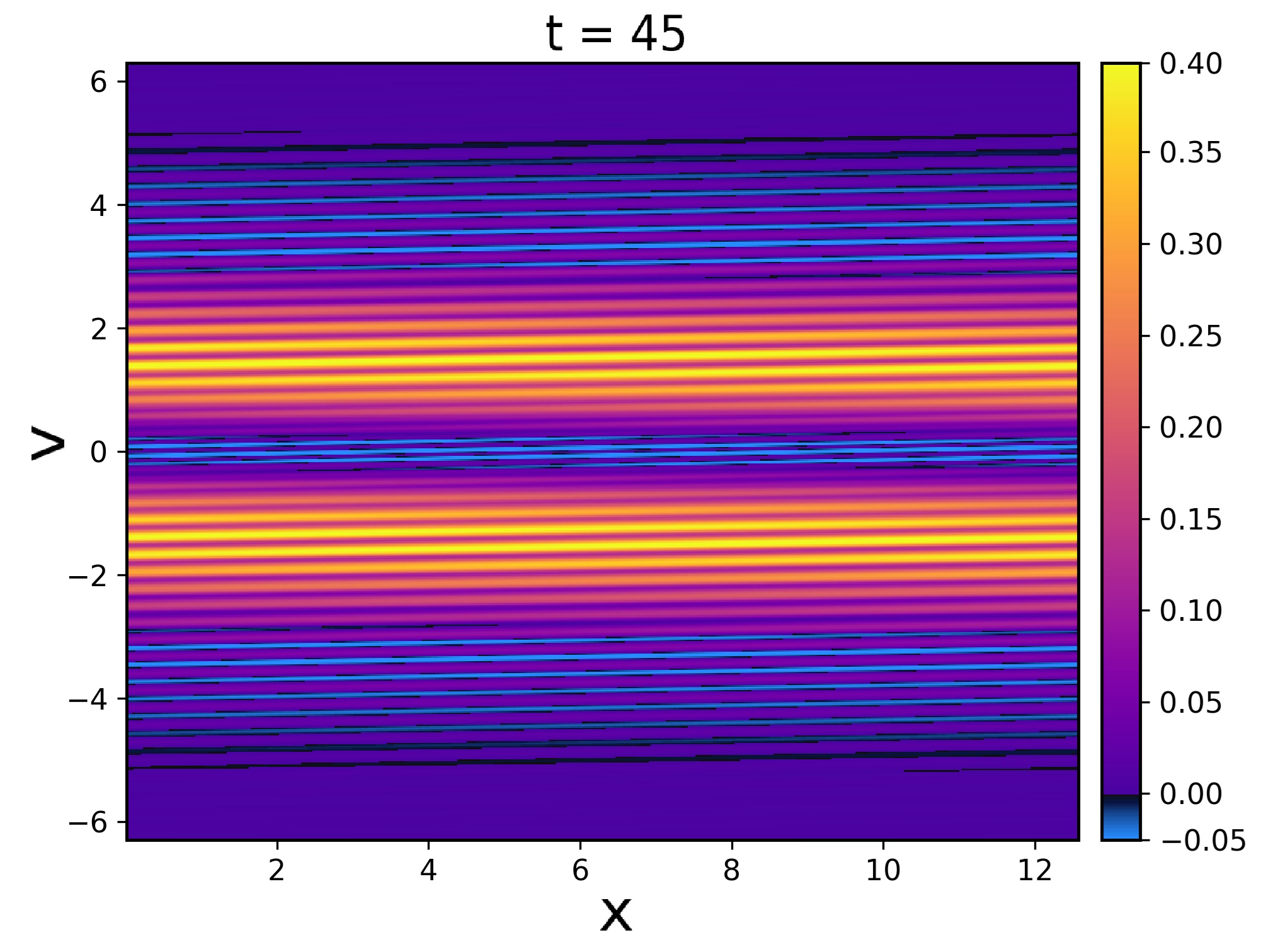}
        \caption{H=8}
        \label{fig:TSI_phase_H8}
    \end{subfigure}

    \caption{
Phase-space distribution \(f(x,v,t)\) for the two-stream instability at \(t=45\), with \(N_x=N_v=512\), CFL \(=10\), \(x\in[0,4\pi]\), and \(v\in[-2\pi,2\pi]\). Panels (a)--(c) show the conservative adaptive rank solution for \(H=1\), \(H=2\), and \(H=8\), respectively. Across the panels, increasing $H$ reduces the sharp central phase-space deformation and leaves smoother, more horizontally banded quantum-dispersive structure
}
    \label{fig:TSI_phase}
\end{figure}

The corresponding diagnostics are shown in Fig. \ref{fig:TSI_electric_energy} shows that the $H=1$ and $H=2$ cases undergo nonlinear growth followed by sustained oscillations, while $H=8$ case damps more strongly. Fig~ \ref{fig:TSI_rank_used} shows that the adaptive rank rises during the instability and then plateaus, with the largest retained rank for $H=1$. Figs.~\ref{fig:TSI_total_mass}-\ref{fig:TSI_total_energy} show that the mass, momentum and total energy deviations remain near machine precision throughout the run. Together, these results demonstrate that the proposed Fermi-Dirac-type correction and global moment correction successfully restore the macroscopic invariants after the non-conservative low-rank kinetic update.

\begin{figure}[htbp]
    \centering

    \begin{subfigure}{0.42\textwidth}
        \centering
        \includegraphics[width=\linewidth]{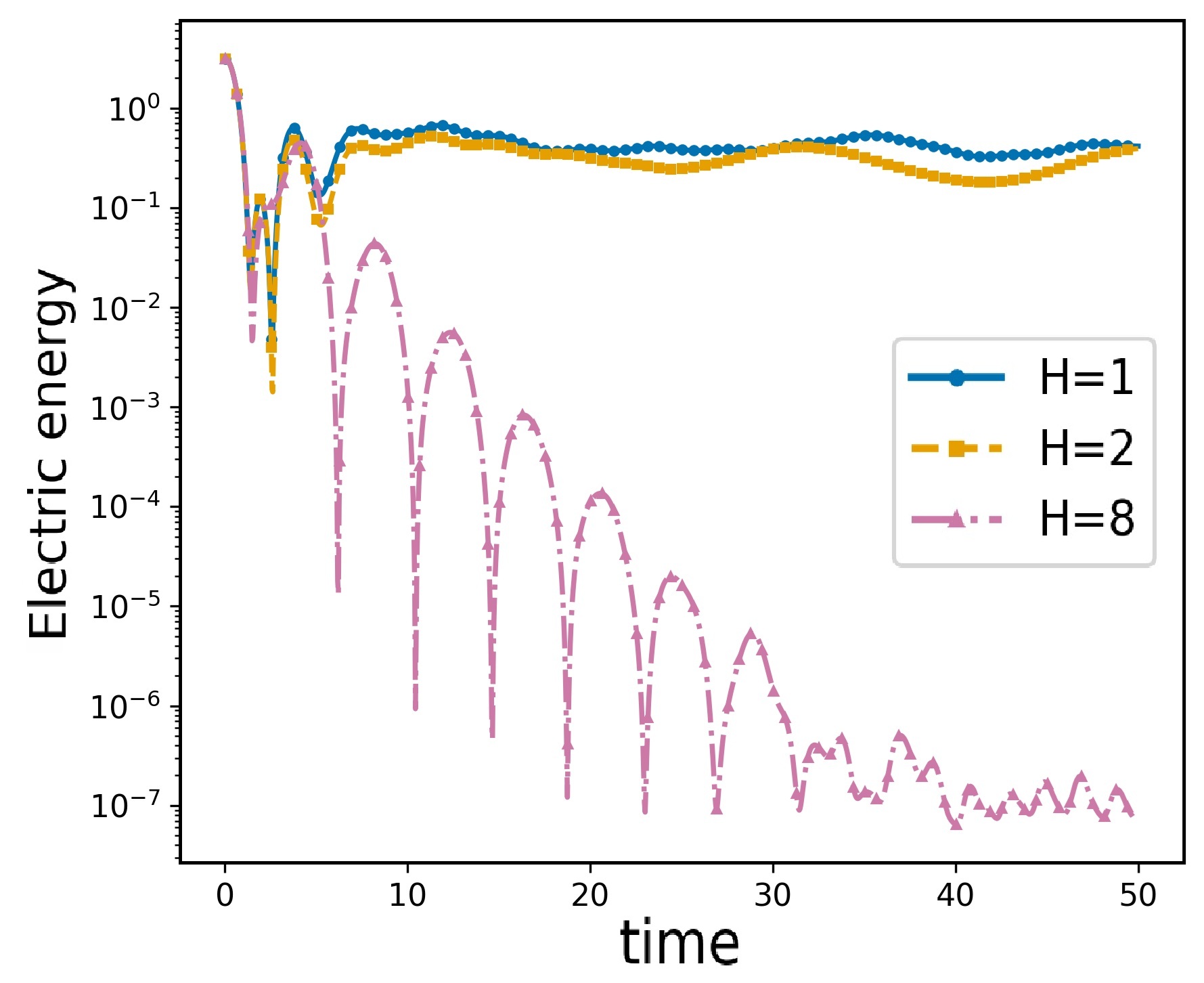}
        \caption{}
        \label{fig:TSI_electric_energy}
    \end{subfigure}
    \hfill
    \begin{subfigure}{0.42\textwidth}
        \centering
        \includegraphics[width=\linewidth]{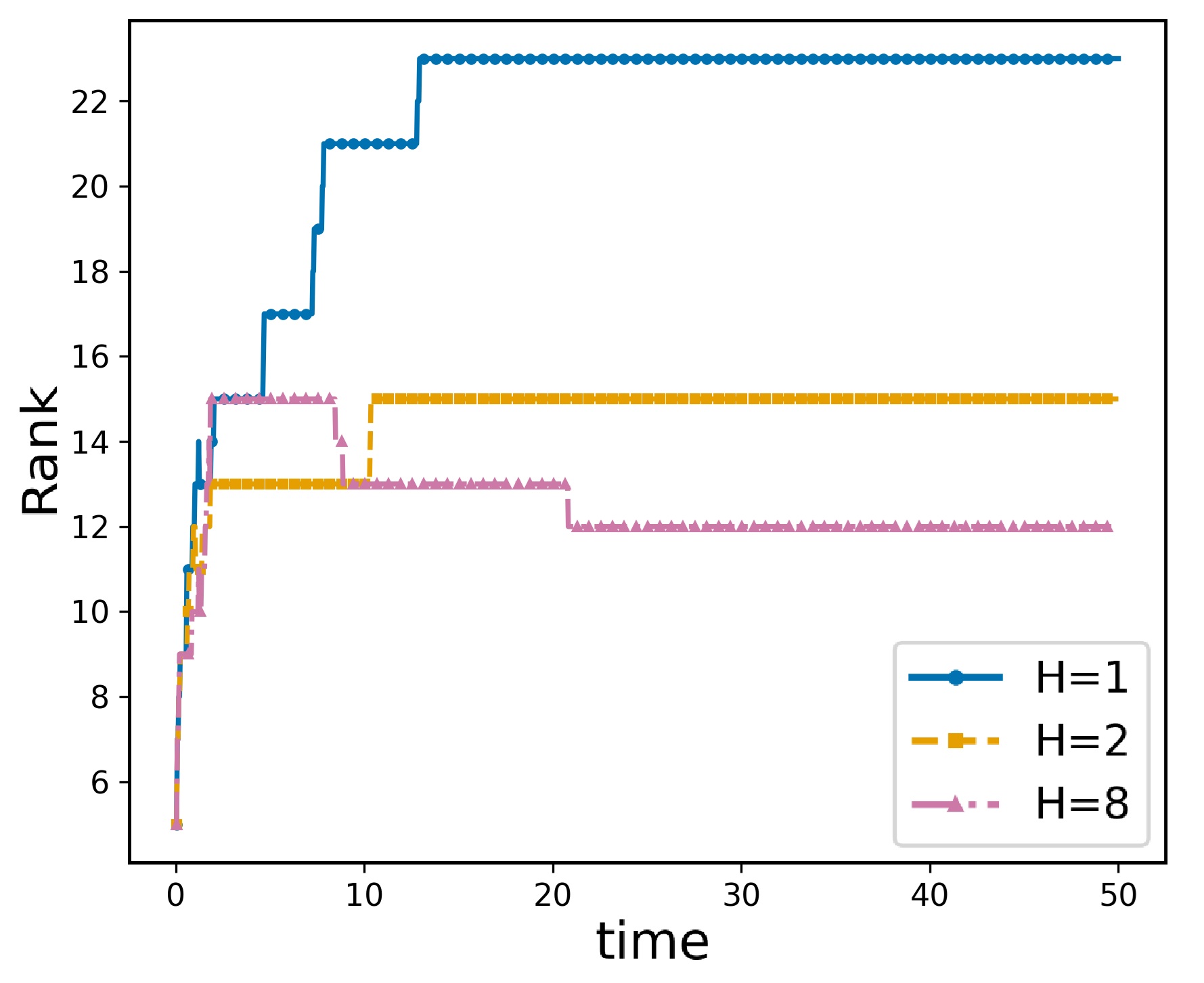}
        \caption{}
        \label{fig:TSI_rank_used}
    \end{subfigure}

    \vspace{1em} 

    \begin{subfigure}{0.42\textwidth}
        \centering
        \includegraphics[width=\linewidth]{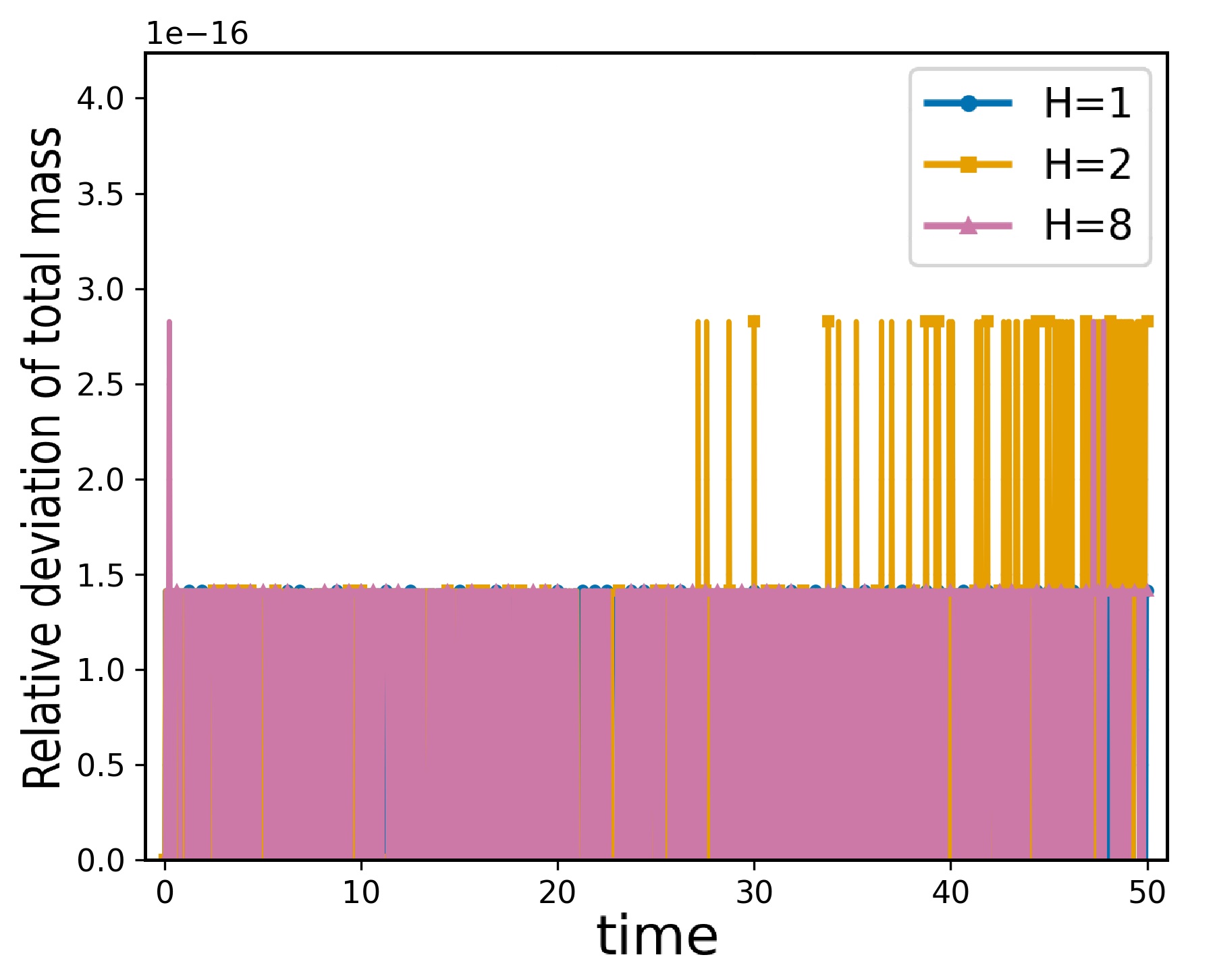}
        \caption{}
        \label{fig:TSI_total_mass}
    \end{subfigure}
    \hfill
    \begin{subfigure}{0.42\textwidth}
        \centering
        \includegraphics[width=\linewidth]{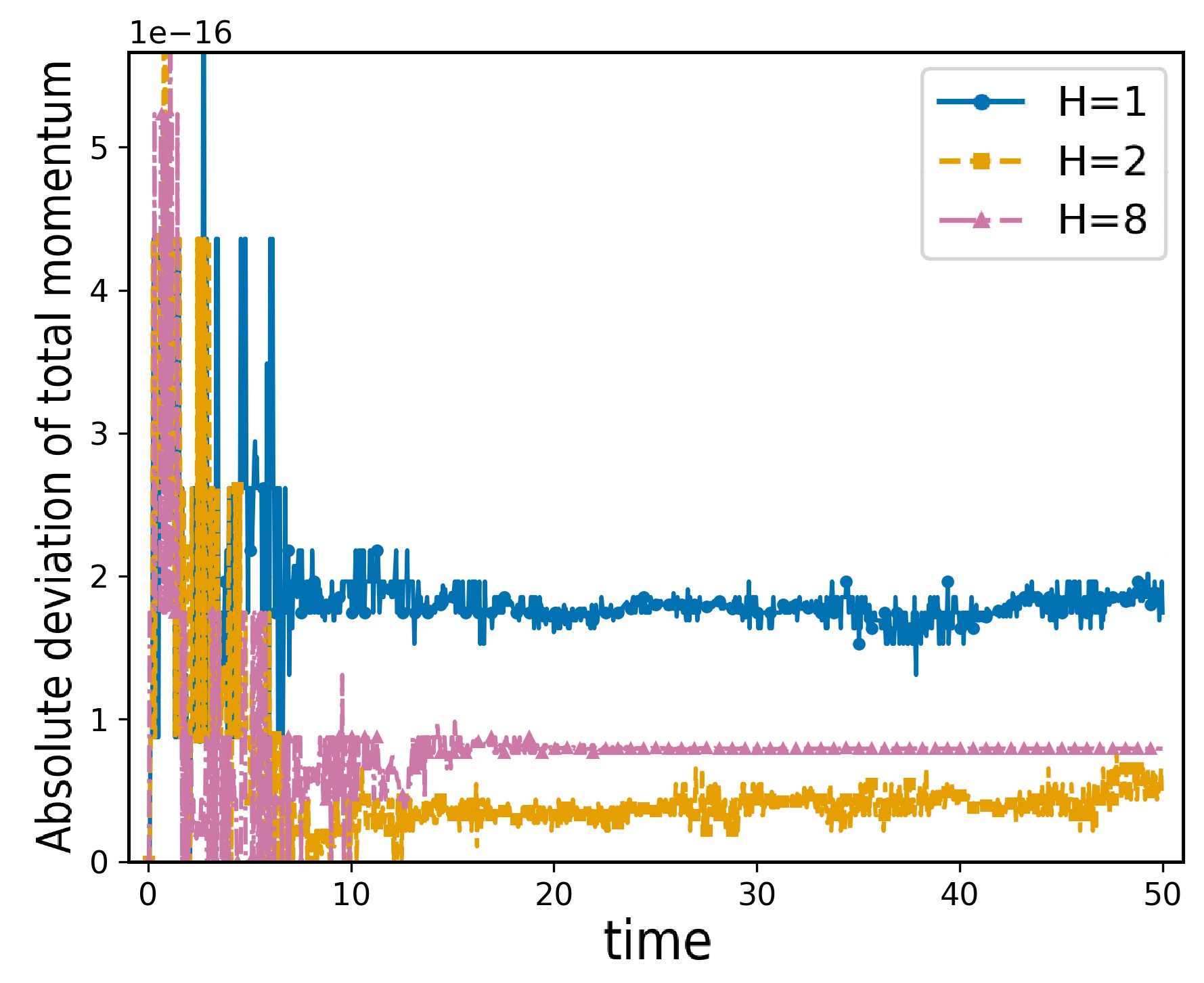}
        \caption{}
        \label{fig:TSI_total_momentum}
    \end{subfigure}

    \vspace{1em} 

    \begin{subfigure}{0.42\textwidth}
        \centering
        \includegraphics[width=\linewidth]{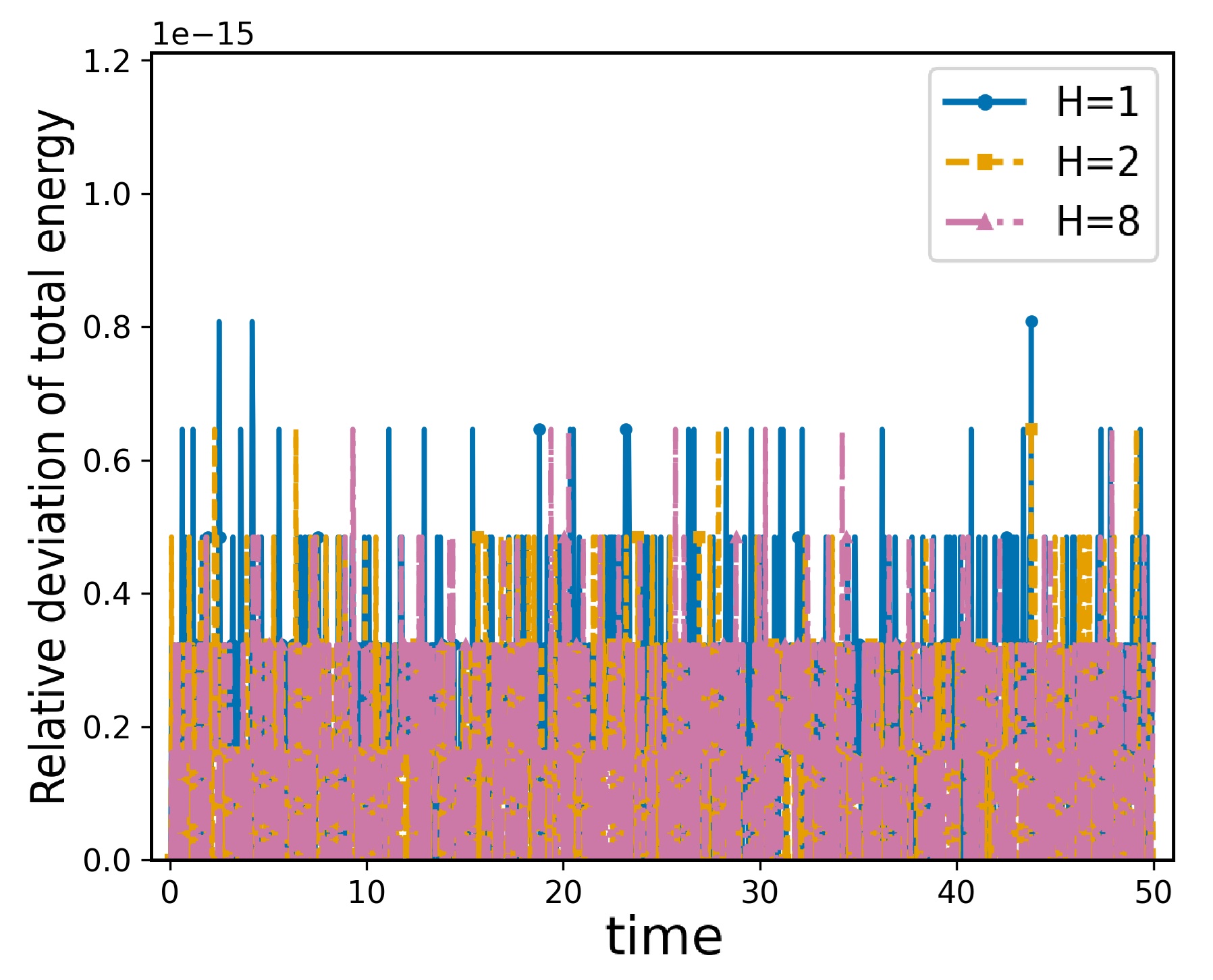}
        \caption{}
        \label{fig:TSI_total_energy}
    \end{subfigure}

    \caption{
Diagnostics for the two-stream instability with \(N_x=N_v=512\), CFL \(=10\), \(x\in[0,4\pi]\), \(v\in[-2\pi,2\pi]\), and \(H=1,2,8\). \textbf{(a)} shows the electric-field energy \(U(t)=\frac12\Delta x\sum_i E_i(t)^2\). \textbf{(b)} shows the adaptive rank retained after ACA-SVD recompression. \textbf{(c)}--\textbf{(e)} show the global conservation diagnostics computed from the final corrected low-rank state: relative mass deviation \(|M(t)-M(0)|/|M(0)|\), absolute deviation of total momentum \(|P(t) -P(0)|\), and relative total-energy deviation \(|\mathcal E_{tot}(t)-\mathcal E_{tot}(0)|/|\mathcal E_{tot}(0)|\), where \(\mathcal E_{tot}(t)=K(t)+U(t)\) and \(K(t)=\frac12\Delta x\Delta v\sum_{i,j}w_j v_j^2 f_{i,j}(t)\).
}\label{fig:TSI_diagnostics}
\end{figure}

This test is important because the two-stream instability produces strong nonlinear deformation in phase-space. The results show that the method can enforce conservation without  suppressing the physical instability and without requiring a full-rank representation.

\begin{table}[t]
\centering
\caption{CFL sensitivity study for the two-stream instability with $H=1$, $N_x=512$, and final time $t=50$. The CFL $=1$ solution is used as the reference. Here $N_t$ denotes the number of time steps, $\mathcal E_E$ is the electric energy, and $\Delta\mathcal E$ is the relative deviation of the total discrete energy. CFL $=100$ is included as a large-time-step stress test.}
\label{tab:cfl_sensitivity}
\begin{tabular}{c c c c c}
\hline
CFL & $N_t$ &
$\max_t |\mathcal E_E-\mathcal E_E^{\rm ref}|$ &
$\max_t \|f-f^{\rm ref}\|_2/\|f^{\rm ref}\|_2$ &
$\max_t \Delta \mathcal E$ \\
\hline
2   & 640 & $6.416\times 10^{-3}$ & $4.979\times 10^{-3}$ & $8.078\times 10^{-16}$ \\
5   & 256 & $2.533\times 10^{-2}$ & $1.972\times 10^{-2}$ & $6.462\times 10^{-16}$ \\
10  & 128 & $5.348\times 10^{-2}$ & $4.301\times 10^{-2}$ & $8.078\times 10^{-16}$ \\
50  & 26  & $1.880\times 10^{-1}$ & $1.660\times 10^{-1}$ & $4.847\times 10^{-16}$ \\
100 & 13  & $6.747\times 10^{-1}$ & $5.138\times 10^{-1}$ & $3.231\times 10^{-16}$ \\
\hline
\end{tabular}\label{tb: CFL}
\end{table}

Table \ref{tab:cfl_sensitivity} shows that the global moment correction makes the total energy diagnostic essentially insensitive to the time step, with energy errors remaining machine precision for all test CFL values. The electric energy and distribution function errors relative to the CFL$=1$ reference increase as the CFL grows. In particular, CFL$=100$ is included as a stress test.

\subsection{Strong Landau Damping}

We simulate strong Landau damping with the initial condition 
\begin{equation}
    f(x,v,t=0)=\frac{1}{\sqrt{2\pi}}\exp \left( -\frac{v^2}{2} \right) \left( 1+0.2 \cos(0.4x) \right)
\end{equation}
The domain is $x\in \left[ 0, \frac{2\pi}{0.4}\right],\; v\in[-2\pi, 2\pi]$ from \cite{christlieb2025sampling}

Fig.~\ref{fig:SLD_phase} shows that the strong Landau damping solution remains organized around a central velocity band but develops $H$-dependent fine-scale structure. Fig.~\ref{fig:SLD_phase_H1}, $H = 1$, contains the most visible velocity-space striations; Fig.~\ref{fig:SLD_phase_H2}, $H = 2$, is smoother; and Fig.~\ref{fig:SLD_phase_H8}, $H = 8$, has the smoothest banded structure. This trend is consistent with stronger quantum regularization as H increases.

\begin{figure}[htbp]
    \centering

    \begin{subfigure}{0.49\textwidth}
        \centering
        \includegraphics[width=\linewidth]{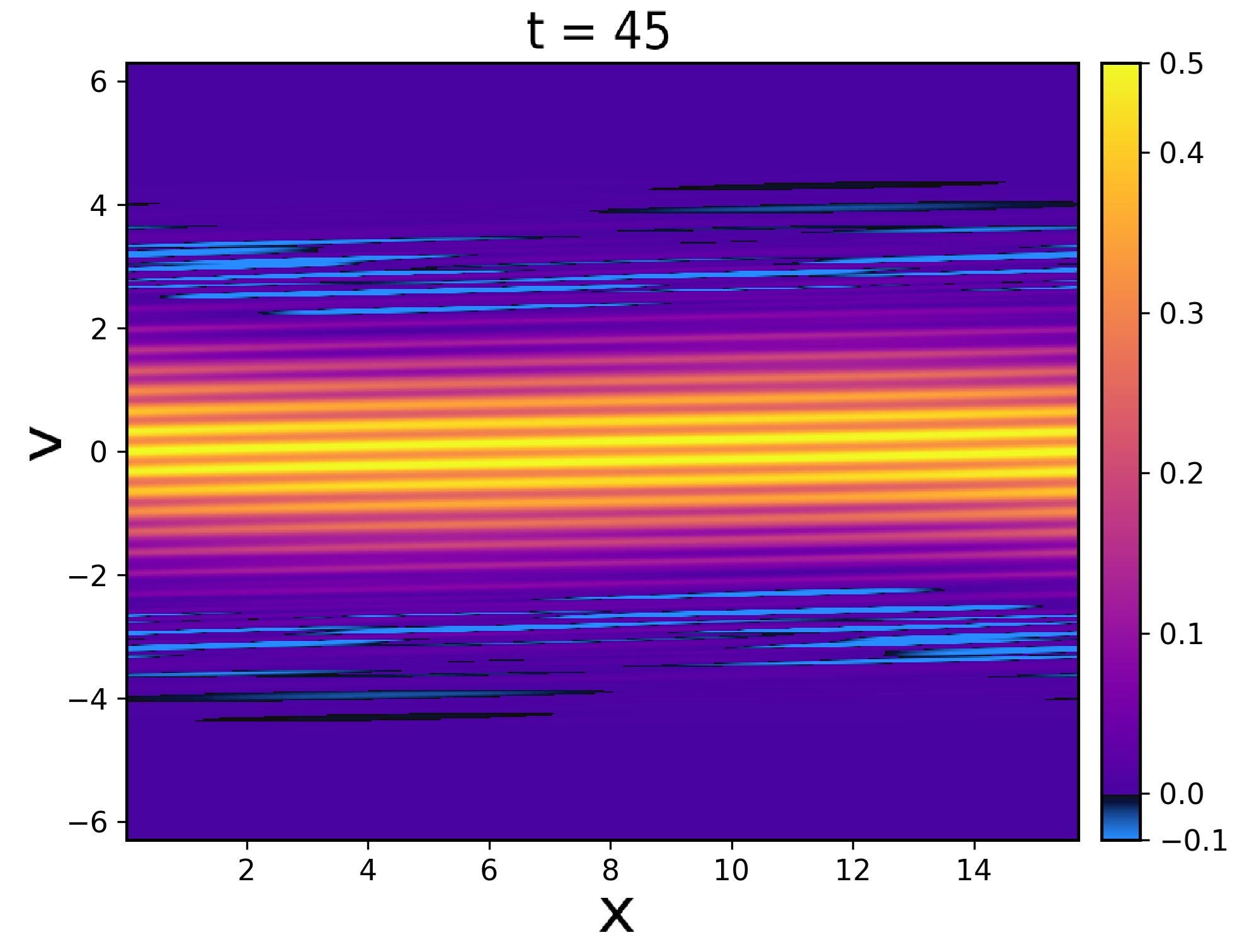}
        \caption{H=1}
        \label{fig:SLD_phase_H1}
    \end{subfigure}
    \hfill
    \begin{subfigure}{0.49\textwidth}
        \centering
        \includegraphics[width=\linewidth]{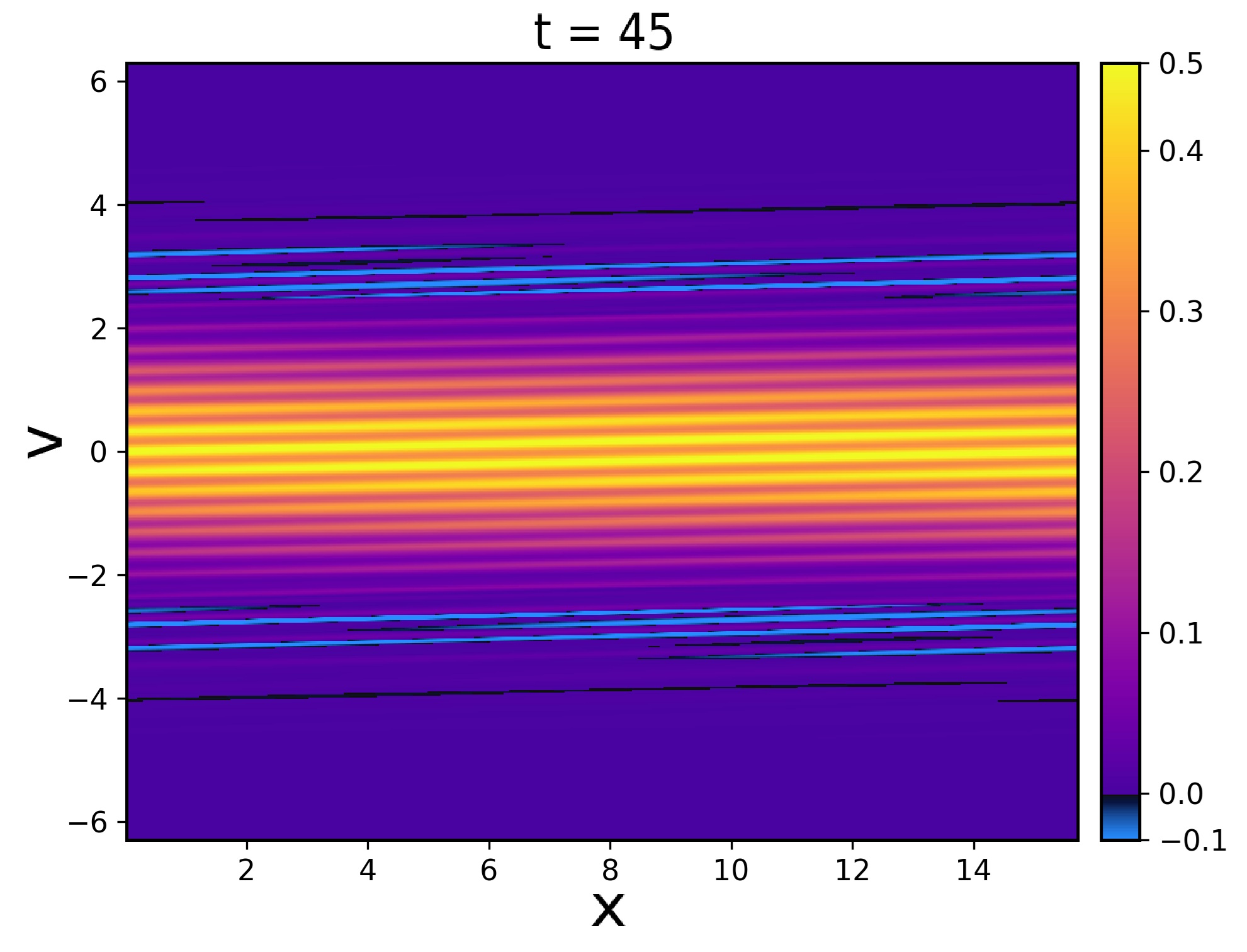}
        \caption{H=2}
        \label{fig:SLD_phase_H2}
    \end{subfigure}

    \vspace{1em} 

    \begin{subfigure}{0.49\textwidth}
        \centering
        \includegraphics[width=\linewidth]{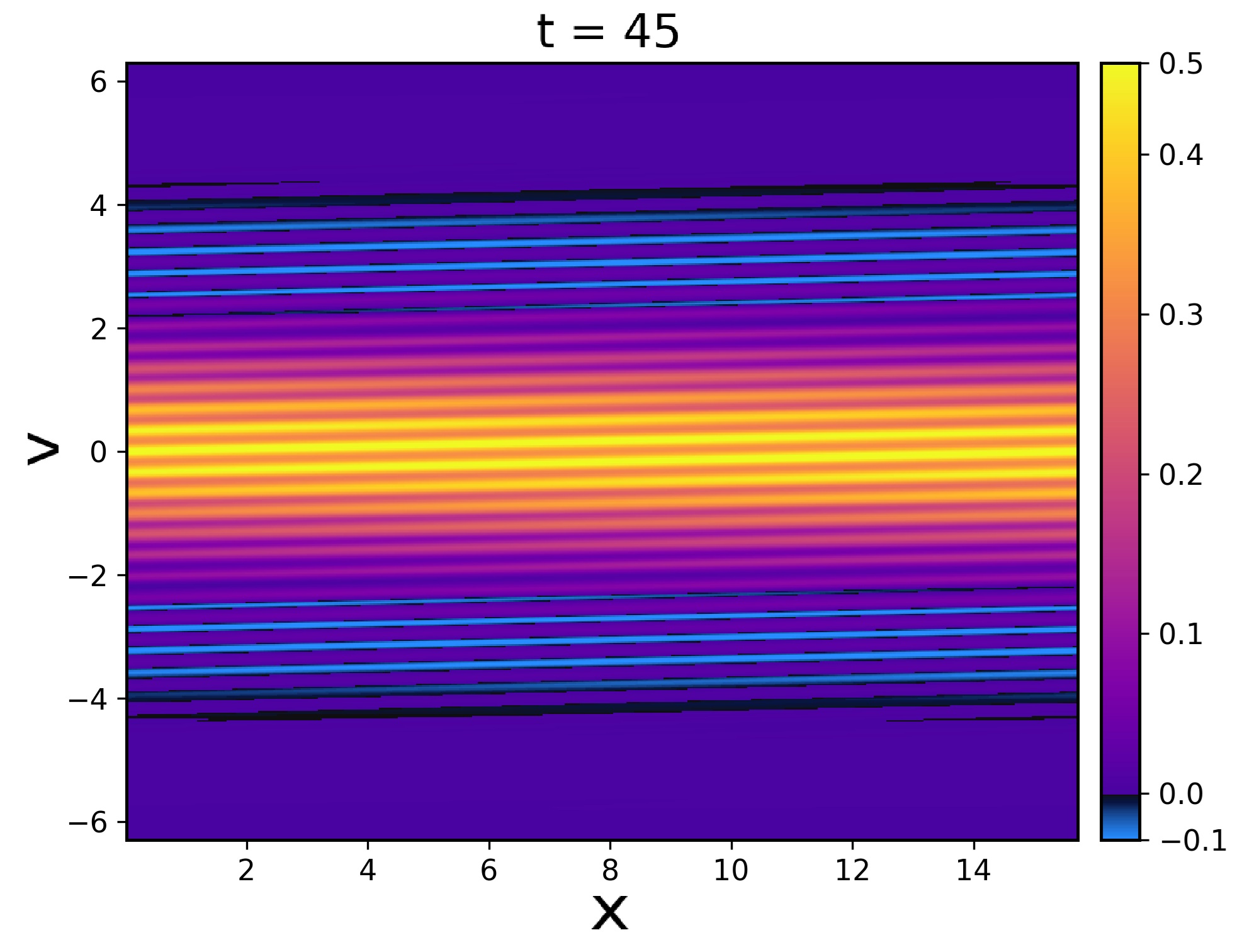}
        \caption{H=8}
        \label{fig:SLD_phase_H8}
    \end{subfigure}

    \caption{
Phase-space distribution \(f(x,v,t)\) for strong Landau damping at \(t=45\), with \(N_x=N_v=512\), CFL \(=10\), \(x\in[0,2\pi/0.4]\), and \(v\in[-2\pi,2\pi]\). Panels (a)--(c) show the conservative adaptive rank solution for \(H=1\), \(H=2\), and \(H=8\), respectively. As $H$ increases, the fine velocity-space filamentation is reduced and the phase-space bands become smoother, illustrating the regularizing effect of the quantum parameter in this damping-dominated benchmark.
}
    \label{fig:SLD_phase}
\end{figure}

As an internal consistency check, Fig~\ref{fig:SLD_phase2} compares the present density-momentum correction with an alternative global-moment correction implemented using the same numerical parameters \cite{christlieb2026structurepreservingadaptiverankapproachhighdimensional}. This comparison only checks that, for this periodic damping test, the present correction produces a phase-space solution comparable to an alternative conservative correction.

\begin{figure}[htbp]
    \centering

    \begin{subfigure}{0.49\textwidth}
        \centering
        \includegraphics[width=\linewidth]{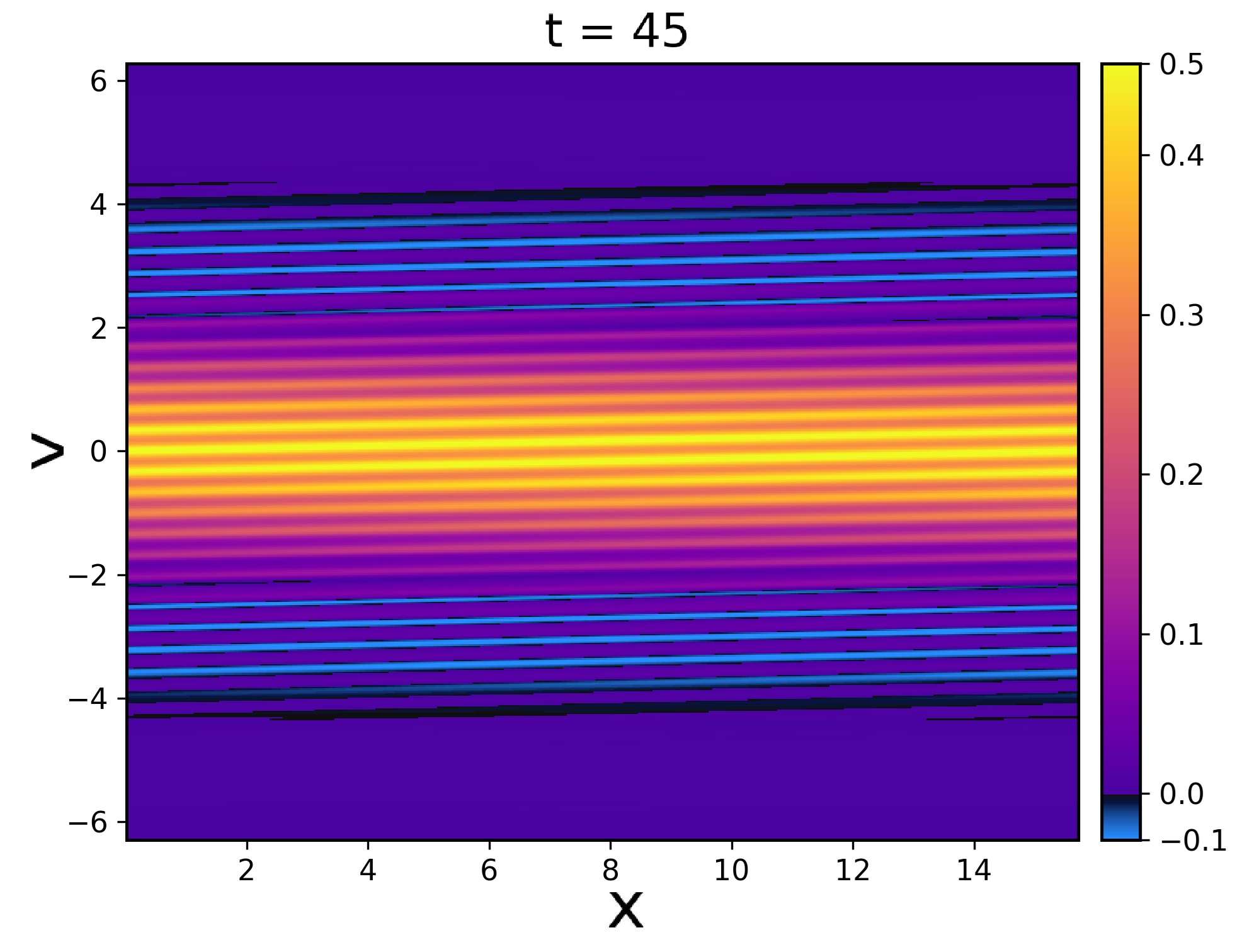}
        \caption{Alternative correction.}
        \label{fig:SLD_phase_globalH8}
    \end{subfigure}
    \hfill
    \begin{subfigure}{0.49\textwidth}
        \centering
        \includegraphics[width=\linewidth]{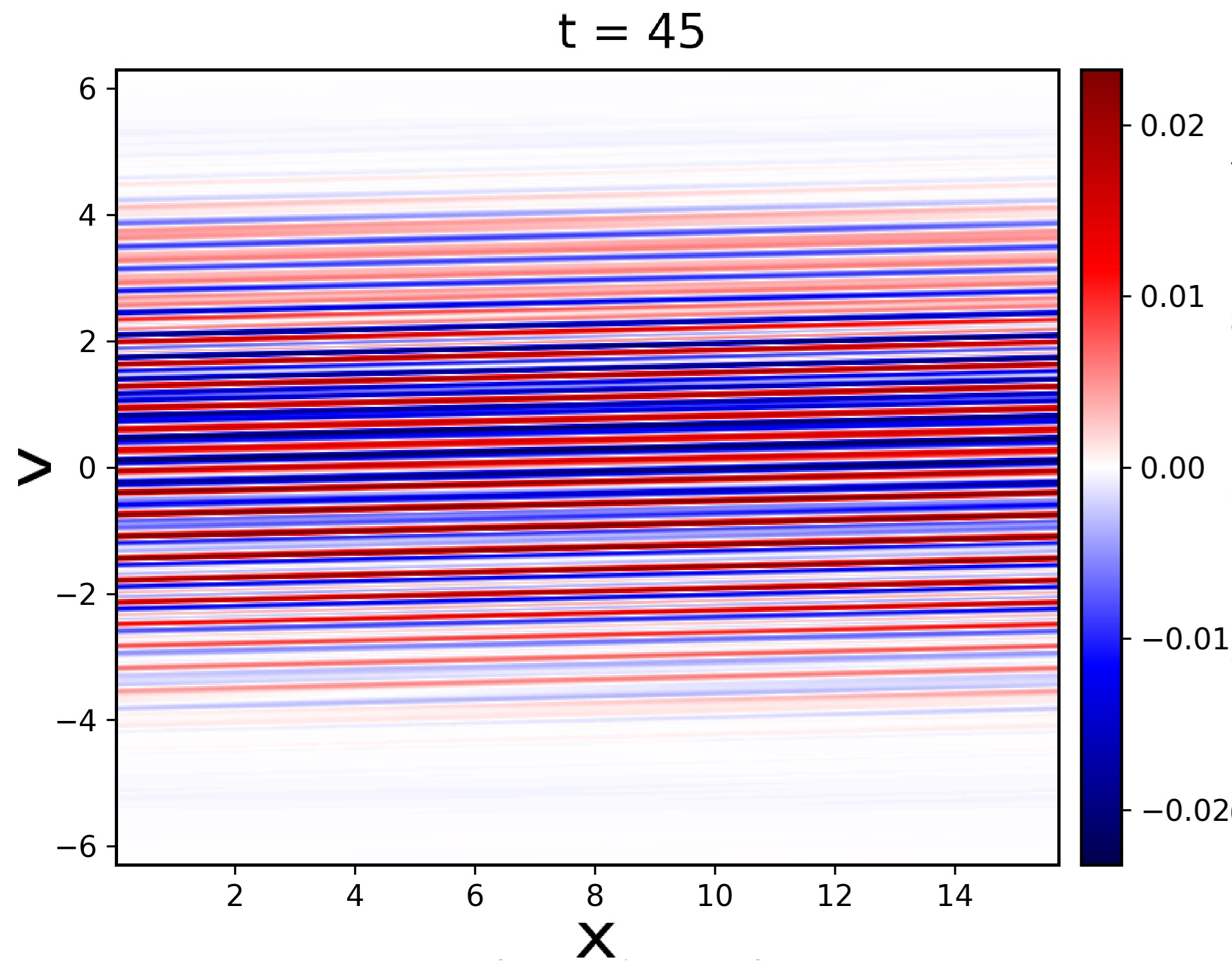}
        \caption{Difference.}
        \label{fig:SLD_phase_diffH8}
    \end{subfigure}

    \caption{
    Internal comparison for strong Landau damping at $t=45$ with $H=8$, $N_x = N_v=512$, CFL $=10$, $x\in[0,\frac{2\pi}{0.4}]$, and $v\in[-2\pi, 2\pi]$. \textbf{(a)} shows the phase-space distribution $f(x,v,t)$ computed using the globally conservative formulation in \cite{christlieb2026structurepreservingadaptiverankapproachhighdimensional}. \textbf{(b)} shows the pointwise difference between this internal reference and the present density-momentum correction. The difference is small relative to the scale of $f$ and is concentrated along the banded phase-space structures. This comparison is used only as an internal consistency check.}
    \label{fig:SLD_phase2}
\end{figure}

Fig. \ref{fig:SLD_diagnostics} reports the electric energy, adaptive rank, and conservation diagnostics. The electric energy shows the expected damping and nonlinear oscillatory behavior. The adaptive rank increases during the initial transient and then remains bounded, indicating that the low-rank representation adapts to the solution complexity without uncontrolled rank growth.

The rank behavior in the strong Landau damping test is slightly less monotone than in the two-stream and bump-on-tail tests. In those two benchmarks, stronger phase-space deformation dominates the rank behavior, and increasing $H$ gives smoother dynamics and lower adaptive ranks. In Fig. \ref{fig:SLD_diagnostics}~(b), however, the larger-$H$ case maintains a mildly higher rank. We do not interpret this as contradicting the regularizing role of $H$. The SLD solution remains relatively low-rank, so the measured rank is more sensitive to secondary effects such as the conservative correction and ACA-SVD recompression. In particular the Fermi-Dirac parameter
\[
\alpha = \frac{1}{\pi H\lambda_D n_0}
\]
decreases as $H$ increases, so the correction moves farther from the Maxwell-Boltzmann limit and may be less separable after coupling with $\rho(x), \; u(x), \; \mathcal T(x)$. Thus, for this low-rank benchmark, the adaptive rank reflects the combined influence of $H$, the conservative Fermi-Dirac correction, and the truncation tolerance.

\begin{figure}[htbp]
    \centering

    \begin{subfigure}{0.42\textwidth}
        \centering
        \includegraphics[width=\linewidth]{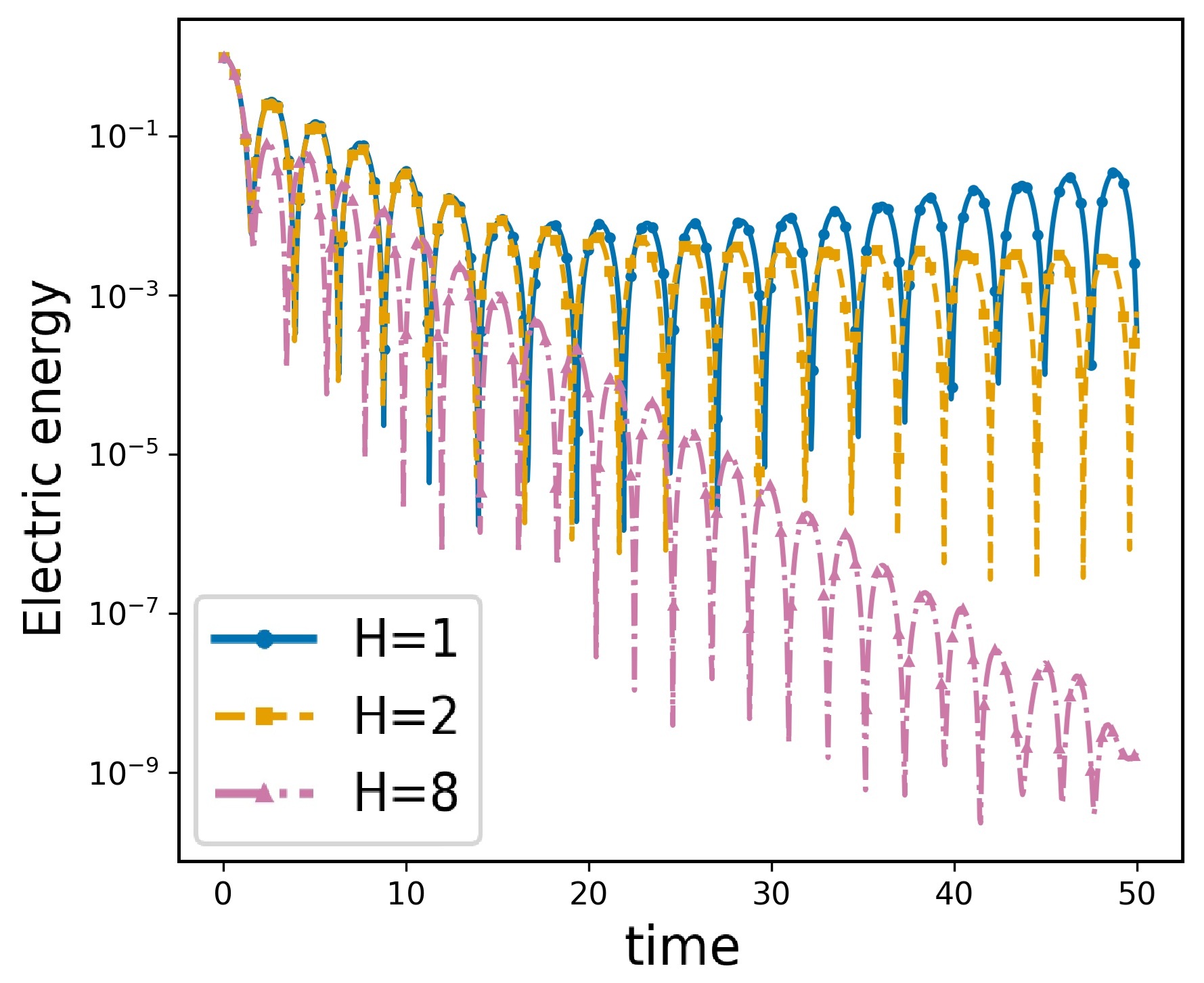}
        \caption{}
        \label{fig:SLD_electric_energy}
    \end{subfigure}
    \hfill
    \begin{subfigure}{0.42\textwidth}
        \centering
        \includegraphics[width=\linewidth]{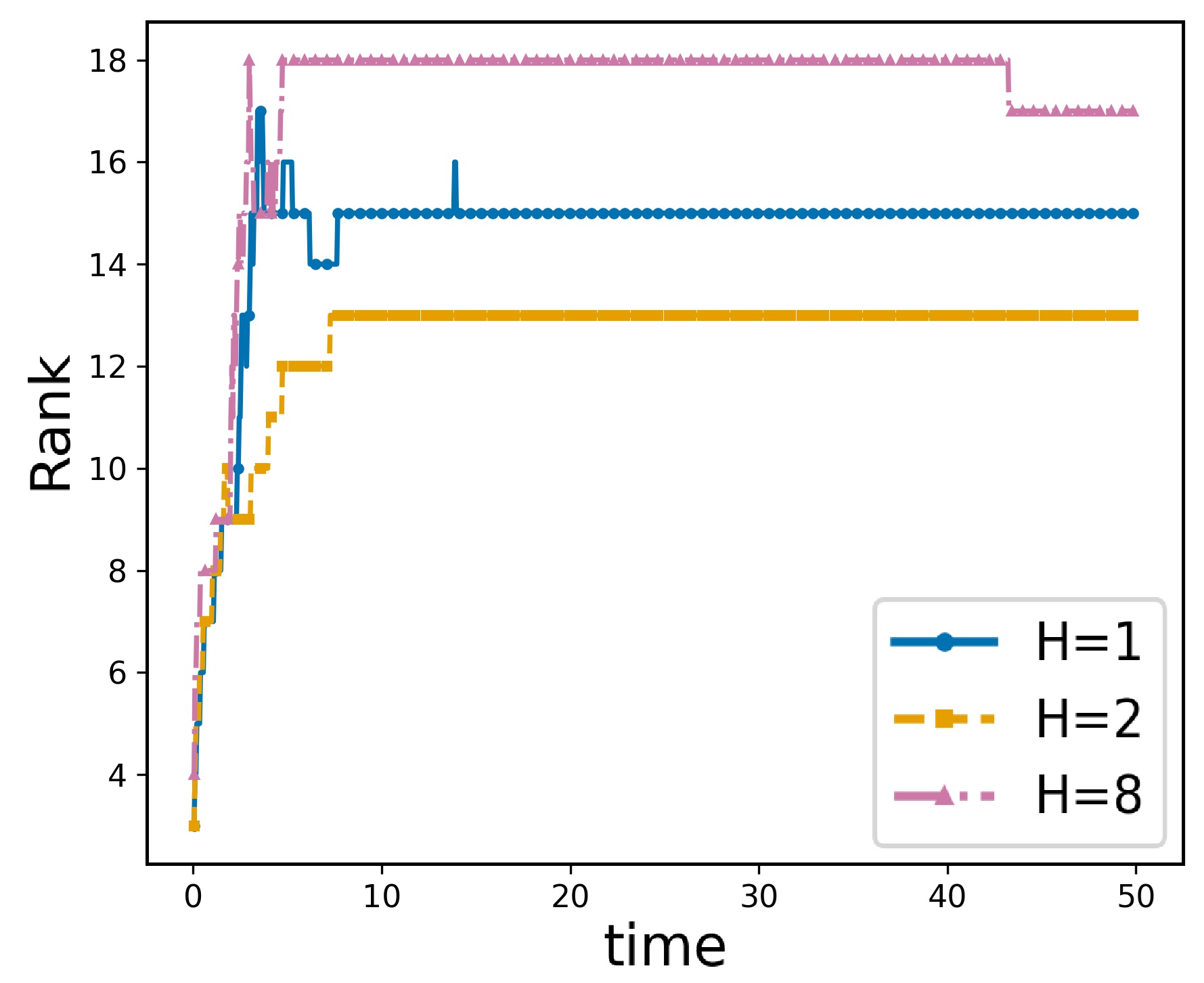}
        \caption{}
        \label{fig:SLD_rank_used}
    \end{subfigure}

    \vspace{1em} 

    \begin{subfigure}{0.42\textwidth}
        \centering
        \includegraphics[width=\linewidth]{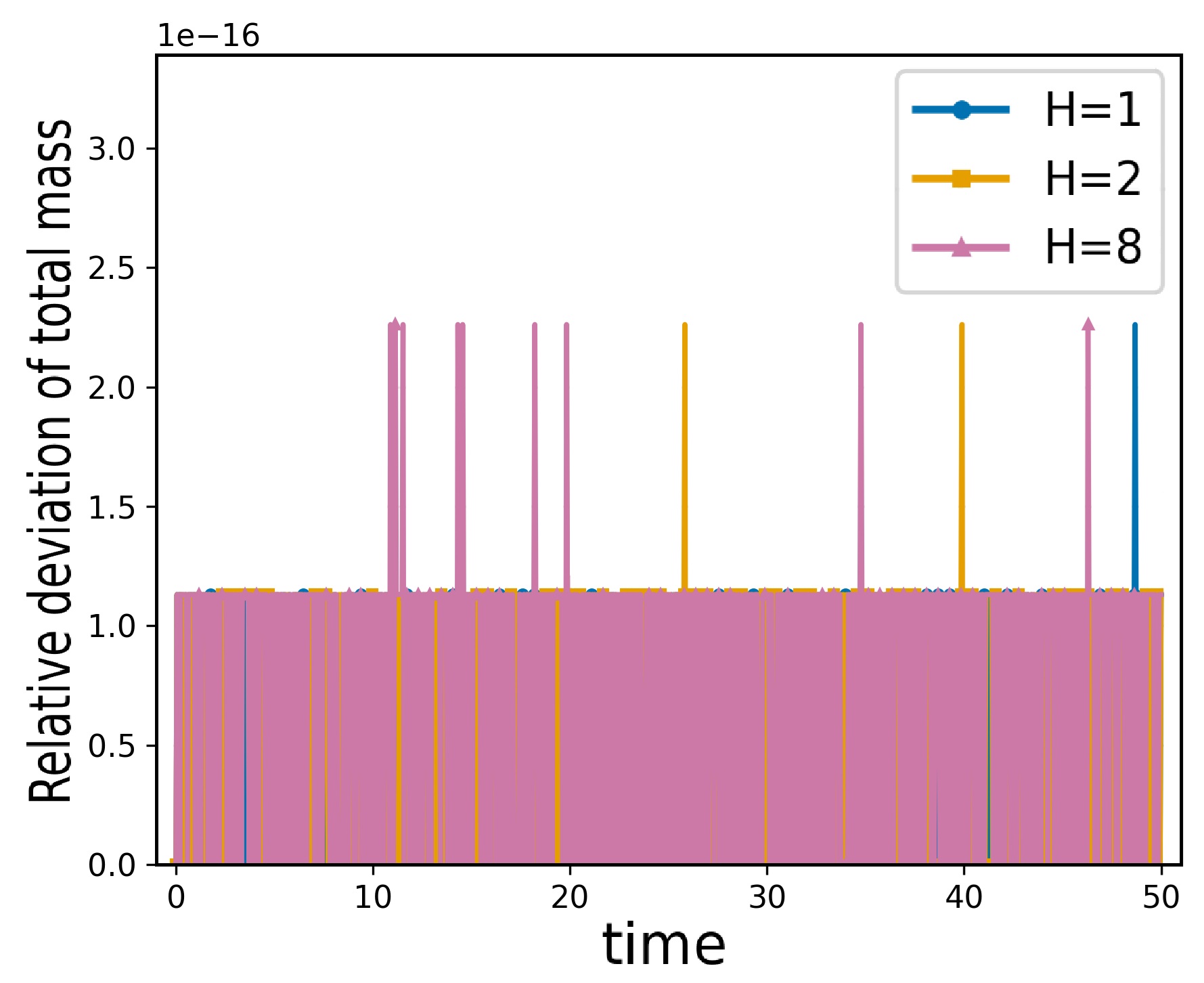}
        \caption{}
        \label{fig:SLD_total_mass}
    \end{subfigure}
    \hfill
    \begin{subfigure}{0.42\textwidth}
        \centering
        \includegraphics[width=\linewidth]{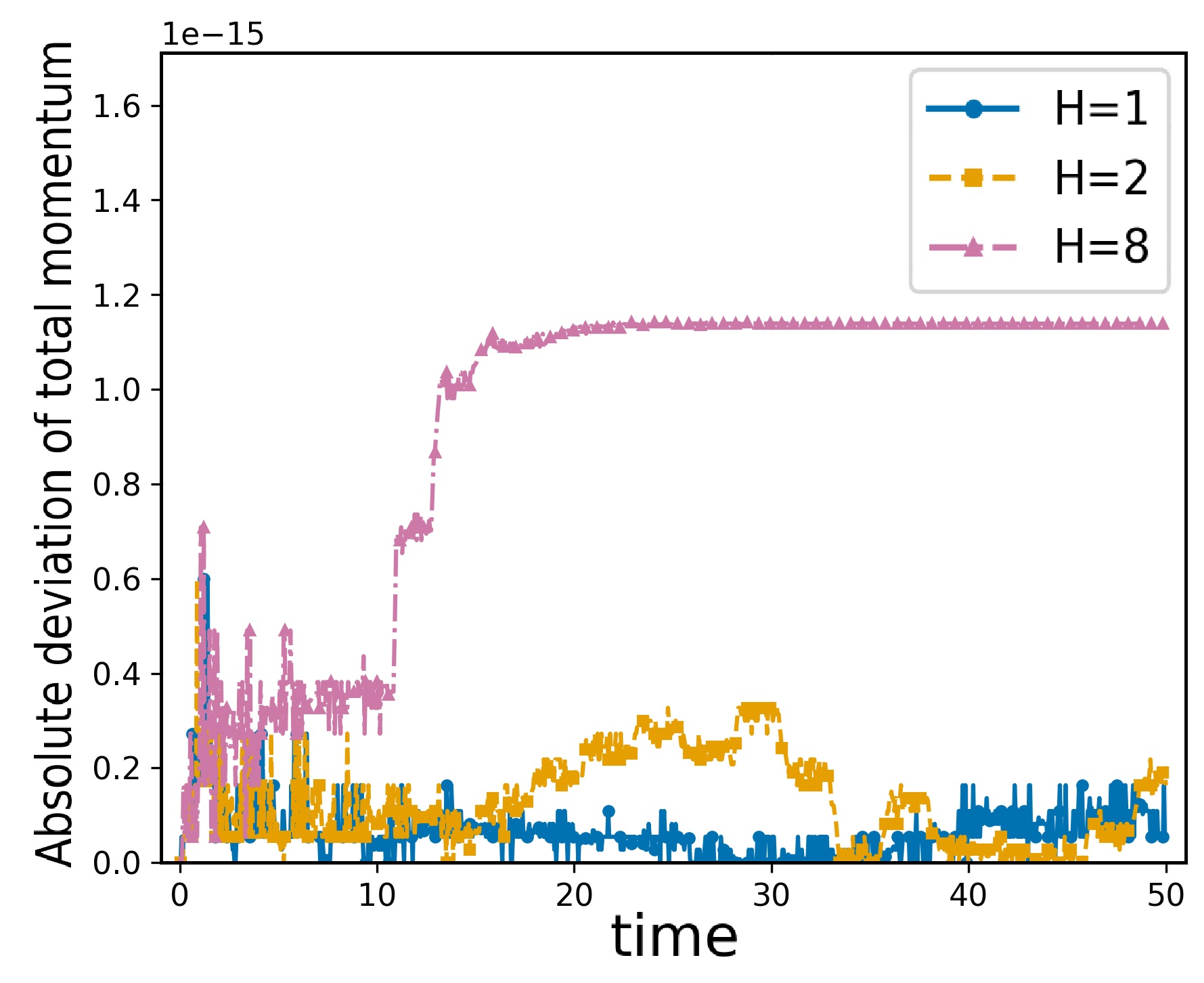}
        \caption{}
        \label{fig:SLD_total_momentum}
    \end{subfigure}

    \vspace{1em} 

    \begin{subfigure}{0.42\textwidth}
        \centering
        \includegraphics[width=\linewidth]{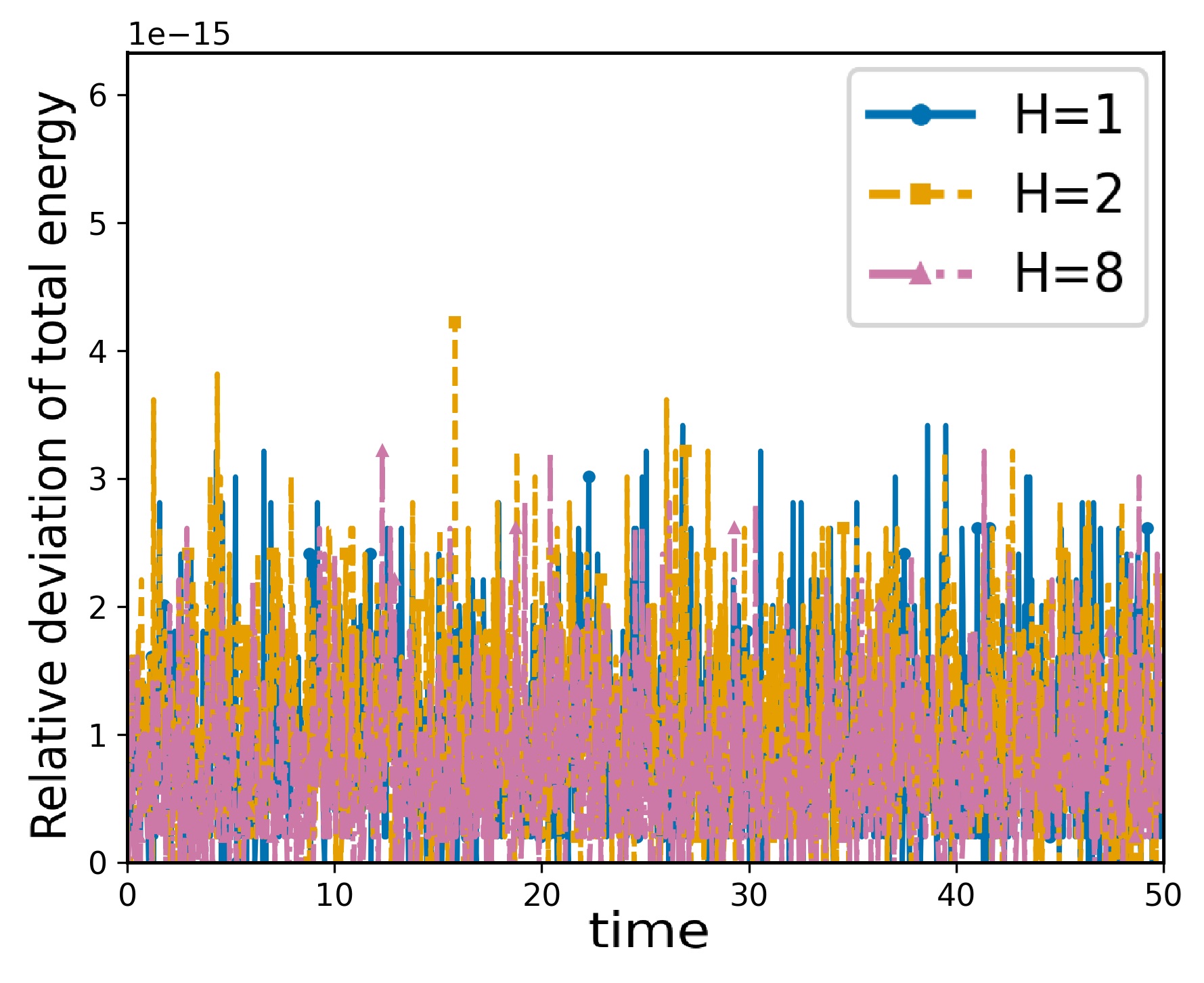}
        \caption{}
        \label{fig:SLD_total_energy}
    \end{subfigure}

    \caption{
Diagnostics for strong Landau damping with \(N_x=N_v=512\), CFL \(=10\), \(x\in[0,2\pi/0.4]\), \(v\in[-2\pi,2\pi]\), and \(H=1,2,8\). \textbf{(a)} shows the electric-field energy \(U(t)=\frac12\Delta x\sum_i E_i(t)^2\). \textbf{(b)} shows the adaptive rank retained after ACA-SVD recompression. \textbf{(c)}--\textbf{(e)} show the global conservation diagnostics computed from the final corrected low-rank state: relative mass deviation \(|M(t)-M(0)|/|M(0)|\), absolute deviation of total momentum \(|P(t)-P(0)|\), and relative total-energy deviation \(|\mathcal E_{tot}(t)-\mathcal E_{tot}(0)|/|\mathcal E_{tot}(0)|\), where \(\mathcal E_{tot}(t)=K(t)+U(t)\). Unlike the two-stream and bump-on-tail diagnostics, the $H = 8$ Landau damping run maintains a mildly higher adaptive rank, indicating that in this relatively low-rank benchmark the retained rank reflects the combined effect of $H$, the Fermi-Dirac correction, and ACA-SVD recompression.
}\label{fig:SLD_diagnostics}
\end{figure}

The conservation diagnostics again show that the proposed method preserves the macroscopic invariants to machine precision. This confirms that the Fermi-Dirac conservative correction is robust not only to instability-driven phase-space deformation, but also for damping-dominated kinetic dynamics.

\subsection{Bump-on-tail instability}

We simulate the bump-on-tail instability with the initial condition
\begin{equation}
    f(x,v,t=0)= \left( 1+ \alpha \; \cos(kx)\right) \left(n_p \exp \left(-\frac{v^2}{2}\right) + n_b \exp \left(-\frac{(v-u)^2}{2v_t}\right) \right)
\end{equation}
where $\alpha = 0.04,\;k=0.3,\;n_p=\frac{9}{10\sqrt{2\pi}},\;  n_b = \frac{2}{10\sqrt{2\pi}}, u=4.5, \; v_t= 0.5$ following \cite{guo2024local}

In Fig.~\ref{fig:BOT_phase}, the lower central band corresponds to the background population, while the separated upper bands show the beam population. For $H = 1$, the beam and background structures are more sharply deformed; for $H = 2$, the deformation remains visible but is smoother; for $H = 8$, the beam bands are smoother and less filamentary.

\begin{figure}[htbp]
    \centering

    \begin{subfigure}{0.49\textwidth}
        \centering
        \includegraphics[width=\linewidth]{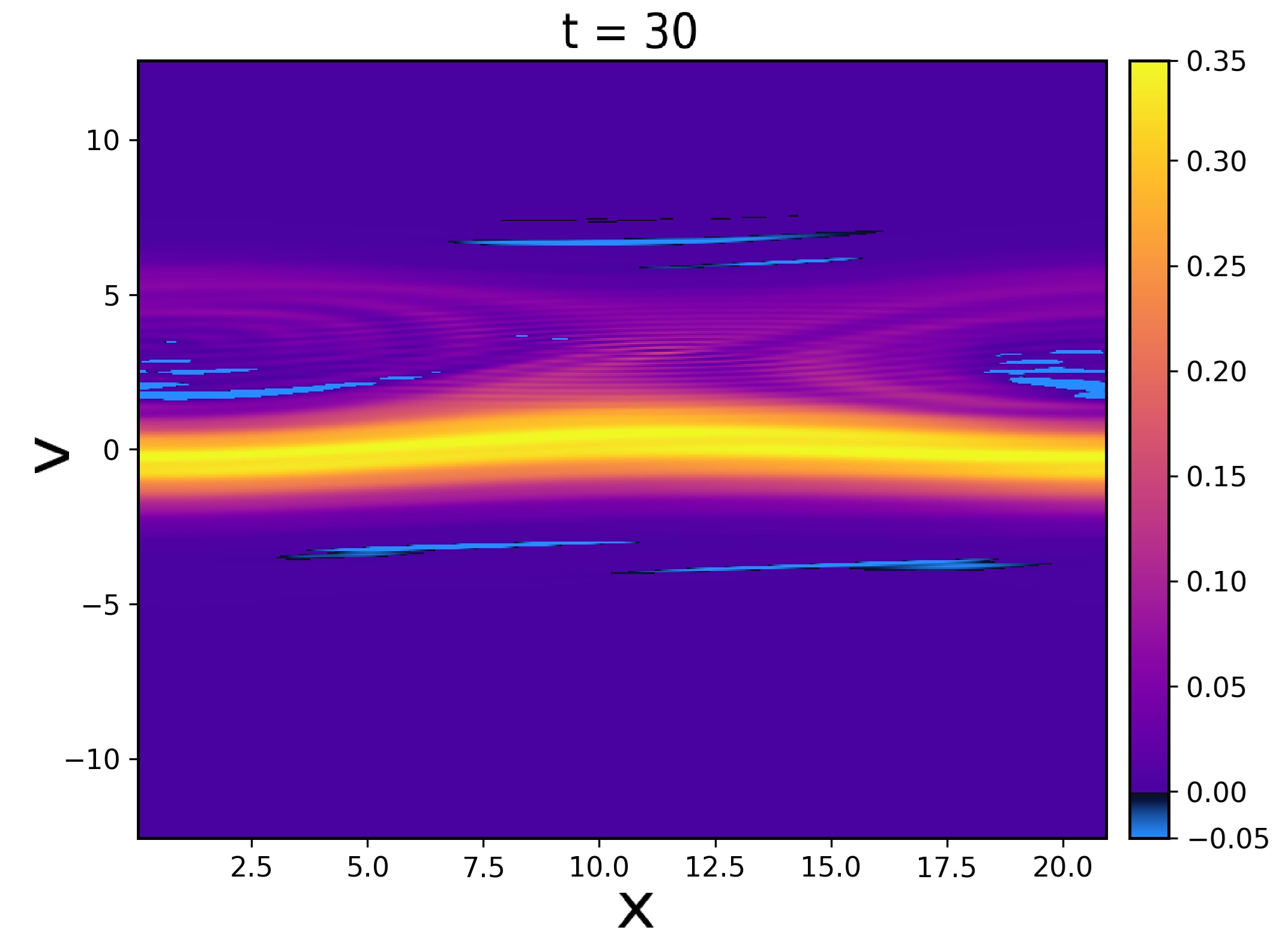}
        \caption{H=1, Low rank}
        \label{fig:BOT_phase_H1}
    \end{subfigure}
    \hfill
    \begin{subfigure}{0.49\textwidth}
        \centering
        \includegraphics[width=\linewidth]{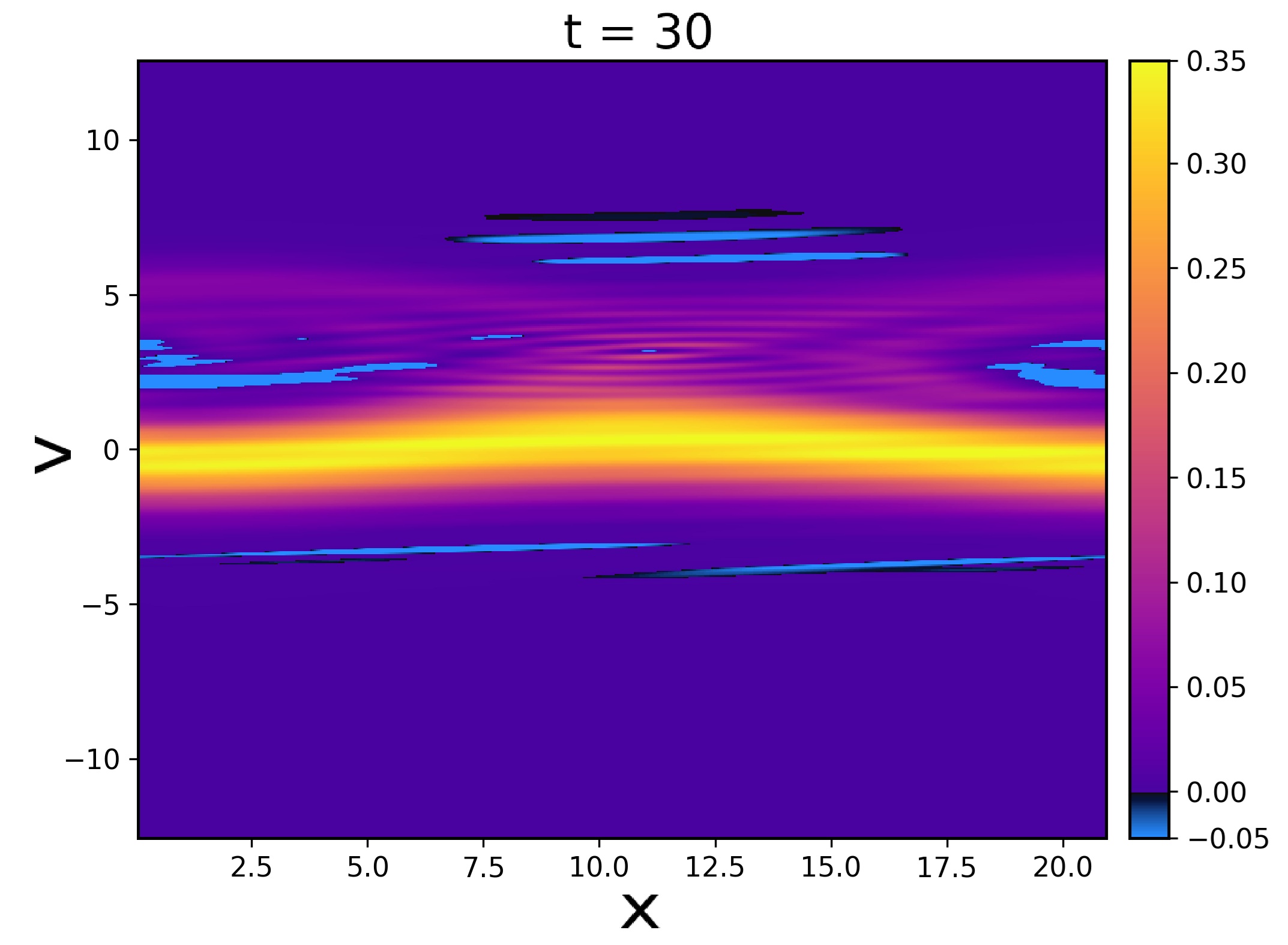}
        \caption{H=2, Low rank}
        \label{fig:BOT_phase_H2}
    \end{subfigure}

    \vspace{1em} 

    \begin{subfigure}{0.49\textwidth}
        \centering
        \includegraphics[width=\linewidth]{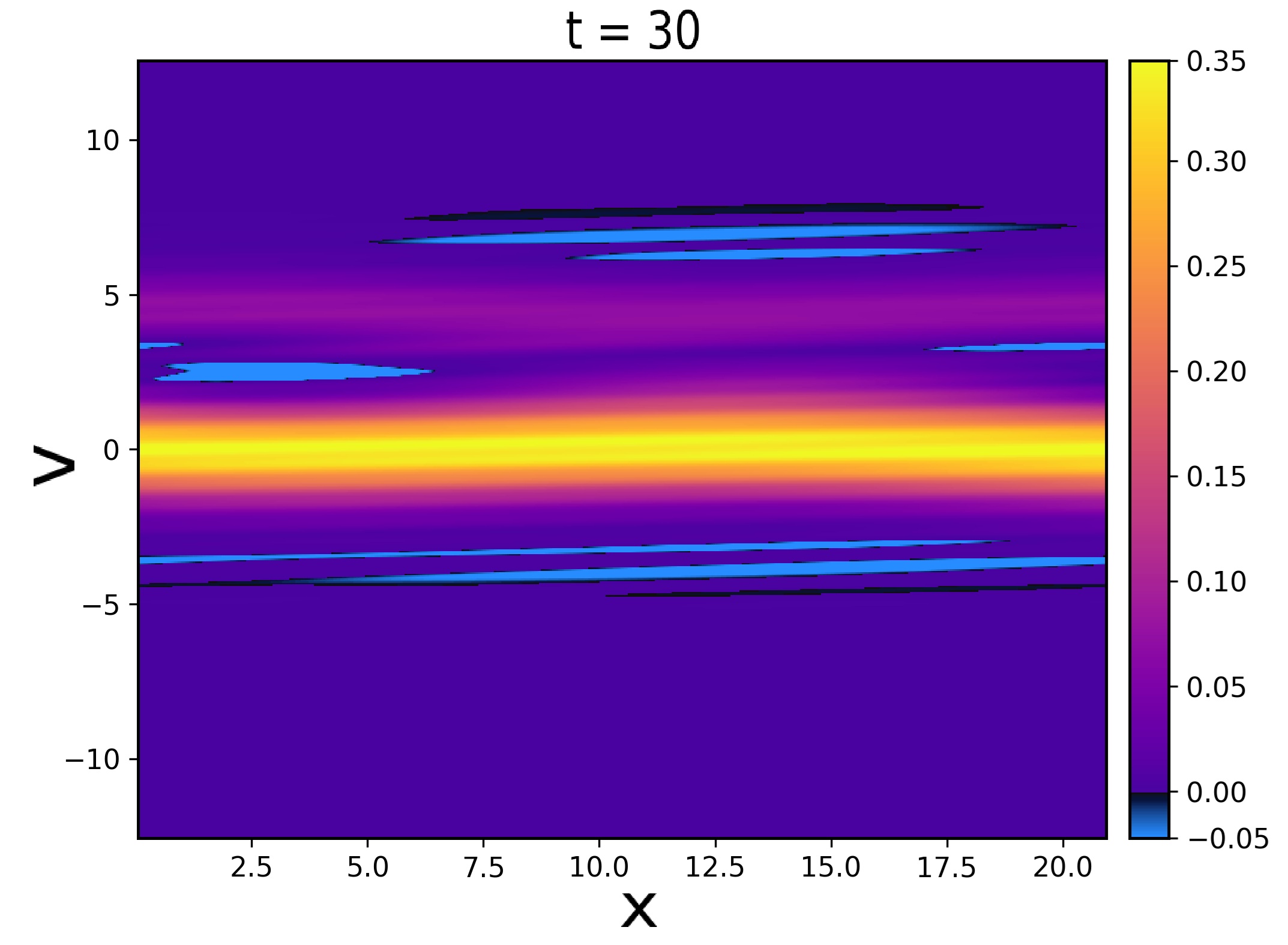}
        \caption{H=8, Low rank}
        \label{fig:BOT_phase_H8}
    \end{subfigure}

    \caption{
Phase-space distribution \(f(x,v,t)\) for the bump-on-tail instability at \(t=30\), with \(N_x=N_v=512\), CFL \(=10\), \(x\in[0,\frac{2\pi}{0.3}]\), and \(v\in[-4\pi, 4\pi]\). Panels (a)--(c) show the conservative adaptive rank solution for \(H=1\), \(H=2\), and \(H=8\), respectively. The beam and background structures are sharpest for $H = 1$ and become progressively smoother as $H$ increases, showing that the method captures the beam-driven instability while reflecting stronger quantum regularization at larger $H$.
}
    \label{fig:BOT_phase}
\end{figure}

The diagnostic results in Fig.~\ref{fig:BOT_total_energy} shows the electric energy growth and saturation for the beam-driven instability. Fig.~\ref{fig:BOT_rank_used} shows that the adaptive rank grows substantially during nonlinear beam evolution, reaching the largest values for $H=1$ but remaining far below the full phase-space dimension. Figs.~\ref{fig:BOT_total_mass}-\ref{fig:BOT_total_energy} show that mass, momentum, and total energy deviations remain near machine precision. Thus, the method handles a more nonequilibrium, multi-population test without losing the specified invariants.

\begin{figure}[htbp]
    \centering

    \begin{subfigure}{0.42\textwidth}
        \centering
        \includegraphics[width=\linewidth]{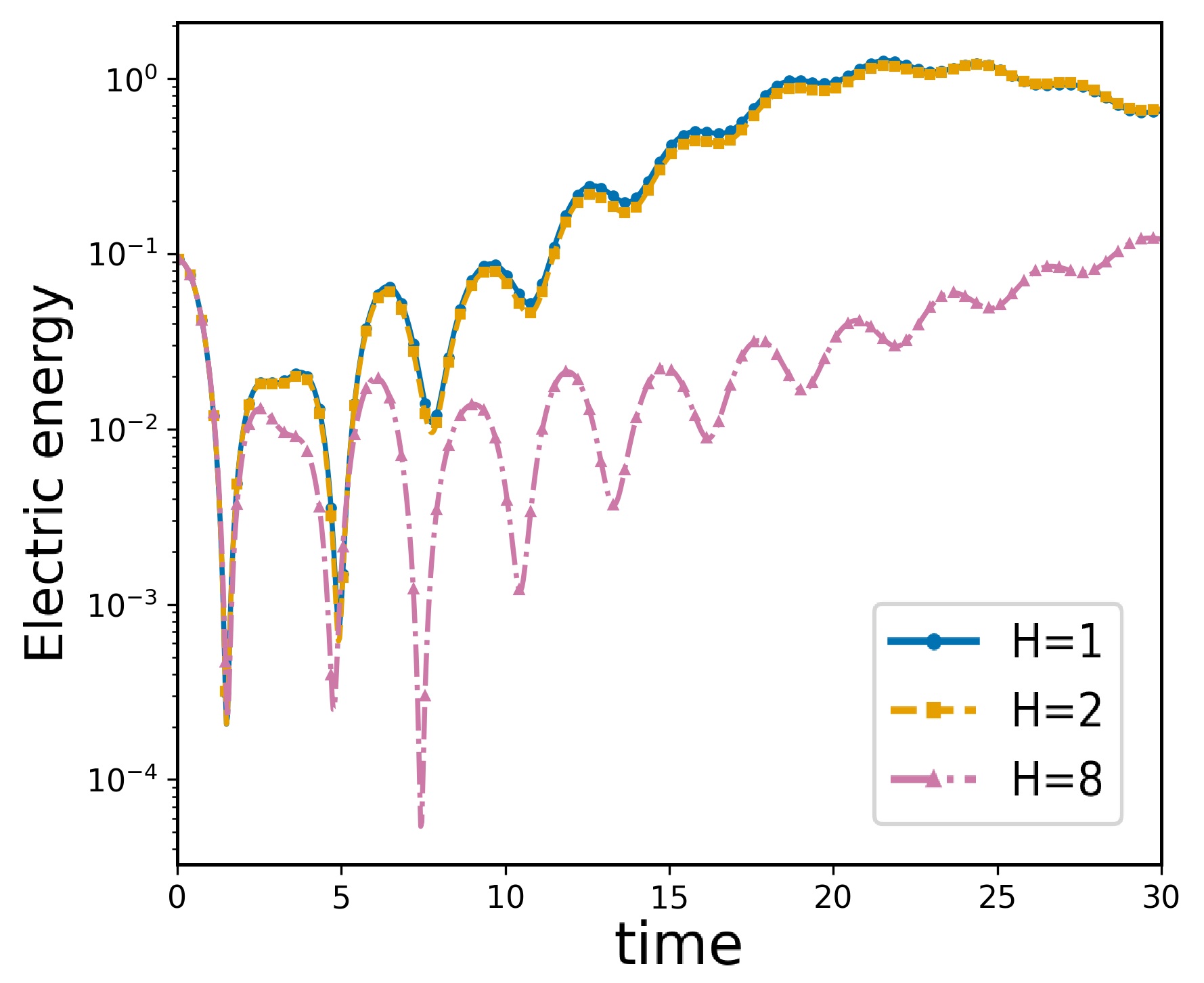}
        \caption{}
        \label{fig:BOT_electric_energy}
    \end{subfigure}
    \hfill
    \begin{subfigure}{0.42\textwidth}
        \centering
        \includegraphics[width=\linewidth]{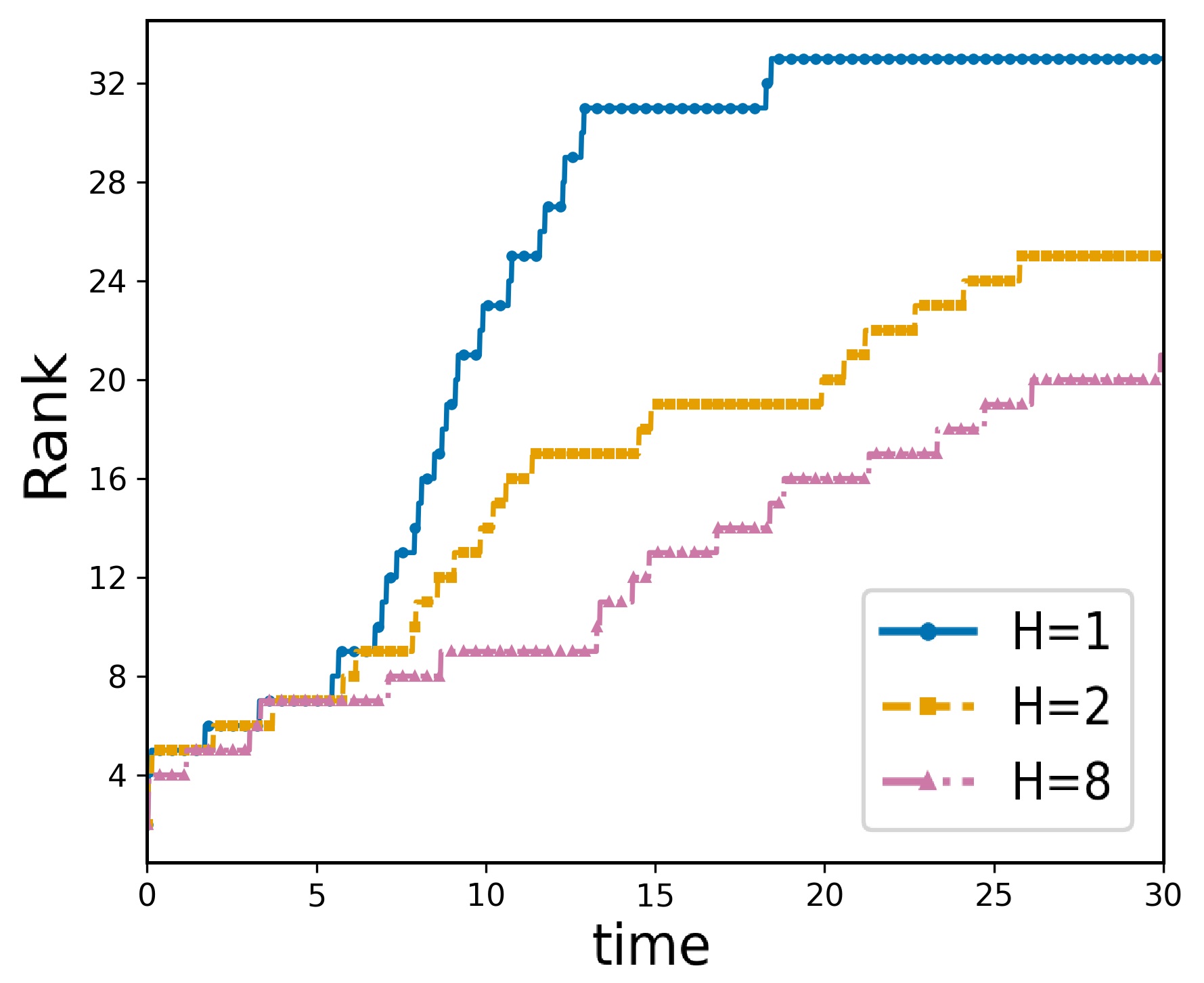}
        \caption{}
        \label{fig:BOT_rank_used}
    \end{subfigure}

    \vspace{1em} 

    \begin{subfigure}{0.42\textwidth}
        \centering
        \includegraphics[width=\linewidth]{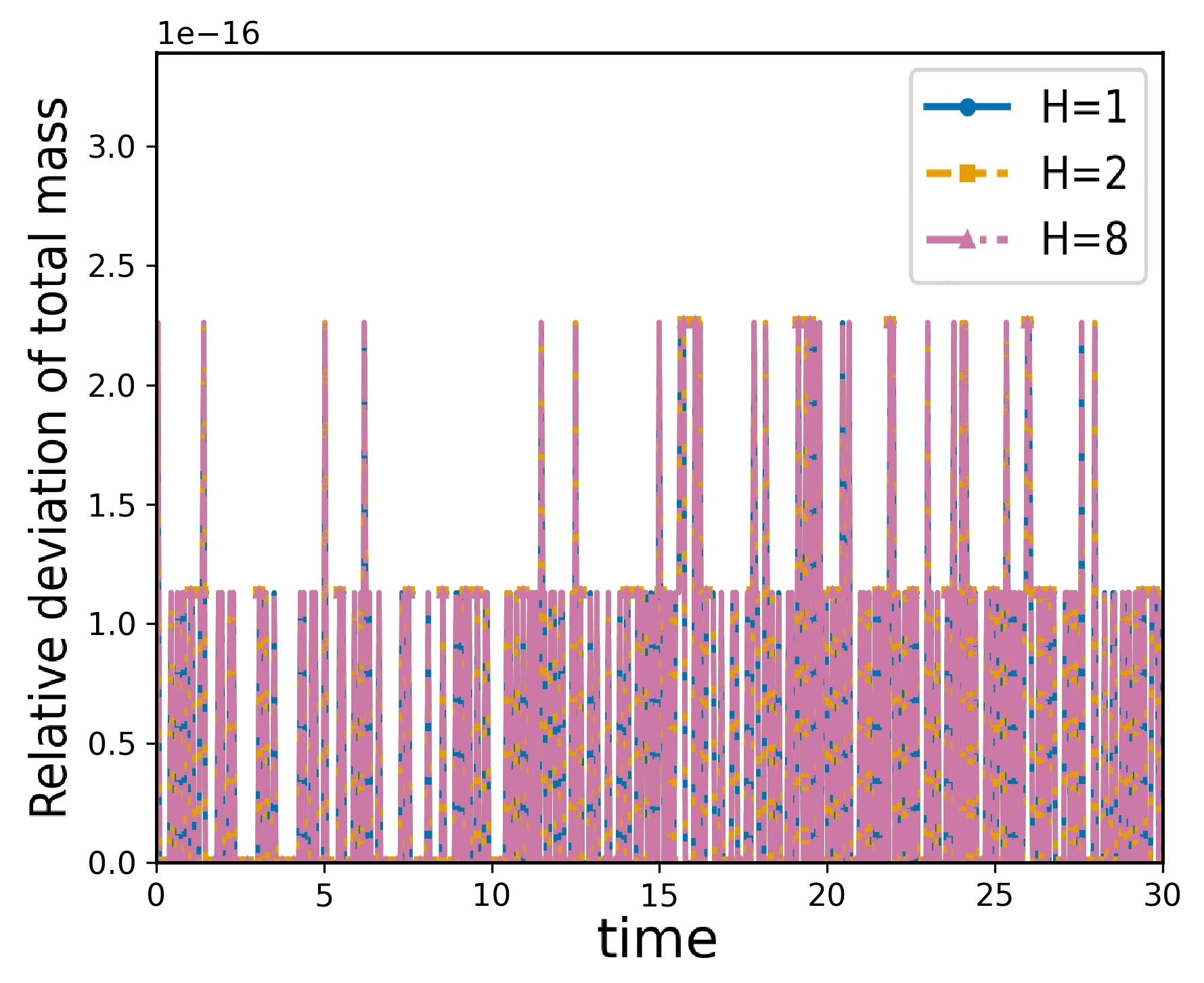}
        \caption{}
        \label{fig:BOT_total_mass}
    \end{subfigure}
    \hfill
    \begin{subfigure}{0.42\textwidth}
        \centering
        \includegraphics[width=\linewidth]{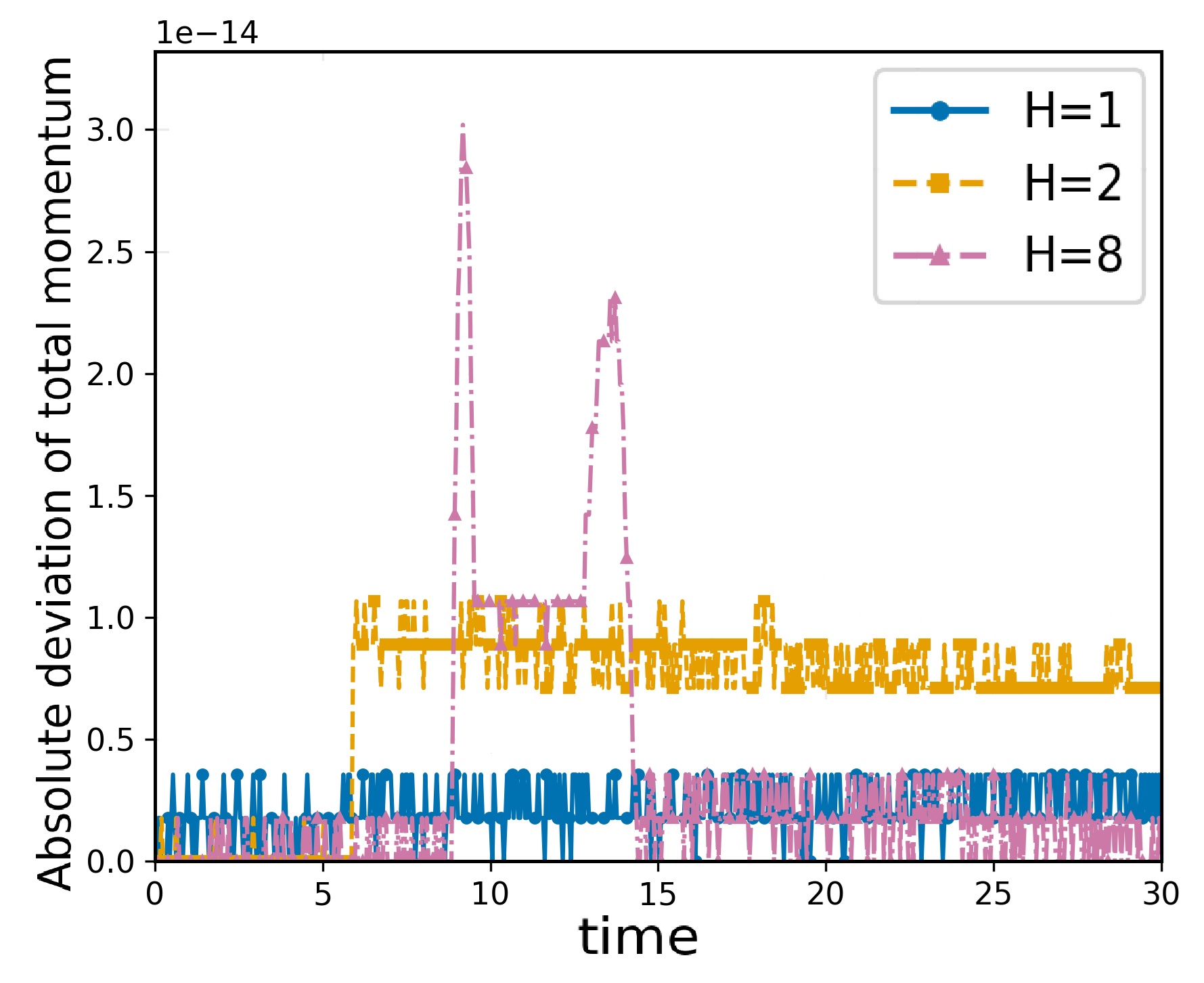}
        \caption{}
        \label{fig:BOT_total_momentum}
    \end{subfigure}

    \vspace{1em} 

    \begin{subfigure}{0.42\textwidth}
        \centering
        \includegraphics[width=\linewidth]{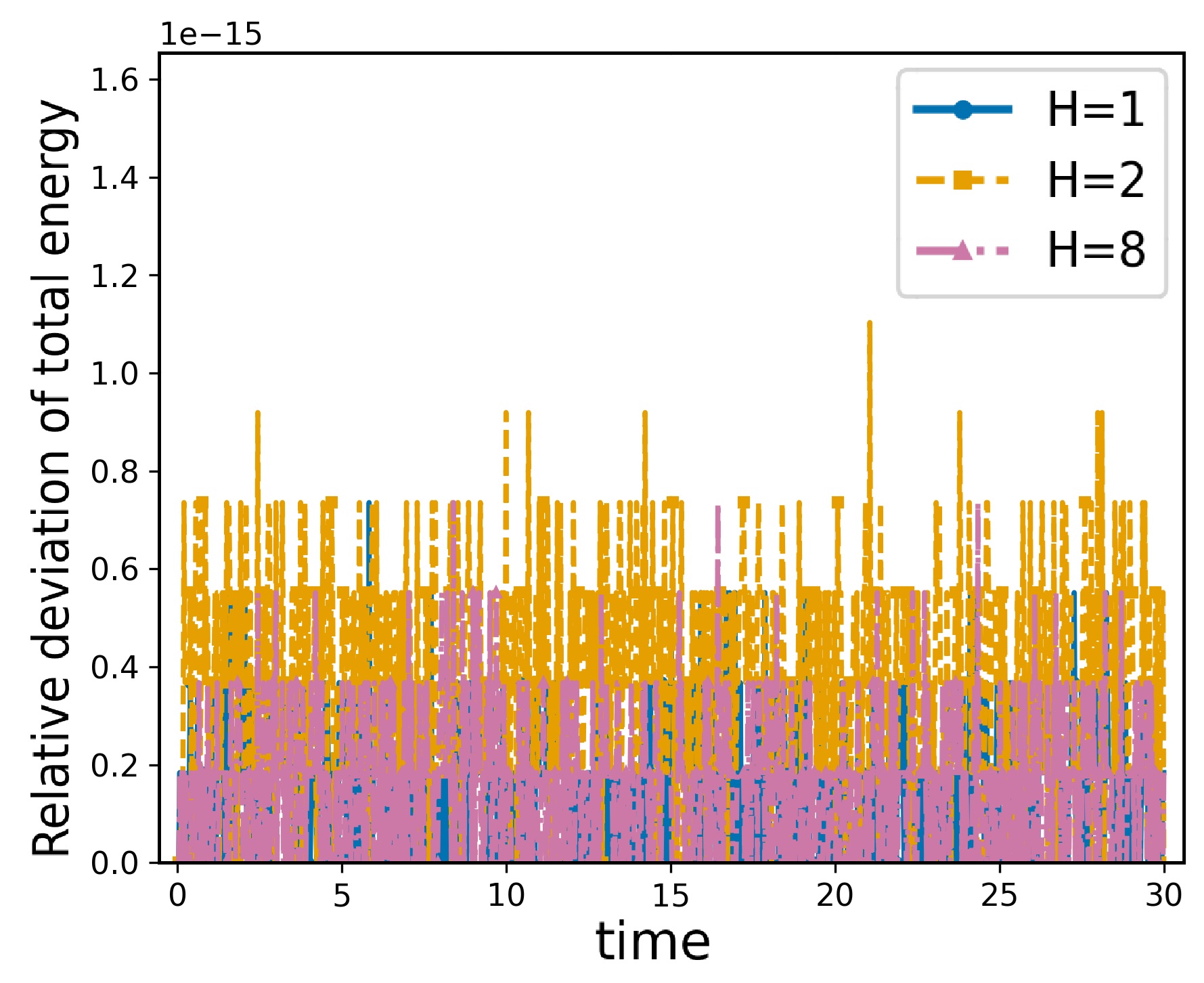}
        \caption{}
        \label{fig:BOT_total_energy}
    \end{subfigure}

    \caption{
Diagnostics for the bump-on-tail instability with \(N_x=N_v=512\), CFL \(=10\), final time \(T=30\), \(x\in[0, \frac{2\pi}{0.3}]\), \(v\in[-4\pi, 4\pi]\), and \(H=1,2,8\). \textbf{(a)} shows the electric-field energy \(U(t)=\frac12\Delta x\sum_i E_i(t)^2\). \textbf{(b)} shows the adaptive rank retained after ACA-SVD recompression. \textbf{(c)}--\textbf{(e)} show the global conservation diagnostics computed from the final corrected low-rank state: relative mass deviation \(|M(t)-M(0)|/|M(0)|\), absolute deviation of total momentum \(|P(t)-P(0)|\), and relative total-energy deviation \(|\mathcal E_{tot}(t)-\mathcal E_{tot}(0)|/|\mathcal E_{tot}(0)|\), where \(\mathcal E_{tot}(t)=K(t)+U(t)\). The diagnostic panels show electric energy growth and rank growth during the beam instability, while the mass, momentum, and total-energy errors remain at near machine precision, so the higher rank demand does not compromise the targeted conservation properties.
}\label{fig:BOT_diagnostics}
\end{figure}

\section{Conclusion}\label{sec:conclusion}

We have developed a conservative adaptive rank method for the 1D1V periodic Wigner-Poisson system. The method combines a sampling-based low-rank Wigner-Poisson update with a conservative macroscopic correction. The updated density and momentum are obtained from a conservative moment solve, the update is transferred to the kinetic solution through a Fermi-Dirac-type reconstruction, and the total energy is enforced by a global quadratic moment correction.

Numerical experiments for the two-stream instability, strong Landau damping, and bump-on-tail instability show that the method captures the expected Wigner-Poisson dynamics for several values of the quantum parameter $H$. In the reported 1D1V periodic tests, the correction preserves the global discrete invariants near machine precision while maintaining bounded adaptive ranks. The spectral and timing diagnostics further indicate that the conservative correction is compatible with ACA-SVD compression and does not eliminate the computational advantage of the adaptive representation.

This work provides a conservative low-rank framework for Wigner-Poisson dynamics in the 1D1V periodic setting. Future work will address higher-dimensional scalable implementations, sharper discrete moment matching, and extensions to Wigner-Poisson-BGK models.

\section*{Acknowledgements}
The authors acknowledge support from AFOSR grants FA9550-24-1-0254 and DOE grant DE-SC0023164. AJC is also supported by DOE/NNSA grant DE-NA0004265 and ONR grant N00014-24-1-2242, and the MSU Institute for Cyber-Enabled Research for their computer resources. The authors acknowledge assistance from ChatGPT with grammar editing.

\appendix
\section{Proof of Theorem \ref{thm:fluid model}}\label{apdix:thm 2.1}
First, by the definitions of the macroscopic moments, we have
\begin{equation}
    \int v f\,dv = \rho u,
\end{equation}
and
\begin{equation}
    \int v^2 f\,dv
    =
    \int \bigl((v-u)+u\bigr)^2 f\,dv
    =
    P+\rho u^2.
\end{equation}
Moreover, since
\begin{equation}
    Q=\int (v-u)^3 f\,dv,
\end{equation}
we obtain
\begin{align}
    \int v^3 f\,dv
    &=
    \int \bigl((v-u)+u\bigr)^3 f\,dv  \notag\\
    &=
    \int (v-u)^3 f\,dv
    +3u\int (v-u)^2 f\,dv
    +3u^2\int (v-u)f\,dv
    +u^3\int f\,dv \notag\\
    &=
    Q+3uP+\rho u^3,
\end{align}
where we used
\begin{equation}
    \int (v-u)f\,dv=0.
\end{equation}

We now derive the moment equations. Integrating the Wigner equation against
$1$ in velocity and using \eqref{eqn:WignerIdentity1}, we find
\begin{align}
    0 = 
    \int W\,dv =
    \int \left(f_t+v f_x\right)\,dv =
    \rho_t+(\rho u)_x.
\end{align}
Hence,
\begin{equation}
    \rho_t+(\rho u)_x=0.
\end{equation}

Next, multiplying the Wigner equation by $v$, integrating in velocity, and
using \eqref{eqn:WignerIdentity2}, we obtain
\begin{align}
    -\rho \Phi_x =
    \int vW\,dv  =
    \int \left(v f_t+v^2 f_x\right)\,dv =
    (\rho u)_t+\left(P+\rho u^2\right)_x.
\end{align}
Therefore,
\begin{equation}
    (\rho u)_t+\left(P+\rho u^2\right)_x
    =
    -\rho \Phi_x.
\end{equation}

Similarly, multiplying the Wigner equation by $v^2$, integrating in velocity,
and using \eqref{eqn:WignerIdentity3}, gives
\begin{align}
    -2\rho u\Phi_x =
    \int v^2 W\,dv =
    \int \left(v^2 f_t+v^3 f_x\right)\,dv =
    \left(P+\rho u^2\right)_t
    +
    \left(\rho u^3+3uP+Q\right)_x.
\end{align}
Thus,
\begin{equation}
    \left(P+\rho u^2\right)_t
    +
    \left(\rho u^3+3uP+Q\right)_x
    =
    -2\rho u\Phi_x.
\end{equation}

We next derive the corresponding conservation laws. Integrating the continuity equation with respect to $x$ yields
\begin{align}
    \frac{d}{dt}\int \rho\,dx =
    -\int (\rho u)_x\,dx =
    0.
\end{align}
Therefore, the total mass is conserved:
\begin{equation}
    \frac{d}{dt}\int \rho\,dx=0.
\end{equation}

For the momentum, integrating the momentum equation in $x$ gives
\begin{equation}
    \frac{d}{dt}\int \rho u\,dx
    =
    -\int \rho\Phi_x\,dx.
\end{equation}
Using the Poisson equation
\begin{equation}
    \Phi_{xx}=1-\rho,
\end{equation}
we have $\rho=1-\Phi_{xx}$. Hence,
\begin{align}
    -\int \rho\Phi_x\,dx=
    -\int (1-\Phi_{xx})\Phi_x\,dx =
    -\int \Phi_x\,dx
    +
    \int \Phi_{xx}\Phi_x\,dx =
    -\int \Phi_x\,dx
    +
    \frac12\int \bigl((\Phi_x)^2\bigr)_x\,dx =
    0.
\end{align}
Consequently,
\begin{equation}
    \frac{d}{dt}\int \rho u\,dx=0,
\end{equation}
and the total momentum is conserved.

It remains to prove conservation of energy. Since
\begin{equation}
    P+\rho u^2=\int v^2 f\,dv,
\end{equation}
the second-moment equation implies
\begin{align}
    \frac{d}{dt}
    \left(
        \frac12\iint v^2 f\,dv\,dx
    \right)
    &=
    \frac12\int
    \left(P+\rho u^2\right)_t\,dx =
    -\int \rho u\Phi_x\,dx.
\end{align}

On the other hand, differentiating the Poisson equation in time gives
\begin{equation}
    \rho_t=-\Phi_{xxt}.
\end{equation}
Therefore,
\begin{align}
    \frac{d}{dt}
    \left(
        \frac12\int (\Phi_x)^2\,dx
    \right)
    &=
    \int \Phi_x\Phi_{xt}\,dx =
    -\int \Phi\Phi_{xxt}\,dx =
    \int \Phi\rho_t\,dx.
\end{align}
Using the continuity equation, $\rho_t=-(\rho u)_x$, we obtain
\begin{align}
    \int \Phi\rho_t\,dx
    &=
    -\int \Phi(\rho u)_x\,dx =
    \int \rho u\Phi_x\,dx.
\end{align}
Thus,
\begin{equation}
    \frac{d}{dt}
    \left(
        \frac12\int (\Phi_x)^2\,dx
    \right)
    =
    \int \rho u\Phi_x\,dx.
\end{equation}

Adding the kinetic and field energy identities gives
\begin{equation}
    \frac{d}{dt}
    \left[
        \frac12\iint v^2 f\,dv\,dx
        +
        \frac12\int (\Phi_x)^2\,dx
    \right]
    =
    0.
\end{equation}
Hence, the total energy
\begin{equation}
    \mathcal{E}(t)
    =
    \frac12\iint v^2 f\,dv\,dx
    +
    \frac12\int (\Phi_x)^2\,dx
\end{equation}
is conserved. 

\section{Proof of Lemma \ref{lem:Fermi consistency}}\label{apdix:lem fermi}
\begin{proof}
Let \(w=v-u\), so that \(dv=dw\). Then
\begin{equation}
    f^{FD}=\frac{\alpha\rho}{C}\frac{1}{A\exp\left(\frac{w^2}{2\hat {\mathcal{T}}}\right)+\rho}.
\end{equation}
Using \(s=w/\sqrt{2\hat {\mathcal{T}}}\), so that \(dw=\sqrt{2\hat {\mathcal{T}}}\,ds\), we obtain
\begin{align}
    \int f^{FD}\,dv
    &=\frac{\alpha\rho}{C}\int \frac{dw}{A\exp\left(\frac{w^2}{2\hat {\mathcal{T}}}\right)+\rho}
    =\frac{\alpha\rho\sqrt{2\hat {\mathcal{T}}}}{C}\int \frac{ds}{A\exp(s^2)+\rho}
    =\frac{\alpha\rho\sqrt{2\hat {\mathcal{T}}}}{C}I_0=\rho,
\end{align}
by the definition of \(C\). Next,
\begin{align}
    \int v f^{FD}\,dv
    &=\int (w+u)f^{FD}\,dw
    =\int w f^{FD}\,dw+u\int f^{FD}\,dw
    =\rho u,
\end{align}
since \(f^{FD}\) is even in \(w\), and therefore \(w f^{FD}\) is odd. For the second moment,
\begin{align}
    \int \frac{1}{2}v^2f^{FD}\,dv
    &=\frac{1}{2}\int (w+u)^2f^{FD}\,dw
    =\frac{1}{2}\left(\int w^2f^{FD}\,dw+\rho u^2\right) \nonumber\\
    &=\frac{1}{2}\left(\frac{\alpha\rho}{C}\int \frac{w^2\,dw}{A\exp\left(\frac{w^2}{2\hat {\mathcal{T}}}\right)+\rho}+\rho u^2\right) \nonumber\\
    &=\frac{1}{2}\left(\frac{\alpha\rho}{C}2\hat {\mathcal{T}}\sqrt{2\hat {\mathcal{T}}}\int \frac{s^2\,ds}{A\exp(s^2)+\rho}+\rho u^2\right) \nonumber\\
    &=\frac{1}{2}\left(2\rho\hat {\mathcal{T}}\frac{I_1}{I_0}+\rho u^2\right)
    =\rho\left(\frac{1}{2}u^2+\hat {\mathcal{T}}\frac{I_1}{I_0}\right).
\end{align}
Thus \(e^{FD}=\hat {\mathcal{T}} I_1/I_0\). Finally, since \(Q(f)=\frac{1}{\tau}(f^{FD}-f)\), if \(f\) and \(f^{FD}\) have the same density and momentum, then
\begin{equation}
    \int Q(f)\,dv=\frac{\rho-\rho}{\tau}=0,
    \qquad
    \int vQ(f)\,dv=\frac{\rho u-\rho u}{\tau}=0.
\end{equation}
If \(\hat {\mathcal{T}}\) is chosen so that \(e^{FD}=e\), then
\begin{equation}
    \int \frac{1}{2}v^2Q(f)\,dv=\frac{1}{\tau}\left[\rho\left(\frac{1}{2}u^2+e^{FD}\right)-\rho\left(\frac{1}{2}u^2+e\right)\right]=0.
\end{equation}
This proves the result.
\end{proof}

\bibliographystyle{siam} 
\bibliography{ref2}

@article{einkemmer2018low,
  title={A low-rank projector-splitting integrator for the Vlasov--Poisson equation},
  author={Einkemmer, Lukas and Lubich, Christian},
  journal={SIAM Journal on Scientific Computing},
  volume={40},
  number={5},
  pages={B1330--B1360},
  year={2018},
  publisher={SIAM}
}

@article{koch2007dynamical,
  title={Dynamical low-rank approximation},
  author={Koch, Othmar and Lubich, Christian},
  journal={SIAM Journal on Matrix Analysis and Applications},
  volume={29},
  number={2},
  pages={434--454},
  year={2007},
  publisher={SIAM}
}

@article{sun2024hybrid,
  title={A Hybrid SBP-SAT/Fourier Pseudo-spectral Method for the Transient Wigner Equation Involving Inflow Boundary Conditions},
  author={Sun, Zhangpeng and Yao, Wenqi and Yu, Qiuping},
  journal={Journal of Scientific Computing},
  volume={100},
  number={2},
  pages={38},
  year={2024},
  publisher={Springer}
}

@article{GuoVlasovFlowMap2022,
    author = {Wei Guo and Jing-Mei Qiu},
    title = {A low rank tensor representation of linear transport and nonlinear {V}lasov solutions and their associated flow maps},
    journal = {Journal of Computational Physics},
    volume = {458},
    pages = {111089},
    year = {2022}
}

@article{guo2024local,
  title={A Local Macroscopic Conservative (LoMaC) low rank tensor method for the Vlasov dynamics},
  author={Guo, Wei and Qiu, Jing-Mei},
  journal={Journal of Scientific Computing},
  volume={101},
  number={3},
  pages={61},
  year={2024},
  publisher={Springer}
}

@article{guo2024conservative,
  title={A conservative low rank tensor method for the Vlasov dynamics},
  author={Guo, Wei and Qiu, Jing-Mei},
  journal={SIAM Journal on Scientific Computing},
  volume={46},
  number={1},
  pages={A232--A263},
  year={2024},
  publisher={SIAM}
}

@article{einkemmer2021mass,
  title={A mass, momentum, and energy conservative dynamical low-rank scheme for the Vlasov equation},
  author={Einkemmer, Lukas and Joseph, Ilon},
  journal={Journal of Computational Physics},
  volume={443},
  pages={110495},
  year={2021},
  publisher={Elsevier}
}

@article{wigner1932quantum,
  title={On the quantum correction for thermodynamic equilibrium},
  author={Wigner, Eugene},
  journal={Physical review},
  volume={40},
  number={5},
  pages={749},
  year={1932},
  publisher={APS}
}

@article{pathria2021,
  title={Statistical Mechanics Fourth Edition},
  author={Beale, Paul D and Pathria, RK},
  year={2021}
}

@book{landau1980,
  title={Statistical Physics: Volume 5},
  author={Landau, Lev Davidovich and Lifshitz, Evgenii Mikhailovich},
  volume={5},
  year={2013},
  publisher={Elsevier}
}

@article{ringhofer1990spectral,
  title={A spectral method for the numerical simulation of quantum tunneling phenomena},
  author={Ringhofer, Christian},
  journal={SIAM journal on numerical analysis},
  volume={27},
  number={1},
  pages={32--50},
  year={1990},
  publisher={SIAM}
}

@article{ringhofer1992spectral,
  title={A spectral collocation technique for the solution of the Wigner--Poisson problem},
  author={Ringhofer, Christian},
  journal={SIAM journal on numerical analysis},
  volume={29},
  number={3},
  pages={679--700},
  year={1992},
  publisher={SIAM}
}

@article{suh1991numerical,
  title={Numerical simulation of the quantum Liouville-Poisson system},
  author={Suh, Nam-Duk and Feix, Marl R and Bertrand, Pierre},
  journal={Journal of Computational Physics},
  volume={94},
  number={2},
  pages={403--418},
  year={1991},
  publisher={Elsevier}
}

@article{arnold1995operator,
  title={Operator splitting methods applied to spectral discretizations of quantum transport equations},
  author={Arnold, Anton and Ringhofer, Christian},
  journal={SIAM journal on numerical analysis},
  volume={32},
  number={6},
  pages={1876--1894},
  year={1995},
  publisher={SIAM}
}

@article{furtmaier2016semi,
  title={Semi-spectral method for the Wigner equation},
  author={Furtmaier, Oliver and Succi, Sauro and Mendoza, Miller},
  journal={Journal of Computational Physics},
  volume={305},
  pages={1015--1036},
  year={2016},
  publisher={Elsevier}
}

@article{chen2022higher,
  title={A higher-order accurate operator splitting spectral method for the Wigner--Poisson system},
  author={Chen, Zhenzhu and Jiang, Haiyan and Shao, Sihong},
  journal={Journal of Computational Electronics},
  volume={21},
  number={4},
  pages={756--770},
  year={2022},
  publisher={Springer}
}

@article{shao2011adaptive,
  title={Adaptive conservative cell average spectral element methods for transient Wigner equation in quantum transport},
  author={Shao, Sihong and Lu, Tiao and Cai, Wei},
  journal={Communications in Computational Physics},
  volume={9},
  number={3},
  pages={711--739},
  year={2011},
  publisher={Cambridge University Press}
}

@article{xiong2016advective,
  title={An advective-spectral-mixed method for time-dependent many-body Wigner simulations},
  author={Xiong, Yunfeng and Chen, Zhenzhu and Shao, Sihong},
  journal={SIAM Journal on Scientific Computing},
  volume={38},
  number={4},
  pages={B491--B520},
  year={2016},
  publisher={SIAM}
}

@article{dorda2015weno,
  title={A WENO-solver combined with adaptive momentum discretization for the Wigner transport equation and its application to resonant tunneling diodes},
  author={Dorda, Antonius and Sch{\"u}rrer, Ferdinand},
  journal={Journal of computational physics},
  volume={284},
  pages={95--116},
  year={2015},
  publisher={Elsevier}
}

@book{haas2011quantum,
  title={Quantum plasmas: An hydrodynamic approach},
  author={Haas, Fernando},
  volume={65},
  year={2011},
  publisher={Springer Science \& Business Media}
}

@article{lubich2014projector,
  title={A projector-splitting integrator for dynamical low-rank approximation},
  author={Lubich, Christian and Oseledets, Ivan V},
  journal={BIT Numerical Mathematics},
  volume={54},
  number={1},
  pages={171--188},
  year={2014},
  publisher={Springer}
}

@article{christlieb2025sampling,
  title={A Sampling-Based Adaptive Rank Approach to the Wigner-Poisson System},
  author={Christlieb, Andrew and Gong, Sining and Qiu, Jing-Mei and Zheng, Nanyi},
  journal={arXiv preprint arXiv:2506.21314},
  year={2025}
}

@article{bonilla2005wigner,
  title={Wigner--Poisson and nonlocal drift-diffusion model equations for semiconductor superlattices},
  author={Bonilla, LL and Escobedo, R},
  journal={Mathematical Models and Methods in Applied Sciences},
  volume={15},
  number={08},
  pages={1253--1272},
  year={2005},
  publisher={World Scientific}
}

@article{crouseilles2014asymptotic,
  title={Asymptotic preserving schemes for the Wigner--Poisson--BGK equations in the diffusion limit},
  author={Crouseilles, Nicolas and Manfredi, Giovanni},
  journal={Computer Physics Communications},
  volume={185},
  number={2},
  pages={448--458},
  year={2014},
  publisher={Elsevier}
}

@book{markowich2012semiconductor,
  title={Semiconductor equations},
  author={Markowich, Peter A and Ringhofer, Christian A and Schmeiser, Christian},
  year={2012},
  publisher={Springer Science \& Business Media}
}

@article{weinbub2018recent,
  title={Recent advances in Wigner function approaches},
  author={Weinbub, Josef and Ferry, DK},
  journal={Applied Physics Reviews},
  volume={5},
  number={4},
  year={2018},
  publisher={AIP Publishing}
}

@article{frensley1990boundary,
  title={Boundary conditions for open quantum systems driven far from equilibrium},
  author={Frensley, William R},
  journal={Reviews of Modern Physics},
  volume={62},
  number={3},
  pages={745},
  year={1990},
  publisher={APS}
}

@article{hu2022kinetic,
  title={Kinetic investigations of nonlinear electrostatic excitations in quantum plasmas},
  author={Hu, Tian-Xing and Liang, Jiong-Hang and Sheng, Zheng-Mao and Wu, Dong},
  journal={Physical Review E},
  volume={105},
  number={6},
  pages={065203},
  year={2022},
  publisher={APS}
}

@article{jiang2023numerical,
  title={Numerical Study of Transient Wigner--Poisson Model for RTDs: Numerical Method and Its Applications},
  author={Jiang, Haiyan and Lu, Tiao and Yao, Wenqi and Zhang, Weitong},
  journal={SIAM Journal on Scientific Computing},
  volume={45},
  number={4},
  pages={A1766--A1788},
  year={2023},
  publisher={SIAM}
}

@article{einkemmer2023robust,
  title={A robust and conservative dynamical low-rank algorithm},
  author={Einkemmer, Lukas and Ostermann, Alexander and Scalone, Carmela},
  journal={Journal of Computational Physics},
  volume={484},
  pages={112060},
  year={2023},
  publisher={Elsevier}
}

@article{coughlin2024robust,
  title={Robust and conservative dynamical low-rank methods for the Vlasov equation via a novel macro-micro decomposition},
  author={Coughlin, Jack and Hu, Jingwei and Shumlak, Uri},
  journal={Journal of Computational Physics},
  volume={509},
  pages={113055},
  year={2024},
  publisher={Elsevier}
}

@article{cai2012quantum,
  title={Quantum hydrodynamic model by moment closure of Wigner equation},
  author={Cai, Zhenning and Fan, Yuwei and Li, Ruo and Lu, Tiao and Wang, Yanli},
  journal={Journal of mathematical physics},
  volume={53},
  number={10},
  year={2012},
  publisher={AIP Publishing}
}

@article{li2014numerical,
  title={Numerical validation for high order hyperbolic moment system of Wigner equation},
  author={Li, Ruo and Lu, Tiao and Wang, Yanli and Yao, Wenqi},
  journal={Communications in Computational Physics},
  volume={15},
  number={3},
  pages={569--595},
  year={2014},
  publisher={Cambridge University Press}
}

@article{barletti2012derivation,
  title={Derivation of isothermal quantum fluid equations with Fermi-Dirac and Bose-Einstein statistics},
  author={Barletti, Luigi and Cintolesi, Carlo},
  journal={Journal of Statistical Physics},
  volume={148},
  number={2},
  pages={353--386},
  year={2012},
  publisher={Springer}
}

@misc{christlieb2026structurepreservingadaptiverankapproachhighdimensional,
      title={A Structure-preserving Adaptive-Rank Approach to the High-Dimensional Wigner-Poisson System}, 
      author={Andrew J. Christlieb and Sining Gong and Jing-Mei Qiu and Nanyi Zheng},
      year={2026},
      eprint={2606.15067},
      archivePrefix={arXiv},
      primaryClass={math.NA},
      url={https://arxiv.org/abs/2606.15067}, 
}
\end{document}